\documentclass[12pt,a4paper]{article}

\usepackage[utf8]{inputenc}
\usepackage{amssymb, amsthm, amsmath}
\usepackage[english]{babel}
\usepackage[T1]{fontenc}
\usepackage{graphicx}
\usepackage{mathtools}
\usepackage[font=small,labelfont={normal, bf}]{caption}
\usepackage{subcaption}
\usepackage{comment}
\usepackage{enumitem}
\usepackage[cm]{fullpage}
\usepackage{xcolor}

\usepackage{autonum}

\makeatletter
\newcommand{\@giventhatstar}[2]{#1\;\middle|\;#2)}
\newcommand{\@giventhatnostar}[3][]{#1#2\;#1|\;#3#1}
\newcommand{\giventhat}{\@ifstar\@giventhatstar\@giventhatnostar}
\makeatother

\newtheorem{theorem}{Theorem}
\newtheorem{proposition}{Proposition}
\theoremstyle{definition}

\newtheorem{corollary}{Corollary}
\newtheorem{example}{Example}
\newtheorem{remark}{Remark}
\newtheorem{lemma}{Lemma}

\title{Structure-Preserving Numerical Schemes for a Time-Fractional Porous Medium Equation with Nonlocal Pressure}
\author{
\L{}ukasz P\l{}ociniczak$^*$, \and Hubert Woszczek\thanks{Faculty of Pure and Applied Mathematics, Wroclaw University of Science and Technology, Wyb. Wyspia\'nskiego 27, 50-370 Wroc{\l}aw, Poland}$^{*,}$\thanks{\underline{Corresponding author:} \texttt{hubert.woszczek@pwr.edu.pl}}
}
\date{}

\begin{document}
\maketitle

\begin{abstract}

In this work, we study numerical approximations of a time-fractional porous medium equation with nonlocal pressure of the form
\begin{equation}
	\partial_t^\alpha u-\nabla\cdot\left(u^{m-1}\nabla(-\Delta)^{-\sigma}u\right)=0,
\end{equation}
where $0<\alpha<1$, $0<\sigma<1$ and $m>1$. The Caputo derivative is discretized by the L1 formula, while the nonlinear nonlocal flux is approximated by a conservative upwind finite-difference scheme. 

We establish several structural properties of the implicit scheme in arbitrary space dimension, including a discrete maximum principle, non-negativity, mass conservation, $\ell^\infty$ and $\ell^2$ stability, existence of non-negative discrete solutions, and a discrete energy-dissipation estimate. Our assumptions include natural symmetry, a sign condition on the sum of the directional coefficients, and decay on the discrete weights. We also derive a consistency estimate for the complete nonlinear spatial operator.

For rigorous convergence analysis, we restrict ourselves to one space dimension and introduce the cumulative variable associated with the density. This transforms the discrete problem into a monotone approximation of an equation involving the fractional Laplacian. Interestingly, for semigroup-based pressure-gradient weights, the resulting discrete operator is identified with the fractional discrete Laplacian. Next, using comparison and contraction properties of the integrated scheme, for $m\geq2$ and under the stated spatial regularity and weighted temporal consistency assumptions, we prove uniform convergence of the cumulative numerical solution and convergence of the corresponding density against compactly supported smooth test functions. The temporal estimates allow a weak singularity at the initial time. We also analyze an explicit variant of the scheme and derive a sufficient CFL condition covering the full range $0<\sigma<1$. Numerical experiments illustrate the stability and convergence properties of the proposed methods.
\end{abstract}



\section{Introduction}
Nonlocal porous medium equations combine a density-dependent diffusivity with interactions acting over long spatial scales. In the fractional-pressure model, these interactions are encoded with the pressure given by the fractional gradient $\nabla(-\Delta)^{-\sigma}u$. In addition, replacing the classical time derivative by a Caputo derivative introduces memory into the evolution. The resulting equation combines nonlinear degeneracy with nonlocality in both space and time, making the preservation of positivity, mass, and stability particularly relevant to its numerical approximation. There are several difficulties in constructing such an approximation. The diffusivity vanishes where the density is zero, while the pressure gradient depends on the density in the whole space. Moreover, the time derivative depends on the previous history of the solution. Hence, the discretizations of the flux and of the memory term have to be considered together when proving stability and convergence. 

In this paper, we introduce and analyze numerical approximations of
\begin{equation}\label{fpe}
\begin{aligned}
\partial_t^\alpha u-\nabla\cdot\left(u^{m-1}\nabla(-\Delta)^{-\sigma}u\right)&=0
\quad\text{in }\mathbb{R}^d\times(0,T),\\
u(\cdot,0)&=u_0\quad\text{in }\mathbb{R}^d,
\end{aligned}
\end{equation}
where $m>1$, $0<\sigma<1$, and $0<\alpha<1$. We take non-negative initial data $u_0\in L^1(\mathbb{R}^d)\cap H^r(\mathbb{R}^d)$, with $r\geq d/2+3$, and take the regularity assumption $u\in C([0,T];L^1(\mathbb{R}^d)\cap H^r(\mathbb{R}^d))$. Existence of the solutions of such regularity was shown for $\alpha=1$ and $m=1$ by Choi and Yeong \cite{choi2021classical}. Also, when additional spatial and temporal regularity is required, we impose it where needed for consistency and convergence. In particular, the convergence estimates below are conditional on the stated consistency bounds. We do not establish existence or uniqueness for the continuous problem in this generality.

For sufficiently regular functions, the Caputo derivative can be written as
\begin{equation}\label{eq:CaputoDerivative}
\partial_t^\alpha y(t) =\frac{y(t)-y(0)}{\Gamma(1-\alpha)t^\alpha} +\frac{\alpha}{\Gamma(1-\alpha)} \int_0^t\frac{y(t)-y(s)}{(t-s)^{1+\alpha}}\,ds.
\end{equation}
On the other hand, the pressure is formally $p=(-\Delta)^{-\sigma}u$ and since only its gradient enters \eqref{fpe}, we define the spatial interaction directly by the Fourier multiplier
\begin{equation}
\widehat{\nabla(-\Delta)^{-\sigma}v}(\xi)=i\xi|\xi|^{-2\sigma}\widehat v(\xi).
\end{equation}
For smooth, rapidly decaying functions, the corresponding singular-integral representation is
\begin{equation}
\begin{aligned}
\nabla(-\Delta)^{-\sigma}v(x)
&=c_{d,\sigma}\operatorname{P.V.}\int_{\mathbb{R}^d}
\frac{v(x+y)-v(x)}{|y|^{d+1-2\sigma}}\frac{y}{|y|}\,dy\\
&=c_{d,\sigma}\operatorname{P.V.}\int_{\mathbb{R}^d}
\frac{v(x)-v(z)}{|x-z|^{d+1-2\sigma}}\frac{x-z}{|x-z|}\,dz,
\end{aligned}
\end{equation}
where the normalization constants $c_{d,\sigma}>0$ are known and symmetric limits are used when needed. Defining the gradient directly also avoids the low-frequency normalization of the pressure in one dimension when $\sigma\geq1/2$. This operator was first introduced by Horv\'ath \cite{Horvath1959}.We write
\begin{equation}
\mathcal A u=\nabla\cdot\left(u^{m-1}\nabla(-\Delta)^{-\sigma}u\right).
\end{equation}
For the classical time derivative and $m=2$, Caffarelli and V\'azquez~\cite{CaffarelliVazquez2011} established the existence of bounded weak solutions with finite propagation speed for compactly supported data. Stan, del Teso, and V\'azquez~\cite{Stan2016FiniteInfinite} studied the more general mobility $u^{m-1}$ and the dependence of propagation on $m$. The derivation of this equation for the classical time derivative from physical principles was provided by the first author \cite{Plociniczak2019}. In one dimension, the use of integrated cumulative variables $v(x,t)=\int_{-\infty}^x u(y,t)\,dy$ transforms the density equation to a nondivergence form equation that involves a fractional Laplacian. This formulation has the advantage of providing access to comparison arguments. In addition, a related viscosity theory is developed in~\cite{chasseigne2017nonlocal}. In the context of the mathematical analysis of dislocation density, this equation with $m=2$ and $\sigma=1/2$ was analyzed by Biler, Karch, and Monneau \cite{Biler2010}. For our time-fractional equation, the same integration in space gives, under sufficient regularity and vanishing flux at $-\infty$,
\begin{equation}\label{eq:introduction_integrated}
\partial_t^\alpha v+|\partial_xv|^{m-1}(-\Delta)^sv=0,
\quad s=1-\sigma.
\end{equation}
Notice that $v$ is bounded and non-decreasing when $u$ is non-negative and integrable. This gives a useful class of functions on which to study the comparison and stability of the numerical method. It also allows us to recover the density approximation by discrete differentiation. The relation between these two formulations is used throughout our one-dimensional analysis.

The combination of fractional pressure and a fractional time derivative was studied by Allen, Caffarelli, and Vasseur~\cite{Allenetal2017}. For $m=2$ and $0<\sigma<1/2$, they proved the existence and H\"older regularity of weak solutions under suitable decay assumptions on the initial data and forcing. Their results provide an analytical foundation for the time-fractional model, but do not supply a general uniqueness theorem or the strong regularity assumed in our error analysis. Viscosity methods for parabolic problems with a Caputo derivative have also been developed by Topp and Yangari~\cite{ToppYangari2017}. The closest numerical predecessor is the work of del Teso and Jakobsen~\cite{delTesoJakobsen2026Convergent} for the classical-time equation. They constructed a monotone upwind approximation of the one-dimensional cumulative problem and recovered a conservative density scheme by numerical differentiation. Their convergence theory includes nonsmooth and measure-valued initial data. The cumulative formulation, the choice of the upwind derivative, and the recovery of the density are therefore already available in the classical-time setting. In the present work, we study how this construction can be combined with the L1 history term, both for implicit and explicit time stepping. This requires comparison estimates that retain the contribution of all previous time levels.

A different approach is taken by Carrillo, Fronzoni, and S\"uli~\cite{CarrilloFronzoniSuli2025}, who analyze a finite element method for a parabolically regularized model on bounded domains in dimensions two and three, with pressure defined through a spectral fractional Neumann Laplacian. For the distinct class of equations $\partial_tu-\mathfrak L[\varphi(u)]=f$, robust numerical methods are developed by del Teso, Endal, and Jakobsen~\cite{delTeso2018}. In that class, a linear nonlocal operator acts on a nonlinear function of the density. In \eqref{fpe}, the nonlinearity instead enters through the product of the mobility and the pressure gradient. Thus, although positivity and conservation are common numerical objectives, spatial operators have different structures. The analysis in the following uses the conservative flux in the density variable and the comparison property in the cumulative variable. Our time discretization uses the L1 formula, whose stability and accuracy for time-fractional diffusion were studied, for example, by Lin and Xu~\cite{LinXu2007}.

In addition, a relevant issue in its analysis is the regularity of the solution near the initial time. The usual $O(\tau^{2-\alpha})$ estimate requires time regularity, which does not need to hold even for smooth initial data. Jin, Lazarov, and Zhou~\cite{JinLazarovZhou2016L1} analyzed the L1 method for linear subdiffusion with smooth and nonsmooth initial data. Stynes, O'Riordan, and Gracia~\cite{StynesORiordanGracia2017} studied the effect of the initial singularity on finite difference approximations on uniform and graded meshes. Kopteva~\cite{Kopteva2019L1} provided a framework for L1 error estimates in the maximum and $L^2$ norms. These works explain why requiring the consistency error to tend to zero uniformly over all time levels can be too restrictive. They also provide the background for the weighted temporal estimate used here.

More recently, del Teso and P\l{}ociniczak~\cite{delTeso2026Numerical} combined L1 time stepping with monotone spatial discretizations for diffusion equations nonlocal in space and time, obtaining convergence and regularity results in a viscosity framework. Their spatial operator has a prescribed space-time dependent diffusivity. In \eqref{eq:introduction_integrated}, the coefficient $|\partial_xv|^{m-1}$ depends on the unknown solution and can vanish. For this reason, the monotonicity of the complete nonlinear spatial discretization has to be checked before applying a temporal comparison argument.

The spatial construction used here is based on the fractional powers of the discrete Laplacian studied by Ciaurri et al.~\cite{ciaurri2015fractional,ciaurri2018nonlocal}. These works also provide the kernel formulas and consistency estimates used in our one-dimensional analysis. For the semigroup pressure-gradient weights, taking the negative discrete divergence gives the fractional power $(-\Delta_h)^{1-\sigma}$. The resulting operator has a difference representation with non-negative weights. In several dimensions, the relevant sign condition concerns the sum over the spatial directions. The individual directional coefficients need not be non-negative. We verify the summed sign condition and the required cancellation and summability properties in the appendix.

The present work combines a conservative upwind flux with implicit and explicit L1 time discretizations. We first study the implicit density scheme in an arbitrary space dimension. Under the stated assumptions on the weights and summability of the fluxes, we establish structural estimates for the implicit scheme, including non-negativity, mass conservation, maximum-norm and energy bounds, and existence of discrete solutions. These results use the upwind form of the flux, the summed sign condition on the spatial weights, and the positivity of the L1 history coefficients. The implicit estimates do not require a CFL restriction.

In one dimension, the cumulative formulation yields comparison and contraction estimates. For $m\geq2$, these estimates give a uniform convergence of the cumulative approximation and a convergence of the density against compactly supported smooth test functions, provided that the assumed spatial and weighted temporal consistency bounds hold. More precisely, the comparison estimate for the nonlinear resolvent gives a scalar L1 inequality for the maximum-norm error. A weighted discrete comparison estimate then controls the accumulation of the consistency errors. This yields the convergence result in Theorem~\ref{thm:integrated_convergence}.

The temporal analysis allows for the bounds $\|\partial_t v(t)\|_\infty\leq C_vt^{\alpha-1}$ and $\|\partial_{tt}v(t)\|_\infty\leq C_vt^{\alpha-2}$ for $t>0$. Under these bounds and the stated spatial assumptions, the implicit error estimate takes the form of
\begin{equation}\label{eq:introduction_error}
\max_{1\leq n\leq N}\|V^n-v^n\|_\infty
\leq\|V^0-v^0\|_\infty+C\left(\tau^\alpha+h+\eta_{\mathcal A}(h)\right).
\end{equation}
Here, $h$ and $\tau$ are the spatial and temporal mesh sizes, $t_n=n\tau$, $V^n$ is the discrete cumulative solution, and $v^n$ denotes the sampled exact solution. The quantity $\eta_{\mathcal A}(h)$ is the consistency error of the fractional discrete Laplacian, and $C$ is independent of $h$ and $\tau$. The estimate is uniform over all time levels, including the first one. It follows from a time-weighted residual bound and does not require the unweighted first-step residual to tend to zero. The stated regularity of the cumulative solution remains an assumption on the continuous problem.

We also derive a sufficient CFL condition for the explicit scheme throughout $0<\sigma<1$. The restriction involves $\tau^\alpha$ and reflects the spatial scaling on either side of $\sigma=1/2$, with a logarithmic correction at this value. Under this condition, comparison and contraction give an analogous error estimate, with an additional temporal modulus for the continuous spatial operator. This modulus accounts for evaluating the numerical flux at the previous time level; its required properties are stated in Theorem~\ref{thm:explicit_convergence}. Thus, the analysis extends the cumulative upwind approach to Caputo time stepping while retaining explicit assumptions on the regularity of the solution being approximated.

The paper is organized as follows. Section~2 recalls a continuous maximum principle. Section~3 introduces the schemes and develops their structural properties, consistency, and one-dimensional convergence analysis. Section~4 presents numerical experiments. Appendix~A describes the discretizations of the pressure-gradient and the connection with the discrete fractional Laplacian.

\section{Preliminaries}

\begin{theorem}[Continuous maximum principle]
\label{thm:continuous_maximum_principle}
Let $u(x,t)$ be a sufficiently regular, bounded, and non-negative function such that $\mathcal{A}u$ is well-defined. If, for a fixed $t>0$, $u(\cdot,t)$ has a global maximum at $x_0$, that is,
\begin{equation}
	u(x_0,t)\geq u(x,t), \quad x\in\mathbb{R}^d,
\end{equation}
then
\begin{equation}
	\mathcal{A}u(x_0,t)\leq0.
\end{equation}
\end{theorem}

\begin{proof}
Since $u(\cdot,t)$ has a global maximum at $x_0$, we have
\begin{equation}
	\nabla u(x_0,t)=0.
\end{equation}
Expanding $\mathcal{A}u$, we obtain
\begin{equation}
	\mathcal{A}u=\nabla\left(u^{m-1}\right)\cdot\nabla(-\Delta)^{-\sigma}u+u^{m-1}\Delta(-\Delta)^{-\sigma}u.
\end{equation}
If $u(x_0,t)=0$, then the non-negativity of $u$ and the maximality of $u(x_0,t)$ imply that $u(\cdot,t)\equiv0$, and the conclusion is immediate. Hence, assume that $u(x_0,t)>0$. Then
\begin{equation}
	\nabla\left(u^{m-1}\right)(x_0,t)=(m-1)u(x_0,t)^{m-2}\nabla u(x_0,t)=0.
\end{equation}
Therefore, the first term vanishes at $x_0$.

For the second term from the identity
\begin{equation}
	\Delta(-\Delta)^{-\sigma}u=-(-\Delta)^{1-\sigma}u,
\end{equation}
and by the singular-integral representation of the fractional Laplacian, we have
\begin{equation}
	\Delta(-\Delta)^{-\sigma}u(x_0,t) =C_{d,1-\sigma}\operatorname{P.V.}\int_{\mathbb{R}^d} \frac{u(y,t)-u(x_0,t)}{|x_0-y|^{d+2(1-\sigma)}}\,dy.
\end{equation}
Since $x_0$ is a global maximum, it follows that
\begin{equation}
	u(y,t)-u(x_0,t)\leq0, \quad y\in\mathbb{R}^d,
\end{equation}
and hence
\begin{equation}
	\Delta(-\Delta)^{-\sigma}u(x_0,t)\leq0.
\end{equation}
Since $u(x_0,t)^{m-1}\geq0$, we conclude that
\begin{equation}
	\mathcal{A}u(x_0,t)\leq0,
\end{equation}
which proves the claim.
\end{proof}

\begin{lemma}[$L^\infty$ stability]
\label{lem:continuous_linf}
Let $u_0\in L^\infty(\mathbb{R}^d)$ be non-negative, and assume that $u(x,t)$ is a sufficiently regular, bounded and non-negative solution to \eqref{fpe}. Then
\begin{equation}
	\sup_{t\in[0,T]}\|u(\cdot,t)\|_\infty\leq\|u_0\|_\infty.
\end{equation}
\end{lemma}

\begin{proof}

Suppose to the contrary that
\begin{equation}
	\sup_{t\in[0,T]}\|u(\cdot,t)\|_\infty>\|u_0\|_\infty.
\end{equation}
Then, there exist $t_*\in(0,T]$ and $x_*\in\mathbb{R}^d$ such that
\begin{equation}
	u(x_*,t_*)=\max_{0\leq s\leq t_*}\sup_{x\in\mathbb{R}^d}u(x,s)>\|u_0\|_\infty.
\end{equation}
In particular,
\begin{equation}
	u(x_*,t_*)\geq u(x_*,s), \quad 0\leq s\leq t_*,
\end{equation}
and
\begin{equation}
	u(x_*,t_*)-u(x_*,0)>0.
\end{equation}
Applying \eqref{eq:CaputoDerivative} to the function $s\mapsto u(x_*,s)$ yields
\begin{equation}
	\partial_t^\alpha u(x_*,t_*)>0.
\end{equation}
On the other hand, $x_*$ is a global maximum of $u(\cdot,t_*)$. By Theorem~\ref{thm:continuous_maximum_principle},
\begin{equation}
	\mathcal{A}u(x_*,t_*)\leq0.
\end{equation}
But since \eqref{fpe} can be written as
\begin{equation}
	\partial_t^\alpha u=\mathcal{A}u,
\end{equation}
we obtain
\begin{equation}
	\partial_t^\alpha u(x_*,t_*)\leq0,
\end{equation}
which is a contradiction. Therefore,
\begin{equation}
	\|u(\cdot,t)\|_\infty\leq\|u_0\|_\infty, \quad t\in[0,T],
\end{equation}
and the proof is complete.

\end{proof}

\section{Numerical scheme}
We introduce the space and time grids
\begin{equation}
	x_j = jh, \quad j\in\mathbb{Z}^d,
	\quad
	t_n=n\tau, \quad n=0,\ldots,N.
\end{equation}

We write
\begin{equation}
	U_j^n \approx u(x_j,t_n),
	\quad
	U_j^0=u_0(x_j),
\end{equation}
that is, $U^n_j$ is the numerical approximation of $u(x_j, t_n)$. We discretize the Caputo derivative by the L1 scheme
\begin{equation}
	D_\tau^\alpha U_j^n
	=
	\frac{\tau^{-\alpha}}{\Gamma(2-\alpha)}
	\sum_{k=0}^{n-1}
	a_k
	\left(
	U_j^{n-k}-U_j^{n-k-1}
	\right),
	\quad n=1,\ldots,N,
\end{equation}
where
\begin{equation}
	a_k=(k+1)^{1-\alpha}-k^{1-\alpha}.
\end{equation}
The coefficients satisfy
\begin{equation}
	1=a_0>a_1>a_2>\cdots>0.
\end{equation}

For each direction $i=1,\ldots,d$, we introduce the staggered lattice
\begin{equation}
	\Gamma_i=\mathbb{Z}^d-\frac{1}{2}e_i.
\end{equation}
Next, the $i$-th component of the pressure gradient at the cell interface
$x_{j+\frac{1}{2}e_i}$ can be approximated by
\begin{equation}
	(G_{h,i}U)_{j+\frac{1}{2}e_i}=\sum_{k\in\mathbb{Z}^d}g^i_{k-j-\frac{1}{2}e_i}(h)U_k,
\end{equation}
where $g_\gamma^i(h)$ with $\gamma\in\Gamma_i$ are discrete weights approximating the $i$-th component of $\nabla(-\Delta)^{-\sigma}$.

Furthermore, we impose the following structural assumptions on the weights. First,
they are antisymmetric,
\begin{equation}
	g_{-\gamma}^i(h)=-g_\gamma^i(h),\quad\gamma\in\Gamma_i,\quad i=1,\ldots,d.
\end{equation}
Note that this means that the pressure gradient operator annihilates constants. More precisely, we assume
\begin{equation}
	\sum_{\gamma\in\Gamma_i}^{\mathrm{sym}}g_\gamma^i(h)=0,\quad i=1,\ldots,d,
\end{equation}
where the sum is understood in the symmetric principal-value sense whenever
it is not absolutely convergent. Second, the weights must satisfy the decay estimate
\begin{equation}
	|g_\gamma^i(h)|\leq C h^{2\sigma-1}|\gamma|^{-(d+1-2\sigma)},\quad|\gamma|\rightarrow\infty,
\end{equation}
where $C>0$ is independent of $h$. Finally, for $\beta\in\mathbb{Z}^d$ we define
\begin{equation}
	c_\beta^i(h)=g^i_{\beta-\frac{1}{2}e_i}(h)-g^i_{\beta+\frac{1}{2}e_i}(h).
\end{equation}
Also, we assume that the sum of the directional coefficients satisfies
\begin{equation}\label{eqn:CoefPositive}
	C_\beta(h):=\frac1h\sum_{i=1}^d c_\beta^i(h)\geq0
	\quad\text{for all }\beta\neq0,
\end{equation}
and
\begin{equation}\label{eqn:CoefConsistency}
	\sum_{\beta\in\mathbb{Z}^d}c_\beta^i(h)=0,\quad i=1,\ldots,d.
\end{equation}
These sign and cancellation properties will be used below to establish the discrete maximum principle.

In the sequel, whenever the pressure-gradient operator $G_{h,i}U$ is used, we assume that its defining series is well-defined, either absolutely or in the symmetric principal-value sense specified above. Moreover, whenever the representation in terms of the coefficients $c_\beta^i(h)$ is used, we assume
\begin{equation}
	\sum_{\beta\in\mathbb{Z}^d}|c_\beta^i(h)|<\infty, \quad i=1,\ldots,d.
\end{equation}
The sign, cancellation, and absolute summability conditions are satisfied by semigroup construction in arbitrary dimension; see Proposition~\ref{prop:semigroup_multidimensional}. No sign condition is imposed on an individual $c_\beta^i(h)$ when $d>1$. In particular, if $U\in\ell^1(\mathbb{Z}^d)$ and the weights $g_\gamma^i(h)$ are bounded for fixed $h$, then the series defining $G_{h,i}U$ is absolutely convergent.

For convenience, let
\begin{equation}
	V^i_{j+\frac{1}{2}e_i}(U)=(G_{h,i}U)_{j+\frac{1}{2}e_i},
\end{equation}
denote
\begin{equation}
	r^+=\max\{r,0\},\quad r^-=\min\{r,0\}.
\end{equation}
We define the numerical flux in the $i$-th direction by
\begin{equation}
	F_{j+\frac{1}{2}e_i}^{(i)}(U)=-\left(V^i_{j+\frac{1}{2}e_i}(U)\right)^+U_{j+e_i}^{m-1}-\left(V^i_{j+\frac{1}{2}e_i}(U)\right)^-U_j^{m-1},
\end{equation}
and
\begin{equation}
	F_{j-\frac{1}{2}e_i}^{(i)}(U)=-\left(V^i_{j-\frac{1}{2}e_i}(U)\right)^+U_j^{m-1}-\left(V^i_{j-\frac{1}{2}e_i}(U)\right)^-U_{j-e_i}^{m-1}.
\end{equation}
The discrete spatial operator $\mathcal{N}_h$ is defined by
\begin{equation}
	(\mathcal{N}_hU)_j=\frac{1}{h}\sum_{i=1}^d\left(F_{j+\frac{1}{2}e_i}^{(i)}(U)-F_{j-\frac{1}{2}e_i}^{(i)}(U)\right).
\end{equation}
The fully discrete implicit scheme is then given by
\begin{equation}\label{eq:discrete}
	D_\tau^\alpha U_j^n=-(\mathcal{N}_hU^n)_j,\quad j\in\mathbb{Z}^d,\quad n=1,\ldots,N.
\end{equation}

\subsection{Properties of the discrete scheme}
Next, we establish the maximum principle for the discrete spatial operator.
\begin{theorem}[Discrete maximum principle]
\label{thm:discrete_maximum_principle}
Let $U=\{U_k\}_{k\in\mathbb{Z}^d}$ be bounded and non-negative.
Assume that the coefficients $c_\beta^i(h)$ introduced above satisfy \eqref{eqn:CoefPositive}-\eqref{eqn:CoefConsistency}. If $U$ has a global maximum at an index $j\in\mathbb{Z}^d$, that is,
\begin{equation}
	U_j\geq U_k \quad \text{for all } k\in\mathbb{Z}^d,
\end{equation}
then
\begin{equation}
	(\mathcal{N}_hU)_j\geq 0.
\end{equation}
\end{theorem}
\begin{proof}
For convenience, write
\begin{equation}
	V^i_{j+\frac{1}{2}e_i}=(G_{h,i}U)_{j+\frac{1}{2}e_i}.
\end{equation}
Then the numerical fluxes are
\begin{equation}
	F_{j+\frac{1}{2}e_i}^{(i)}(U)=-\left(V^i_{j+\frac{1}{2}e_i}\right)^+U_{j+e_i}^{m-1}-\left(V^i_{j+\frac{1}{2}e_i}\right)^-U_j^{m-1},
\end{equation}
and
\begin{equation}
	F_{j-\frac{1}{2}e_i}^{(i)}(U)=-\left(V^i_{j-\frac{1}{2}e_i}\right)^+U_j^{m-1}-\left(V^i_{j-\frac{1}{2}e_i}\right)^-U_{j-e_i}^{m-1}.
\end{equation}
Therefore,
\begin{equation}
	(\mathcal{N}_hU)_j=
	\frac{1}{h}\sum_{i=1}^d
	\left[-\left(V^i_{j+\frac{1}{2}e_i}\right)^+U_{j+e_i}^{m-1}-\left(V^i_{j+\frac{1}{2}e_i}\right)^-U_j^{m-1}+\left(V^i_{j-\frac{1}{2}e_i}\right)^+U_j^{m-1}+\left(V^i_{j-\frac{1}{2}e_i}\right)^-U_{j-e_i}^{m-1}
	\right].
\end{equation}
Since $U_j$ is a global maximum and $U$ is non-negative, we have $0\leq U_{j+e_i}^{m-1}\leq U_j^{m-1}$ and $0\leq U_{j-e_i}^{m-1}\leq U_j^{m-1}$. Moreover,
\begin{equation}
	\left(V^i_m\right)^+\geq0, \quad \left(V^i_m\right)^-\leq0.
\end{equation}
This implies,
\begin{equation}
	-\left(V^i_{j+\frac{1}{2}e_i}\right)^+U_{j+e_i}^{m-1}\geq-\left(V^i_{j+\frac{1}{2}e_i}\right)^+U_j^{m-1},
\end{equation}
while
\begin{equation}
	\left(V^i_{j-\frac{1}{2}e_i}\right)^-U_{j-e_i}^{m-1}\geq\left(V^i_{j-\frac{1}{2}e_i}\right)^-U_j^{m-1}.
\end{equation}
Hence,
\begin{equation}
	(\mathcal{N}_hU)_j\geq\frac{U_j^{m-1}}{h}\sum_{i=1}^d\left[\left(V^i_{j-\frac{1}{2}e_i}\right)^++\left(V^i_{j-\frac{1}{2}e_i}\right)^--\left(V^i_{j+\frac{1}{2}e_i}\right)^+-\left(V^i_{j+\frac{1}{2}e_i}\right)^-\right].
\end{equation}
Using $V^i_m=\left(V^i_m\right)^+ +\left(V^i_m\right)^-$ we obtain
\begin{equation}
	(\mathcal{N}_hU)_j\geq\frac{U_j^{m-1}}{h}\sum_{i=1}^d\left(V^i_{j-\frac{1}{2}e_i}-V^i_{j+\frac{1}{2}e_i}\right).
\end{equation}

By the definition of the discrete pressure gradient, we have that
\begin{equation}
	V^i_{j+\frac{1}{2}e_i}-V^i_{j-\frac{1}{2}e_i}=\sum_{k\in\mathbb{Z}^d}\left(
	g^i_{k-j-\frac{1}{2}e_i}(h)-g^i_{k-j+\frac{1}{2}e_i}(h)\right)U_k=\sum_{\beta\in\mathbb{Z}^d}c_\beta^i(h)U_{j+\beta}.
\end{equation}
Using \eqref{eqn:CoefConsistency} we may rewrite this as
\begin{equation}
	V^i_{j+\frac{1}{2}e_i}-V^i_{j-\frac{1}{2}e_i}=\sum_{\beta\neq0}c_\beta^i(h)\left(U_{j+\beta}-U_j\right).
\end{equation}
But since $U_j$ is a global maximum, we further have
\begin{equation}
	U_{j+\beta}-U_j\leq0 \quad \text{for every } \beta\neq0.
\end{equation}

Next, summing over the spatial directions and using \eqref{eqn:CoefPositive}, we obtain
\begin{equation}
	\frac1h\sum_{i=1}^d\left(V^i_{j+\frac{1}{2}e_i}-V^i_{j-\frac{1}{2}e_i}\right)
	=\sum_{\beta\neq0}C_\beta(h)(U_{j+\beta}-U_j)\leq0.
\end{equation}
Therefore,
\begin{equation}
	\sum_{i=1}^d\left(V^i_{j-\frac{1}{2}e_i}-V^i_{j+\frac{1}{2}e_i}\right)\geq0.
\end{equation}
Finally, since $U_j\geq0$ and $m>1$,
\begin{equation}
	U_j^{m-1}\geq0.
\end{equation}
It follows that
\begin{equation}
	(\mathcal{N}_hU)_j
	\geq
	\frac{U_j^{m-1}}{h}
	\sum_{i=1}^d
	\left(
	V^i_{j-\frac{1}{2}e_i}
	-
	V^i_{j+\frac{1}{2}e_i}
	\right)
	\geq0,
\end{equation}
which proves the claim and concludes the proof. 
\end{proof}

We will also need the following sign property at a zero minimum. Interestingly, this property follows directly from the upwind structure of the numerical flux and does not need any additional assumptions on the sign of the scheme's coefficients.
\begin{lemma}[Sign property at a zero minimum]
\label{lem:zero_minimum_sign}
Let $U=\{U_k\}_{k\in\mathbb{Z}^d}$ be non-negative. If $U_j=0$ for some $j\in\mathbb{Z}^d$, then
\begin{equation}
	(\mathcal{N}_hU)_j\leq0.
\end{equation}
\end{lemma}
\begin{proof}
Notice that since $m>1$ and $U_j=0$, we have $U_j^{m-1}=0$ and, therefore, the two fluxes adjacent to $x_j$ simplify
\begin{equation}
	F_{j+\frac{1}{2}e_i}^{(i)}(U)=-\left(V^i_{j+\frac{1}{2}e_i}\right)^+U_{j+e_i}^{m-1},
\end{equation}
and
\begin{equation}
	F_{j-\frac{1}{2}e_i}^{(i)}(U)=-\left(V^i_{j-\frac{1}{2}e_i}\right)^-U_{j-e_i}^{m-1}.
\end{equation}
Hence,
\begin{equation}
	(\mathcal{N}_hU)_j=\frac{1}{h}\sum_{i=1}^d\left[-\left(V^i_{j+\frac{1}{2}e_i}\right)^+ U_{j+e_i}^{m-1}+\left(V^i_{j-\frac{1}{2}e_i}\right)^-U_{j-e_i}^{m-1}\right].
\end{equation}
Since $U$ is non-negative,
\begin{equation}
	U_{j+e_i}^{m-1}\geq0, \quad U_{j-e_i}^{m-1}\geq0.
\end{equation}
Moreover,
\begin{equation}
	\left(V^i_{j+\frac{1}{2}e_i}\right)^+\geq0, \quad \left(V^i_{j-\frac{1}{2}e_i}\right)^-\leq0.
\end{equation}
Thus, every term in the sum is non-positive and therefore $(\mathcal{N}_hU)_j\leq0$.
\end{proof}
Next, we introduce the discrete diffusion operator associated with the pressure-gradient approximation. For a grid function $U$, define
\begin{equation}
	(L_hU)_j=-\frac{1}{h}\sum_{i=1}^d\left((G_{h,i}U)_{j+\frac{1}{2}e_i}-(G_{h,i}U)_{j-\frac{1}{2}e_i}\right).
\end{equation}
Thus, $L_h=-\nabla_h\cdot G_h$ is the discrete version of $(-\Delta)^{1-\sigma}$. Using the definition of $c_\beta^i(h)$ further, we have
\begin{equation}
	(G_{h,i}U)_{j+\frac{1}{2}e_i}-(G_{h,i}U)_{j-\frac{1}{2}e_i}=\sum_{\beta\in\mathbb{Z}^d}c_\beta^i(h)U_{j+\beta}.
\end{equation}
Because $\sum_{\beta\in\mathbb{Z}^d}c_\beta^i(h)=0$, it follows that
\begin{equation}
	(L_hU)_j=\frac{1}{h}\sum_{i=1}^d\sum_{\beta\neq0}c_\beta^i(h)\left(U_j-U_{j+\beta}\right).
\end{equation}
Moreover, the antisymmetry of the weights $g_\gamma^i(h)$ implies that
\begin{equation}
	c_{-\beta}^i(h)=c_\beta^i(h).
\end{equation}
To see this, note that
\begin{equation}
	c_{-\beta}^i(h)=g^i_{-\beta-\frac{1}{2}e_i}(h)-g^i_{-\beta+\frac{1}{2}e_i}(h)=-g^i_{\beta+\frac{1}{2}e_i}(h)+g^i_{\beta-\frac{1}{2}e_i}(h)=c_\beta^i(h).
\end{equation}
Therefore, the coefficients $C_\beta(h)$ introduced in \eqref{eqn:CoefPositive} satisfy $C_\beta(h)=C_{-\beta}(h)\geq0$, and $L_h$ admits the representation
\begin{equation}
	(L_hU)_j=\sum_{\beta\neq0}C_\beta(h)\left(U_j-U_{j+\beta}\right).
\end{equation}
\begin{lemma}[Discrete Stroock-Varopoulos-type inequality]
\label{lem:discrete_sv}

Let $U=\{U_j\}_{j\in\mathbb{Z}^d}$ be a grid function and let $\Phi:\mathbb{R}\rightarrow\mathbb{R}$ be non-decreasing. Furthermore, assume that all sums appearing below are absolutely convergent. Then
\begin{equation}
	\sum_{j\in\mathbb{Z}^d}\Phi(U_j)(L_hU)_j=\frac{1}{2}\sum_{j\in\mathbb{Z}^d}\sum_{\beta\neq0}C_\beta(h)\left(\Phi(U_j)-\Phi(U_{j+\beta})\right)\left(U_j-U_{j+\beta}\right)\geq0.
\end{equation}
\end{lemma}
\begin{proof}
By the representation of $L_h$ obtained above,
\begin{equation}
	\sum_{j\in\mathbb{Z}^d}\Phi(U_j)(L_hU)_j=\sum_{j\in\mathbb{Z}^d}\sum_{\beta\neq0}C_\beta(h)\Phi(U_j)\left(U_j-U_{j+\beta}\right).
\end{equation}

Changing variables $j\mapsto j+\beta$ and $\beta\mapsto-\beta$, and using $C_{-\beta}(h)=C_\beta(h)$, gives us
\begin{equation}
	\sum_{j\in\mathbb{Z}^d}\Phi(U_j)(L_hU)_j=-\sum_{j\in\mathbb{Z}^d}\sum_{\beta\neq0}C_\beta(h)\Phi(U_{j+\beta})\left(U_j-U_{j+\beta}\right).
\end{equation}
Next, averaging the last two identities yields
\begin{equation}
	\sum_{j\in\mathbb{Z}^d}\Phi(U_j)(L_hU)_j=\frac{1}{2}\sum_{j\in\mathbb{Z}^d}\sum_{\beta\neq0}C_\beta(h)\left(\Phi(U_j)-\Phi(U_{j+\beta})\right)\left(U_j-U_{j+\beta}\right).
\end{equation}
Since $\Phi$ is non-decreasing, we have
\begin{equation}
	\left(\Phi(U_j)-\Phi(U_{j+\beta})\right)\left(U_j-U_{j+\beta}\right)\geq0.
\end{equation}
Together with $C_\beta(h)\geq0$, this concludes the proof. 
\end{proof}

We can now proceed to the positivity property of the nonlinear spatial operator $\mathcal{N}_h$.
\begin{lemma}[Positivity of $\mathcal{N}_h$]
\label{lem:Nh_positivity}

Let $U$ be bounded and non-negative, and assume that all sums appearing below are absolutely convergent. Then
\begin{equation}
	\sum_{j\in\mathbb{Z}^d}U_j(\mathcal{N}_hU)_j\geq0.
\end{equation}
\end{lemma}
\begin{proof}
Define $\Phi(s)=s^m/m$ for $s\geq0$. Since $m>1$, $\Phi$ is increasing, and $\Phi'(s)=s^{m-1}$. Using the conservative form of $\mathcal{N}_h$ and summation by parts, we obtain
\begin{equation}
	\sum_{j\in\mathbb{Z}^d}U_j(\mathcal{N}_hU)_j=-\frac{1}{h}\sum_{i=1}^d\sum_{j\in\mathbb{Z}^d}(U_{j+e_i}-U_j)F_{j+\frac{1}{2}e_i}^{(i)}(U).
\end{equation}
Substituting the upwind flux gives
\begin{equation}
	\sum_{j\in\mathbb{Z}^d}U_j(\mathcal{N}_hU)_j=\frac{1}{h}\sum_{i=1}^d\sum_{j\in\mathbb{Z}^d}(U_{j+e_i}-U_j)\left[\left(V^i_{j+\frac{1}{2}e_i}\right)^+U_{j+e_i}^{m-1}+\left(V^i_{j+\frac{1}{2}e_i}\right)^-U_j^{m-1}\right].
\end{equation}
We claim that for every $i$ and $j$, we have
\begin{equation}
	(U_{j+e_i}-U_j)\left[\left(V^i_{j+\frac{1}{2}e_i}\right)^+U_{j+e_i}^{m-1}+\left(V^i_{j+\frac{1}{2}e_i}\right)^-U_j^{m-1}\right]\geq V^i_{j+\frac{1}{2}e_i}\left(\Phi(U_{j+e_i})-\Phi(U_j)\right).
\end{equation}
To check this, recall that for any $a,b\geq0$, it follows that
\begin{equation}
	\Phi(b)-\Phi(a)=\int_a^b s^{m-1}\,ds.
\end{equation}
If $V^i_{j+\frac{1}{2}e_i}\geq0$, then
\begin{equation}
	(U_{j+e_i}-U_j)U_{j+e_i}^{m-1}\geq\Phi(U_{j+e_i})-\Phi(U_j),
\end{equation}
because $s^{m-1}$ is increasing. Multiplication by $V^i_{j+\frac{1}{2}e_i}\geq0$ does not change the inequality. If $V^i_{j+\frac{1}{2}e_i}<0$, then
\begin{equation}
	(U_{j+e_i}-U_j)U_j^{m-1}\leq\Phi(U_{j+e_i})-\Phi(U_j).
\end{equation}
Multiplication by $V^i_{j+\frac{1}{2}e_i}<0$ reverses the inequality. Hence, the claimed lower bound holds in both cases.

Therefore,
\begin{equation}
	\sum_{j\in\mathbb{Z}^d}U_j(\mathcal{N}_hU)_j\geq\frac{1}{h}\sum_{i=1}^d\sum_{j\in\mathbb{Z}^d}V^i_{j+\frac{1}{2}e_i}\left(\Phi(U_{j+e_i})-\Phi(U_j)\right).
\end{equation}
By the discrete summation by parts and the definition of $L_h$, we further have
\begin{equation}
	\frac{1}{h}\sum_{i=1}^d\sum_{j\in\mathbb{Z}^d}V^i_{j+\frac{1}{2}e_i}\left(\Phi(U_{j+e_i})-\Phi(U_j)\right)=\sum_{j\in\mathbb{Z}^d}\Phi(U_j)(L_hU)_j.
\end{equation}
The discrete Stroock-Varopoulos-type inequality therefore yields
\begin{equation}
	\sum_{j\in\mathbb{Z}^d}U_j(\mathcal{N}_hU)_j\geq\sum_{j\in\mathbb{Z}^d}\Phi(U_j)(L_hU)_j\geq0,
\end{equation}
and this finishes the proof. 
\end{proof}

Next, we establish the $\ell^\infty$ stability of the implicit scheme.
\begin{lemma}[Stability in $\ell^\infty$]
\label{infstab}
Let $U^0\in\ell^\infty(\mathbb{Z}^d)$ be non-negative, and assume that $U^n=\{U_j^n\}_{j\in\mathbb{Z}^d}$ is a bounded non-negative solution to \eqref{eq:discrete} for $n=1,\ldots,N$.

Assume additionally that, for every $n$,
$U_j^n\rightarrow0$ as $|j|\rightarrow\infty$. Then, for any $h>0$ and $\tau>0$,
\begin{equation}
	\|U^n\|_\infty\leq\|U^0\|_\infty, \quad n=0,\ldots,N.
\end{equation}
\end{lemma}

\begin{proof}
Let $\lambda_\alpha=\Gamma(2-\alpha)\tau^\alpha$. Multiplying \eqref{eq:discrete} by $\Gamma(2-\alpha)\tau^\alpha$ and rearranging the L1 formula, we obtain
\begin{equation}
	U_j^n+\lambda_\alpha(\mathcal{N}_hU^n)_j=\sum_{k=0}^{n-2}(a_k-a_{k+1})U_j^{n-1-k}+a_{n-1}U_j^0,
\end{equation}
where the sum is understood to be empty when $n=1$.

Since the coefficients $a_k$ are positive and decreasing,
\begin{equation}
	a_k-a_{k+1}>0.
\end{equation}
Moreover, all the coefficients sum to $1$,
\begin{equation}
	\sum_{k=0}^{n-2}(a_k-a_{k+1})+a_{n-1}=a_0=1.
\end{equation}
 Now, fix $n\in\{1,\ldots,N\}$. Since $U^n$ is non-negative and vanishes as $|j|\rightarrow\infty$, there exists an index $j_n\in\mathbb{Z}^d$ such that
\begin{equation}
	U_{j_n}^n=\max_{j\in\mathbb{Z}^d}U_j^n=\|U^n\|_\infty.
\end{equation}
By the discrete maximum principle,
\begin{equation}
	(\mathcal{N}_hU^n)_{j_n}\geq0.
\end{equation}
Evaluating the scheme at $j_n$ therefore gives
\begin{equation}
	\|U^n\|_\infty\leq\sum_{k=0}^{n-2}(a_k-a_{k+1})U_{j_n}^{n-1-k}+a_{n-1}U_{j_n}^0.
\end{equation}
Consequently,
\begin{equation}
	\|U^n\|_\infty\leq\sum_{k=0}^{n-2}(a_k-a_{k+1})\|U^{n-1-k}\|_\infty+a_{n-1}\|U^0\|_\infty.
\end{equation}

We now proceed by induction. For $n=1$, the preceding inequality reduces to
\begin{equation}
	\|U^1\|_\infty\leq\|U^0\|_\infty.
\end{equation}
Assume that $\|U^r\|_\infty\leq\|U^0\|_\infty$ for $r=1,\ldots,n-1$. Then
\begin{equation}
	\|U^n\|_\infty\leq\left(\sum_{k=0}^{n-2}(a_k-a_{k+1})+a_{n-1}\right)\|U^0\|_\infty.
\end{equation}
Since the coefficients form a convex combination,
\begin{equation}
	\|U^n\|_\infty\leq\|U^0\|_\infty.
\end{equation}
This completes the proof.
\end{proof}

\begin{remark}
If $U^n\in\ell^1(\mathbb{Z}^d)$, then $U_j^n\rightarrow0$ as $|j|\rightarrow\infty$, so the additional decay assumption in Lemma~\ref{infstab} is automatic. Thus, once the existence of non-negative $\ell^1$ solutions is established below, the $\ell^\infty$ estimate applies directly to those solutions.
\end{remark}
Next, we state the mass conservation property of the scheme.

\begin{lemma}[Mass conservation]
\label{lem:mass_conservation}
Let $U^0\in\ell^1(\mathbb{Z}^d)$ be non-negative and assume that
$U^n\in\ell^1(\mathbb{Z}^d)$ is a non-negative solution to
\eqref{eq:discrete} for $n=1,\ldots,N$.

Assume additionally that, for all $n=1,\ldots,N$ and
$i=1,\ldots,d$, the numerical flux is absolutely summable,
\begin{equation}
	\sum_{j\in\mathbb{Z}^d}\left|F_{j+\frac{1}{2}e_i}^{(i)}(U^n)\right|<\infty.
\end{equation}
Then
\begin{equation}
	h^d\sum_{j\in\mathbb{Z}^d}U_j^n=h^d\sum_{j\in\mathbb{Z}^d}U_j^0, \quad n=0,\ldots,N.
\end{equation}
In particular, since the solution is non-negative,
\begin{equation}
	\|U^n\|_{1,h}=\|U^0\|_{1,h}, \quad n=0,\ldots,N,
\end{equation}
where $\|U\|_{1,h}=h^d\sum_{j\in\mathbb{Z}^d}|U_j|$.
\end{lemma}

\begin{proof}
Summing \eqref{eq:discrete} over $j\in\mathbb{Z}^d$, we obtain
\begin{equation}
	\sum_{j\in\mathbb{Z}^d}D_\tau^\alpha U_j^n=-\sum_{j\in\mathbb{Z}^d}(\mathcal{N}_hU^n)_j.
\end{equation}
Using the conservative form of $\mathcal{N}_h$ gives
\begin{equation}
	\sum_{j\in\mathbb{Z}^d}(\mathcal{N}_hU^n)_j=\frac{1}{h}\sum_{i=1}^d\sum_{j\in\mathbb{Z}^d}\left(F_{j+\frac{1}{2}e_i}^{(i)}(U^n)-F_{j-\frac{1}{2}e_i}^{(i)}(U^n)\right).
\end{equation}
For every $i$, we have
\begin{equation}
	\sum_{j\in\mathbb{Z}^d}F_{j-\frac{1}{2}e_i}^{(i)}(U^n)=\sum_{j\in\mathbb{Z}^d}F_{j+\frac{1}{2}e_i}^{(i)}(U^n).
\end{equation}
Hence,
\begin{equation}
	\sum_{j\in\mathbb{Z}^d}(\mathcal{N}_hU^n)_j=0,
\end{equation}
and therefore
\begin{equation}
	\sum_{j\in\mathbb{Z}^d}D_\tau^\alpha U_j^n=0.
\end{equation}
Let
\begin{equation}
	M^n=h^d\sum_{j\in\mathbb{Z}^d}U_j^n.
\end{equation}

Since the L1 operator acts only on the time variable and is linear, we obtain
\begin{equation}
	D_\tau^\alpha M^n=0.
\end{equation}
Equivalently,
\begin{equation}
	M^n=\sum_{k=0}^{n-2}(a_k-a_{k+1})M^{n-1-k}+a_{n-1}M^0.
\end{equation}
We prove by induction that $M^n=M^0$. For $n=1$, the preceding identity gives $M^1=M^0$. Assume that $M^r=M^0$ for $r=1,\ldots,n-1$. Then
\begin{equation}
	M^n=\left(\sum_{k=0}^{n-2}(a_k-a_{k+1})+a_{n-1}\right)M^0.
\end{equation}
Since
\begin{equation}
	\sum_{k=0}^{n-2}(a_k-a_{k+1})+a_{n-1}=a_0=1,
\end{equation}
we conclude that
\begin{equation}
	M^n=M^0.
\end{equation}
Thus
\begin{equation}
	h^d\sum_{j\in\mathbb{Z}^d}U_j^n=h^d\sum_{j\in\mathbb{Z}^d}U_j^0.
\end{equation}
Since $U^n$ is non-negative, this is equivalent to
\begin{equation}
	\|U^n\|_{1,h}=\|U^0\|_{1,h},
\end{equation}
which is our claim. 
\end{proof}
\begin{lemma}[Stability in $\ell^2$]
\label{lem:l2_stability}
Let $U^0\in\ell^2(\mathbb{Z}^d)$ be non-negative, and assume that $U^n=\{U_j^n\}_{j\in\mathbb{Z}^d}$ is a bounded non-negative solution to \eqref{eq:discrete} for $n=1,\ldots,N$.

Assume that the sums required in Lemma~\ref{lem:Nh_positivity} are absolutely convergent. Then
\begin{equation}
	\|U^n\|_{2,h}\leq\|U^0\|_{2,h}, \quad n=0,\ldots,N,
\end{equation}
where $\|U\|_{2,h}^2=h^d\sum_{j\in\mathbb{Z}^d}|U_j|^2$.
\end{lemma}
\begin{proof}
Multiplying \eqref{eq:discrete} by $h^dU_j^n$ and summing over $j\in\mathbb{Z}^d$, we obtain
\begin{equation}
	h^d\sum_{j\in\mathbb{Z}^d}U_j^nD_\tau^\alpha U_j^n=-h^d\sum_{j\in\mathbb{Z}^d}U_j^n(\mathcal{N}_hU^n)_j.
\end{equation}
By Lemma~\ref{lem:Nh_positivity},
\begin{equation}
	h^d\sum_{j\in\mathbb{Z}^d}U_j^n(\mathcal{N}_hU^n)_j\geq0,
\end{equation}
and therefore
\begin{equation}
	h^d\sum_{j\in\mathbb{Z}^d}U_j^nD_\tau^\alpha U_j^n\leq0.
\end{equation}
Next, we use a discrete convexity property of the L1 operator. That is, for any real-valued sequence $\{v^n\}_{n\geq0}$, we have 
\begin{equation}
	2v^nD_\tau^\alpha v^n\geq D_\tau^\alpha\left((v^n)^2\right).
\end{equation}
Indeed, using the equivalent representation of the L1 formula,
\begin{equation}
	D_\tau^\alpha v^n=\frac{\tau^{-\alpha}}{\Gamma(2-\alpha)}\left(v^n-\sum_{k=0}^{n-2}(a_k-a_{k+1})v^{n-1-k}-a_{n-1}v^0\right),
\end{equation}
we obtain
\begin{equation}
	2v^nD_\tau^\alpha v^n-D_\tau^\alpha\left((v^n)^2\right)=\frac{\tau^{-\alpha}}{\Gamma(2-\alpha)}\left[\sum_{k=0}^{n-2}(a_k-a_{k+1})(v^n-v^{n-1-k})^2+a_{n-1}(v^n-v^0)^2\right]\geq0.
\end{equation}
Applying this inequality with $v^n=U_j^n$ and summing over $j$ gives
\begin{equation}
	D_\tau^\alpha\|U^n\|_{2,h}^2\leq2h^d\sum_{j\in\mathbb{Z}^d}U_j^nD_\tau^\alpha U_j^n\leq0.
\end{equation}

Rearranging the L1 operator, we, therefore, obtain
\begin{equation}
	\|U^n\|_{2,h}^2\leq\sum_{k=0}^{n-2}(a_k-a_{k+1})\|U^{n-1-k}\|_{2,h}^2+a_{n-1}\|U^0\|_{2,h}^2.
\end{equation}
We now argue by induction. For $n=1$,
\begin{equation}
	\|U^1\|_{2,h}^2\leq\|U^0\|_{2,h}^2.
\end{equation}
Assume that $\|U^r\|_{2,h}^2\leq\|U^0\|_{2,h}^2$ for $r=1,\ldots,n-1$. Then
\begin{equation}
	\|U^n\|_{2,h}^2\leq\left(\sum_{k=0}^{n-2}(a_k-a_{k+1})+a_{n-1}\right)\|U^0\|_{2,h}^2.
\end{equation}
Since
\begin{equation}
	\sum_{k=0}^{n-2}(a_k-a_{k+1})+a_{n-1}=1,
\end{equation}
we conclude that
\begin{equation}
	\|U^n\|_{2,h}^2\leq\|U^0\|_{2,h}^2.
\end{equation}
Taking the square root completes the proof.
\end{proof}

We next establish the existence of a non-negative solution to the implicit scheme.
\begin{theorem}[Existence of the discrete solution]
\label{thm:discrete_existence}

Let $U^0\in\ell^1(\mathbb{Z}^d)$ be non-negative. Assume that, for each fixed $h>0$, the weights $g_\gamma^i(h)$ are bounded and satisfy $g_\gamma^i(h)\rightarrow0$ as $|\gamma|\rightarrow\infty$. Then, for every $n=1,\ldots,N$, there exists a non-negative solution $U^n\in\ell^1(\mathbb{Z}^d)$ to \eqref{eq:discrete}.
\end{theorem}
\begin{proof}
We proceed by induction in time. Suppose that $U^0,\ldots,U^{n-1}\in\ell^1(\mathbb{Z}^d)$ are non-negative. Define
\begin{equation}
	H_j^n=\sum_{k=1}^{n-1}(a_{k-1}-a_k)U_j^{n-k}+a_{n-1}U_j^0.
\end{equation}
Since the coefficients $a_{k-1}-a_k$ and $a_{n-1}$ are non-negative, we have $H^n\geq0$. Moreover, $H^n\in\ell^1(\mathbb{Z}^d)$.  Using $\lambda_\alpha=\Gamma(2-\alpha)\tau^\alpha$, the equation at time $t_n$ can be written as the nonlinear resolvent problem as follows
\begin{equation}
	U_j+\lambda_\alpha(\mathcal{N}_hU)_j=H_j^n, \quad j\in\mathbb{Z}^d.
\end{equation}
Thus, it remains to prove that for every non-negative $H\in\ell^1(\mathbb{Z}^d)$ the equation
\begin{equation}
	U+\lambda_\alpha\mathcal{N}_hU=H
\end{equation}
admits a non-negative solution $U\in\ell^1(\mathbb{Z}^d)$. For $R\in\mathbb{N}$, let
\begin{equation}
	\Lambda_R=\{j\in\mathbb{Z}^d:|j|_\infty\leq R\}.
\end{equation}
Define the truncated right-hand side by
\begin{equation}
	H_j^R=H_j\mathbf{1}_{\Lambda_R}(j).
\end{equation}
For a vector $U=\{U_j\}_{j\in\Lambda_R}$, let $\widetilde U$ denote its extension by zero outside $\Lambda_R$.

Now, we define $\mathcal{N}_{h,R}$ using the discrete pressure gradient of $\widetilde U$ at the interior interfaces and imposing a zero numerical flux through the boundary of $\Lambda_R$. More precisely, all fluxes connecting two nodes of $\Lambda_R$ are given by the same upwind formula as for $\mathcal{N}_h$, while fluxes crossing the boundary of $\Lambda_R$ are set equal to zero. Consequently,
\begin{equation}
	\sum_{j\in\Lambda_R}(\mathcal{N}_{h,R}U)_j=0.
\end{equation}
We first solve the finite-dimensional problem
\begin{equation}
	U_j^R+\lambda_\alpha(\mathcal{N}_{h,R}U^R)_j=H_j^R, \quad j\in\Lambda_R.
	\label{eq::trunceq}
\end{equation}
Since $\Lambda_R$ is finite, the discrete pressure gradient of $\widetilde U$ at every relevant interface is a finite sum. Hence $\mathcal{N}_{h,R}$ is continuous on the non-negative cone of $\mathbb{R}^{M_R}$, where $M_R=|\Lambda_R|$. Let
\begin{equation}
	S_R=\sum_{j\in\Lambda_R}H_j^R.
\end{equation}
If $S_R=0$, then $H^R=0$ and $U^R=0$ is a solution. We therefore assume $S_R>0$ and introduce the simplex
\begin{equation}
	K_R=\left\{U\in\mathbb{R}^{M_R}:U_j\geq0,\ \sum_{j\in\Lambda_R}U_j=S_R\right\}.
\end{equation}
Its relative interior in the affine hyperplane
\begin{equation}
	X_R=\left\{U\in\mathbb{R}^{M_R}:\sum_{j\in\Lambda_R}U_j=S_R\right\}
\end{equation}
is
\begin{equation}
	K_R^{\mathrm{int}}=\left\{U\in X_R:U_j>0\ \text{for all }j\in\Lambda_R\right\}.
\end{equation}
To avoid possible zeros of $H^R$ on the boundary of the simplex, define
\begin{equation}
	\overline H_j^R=\frac{S_R}{M_R}
\end{equation}
and, for $0<\varepsilon<1$,
\begin{equation}
	H^{R,\varepsilon}=(1-\varepsilon)H^R+\varepsilon\overline H^R.
\end{equation}
Then $H_j^{R,\varepsilon}>0$ for every $j\in\Lambda_R$ and
\begin{equation}
	\sum_{j\in\Lambda_R}H_j^{R,\varepsilon}=S_R.
\end{equation}
In particular, $H^{R,\varepsilon}\in K_R^{\mathrm{int}}$.
For $s\in[0,1]$, define
\begin{equation}
	\Phi_s(U)=U-H^{R,\varepsilon}+s\lambda_\alpha\mathcal{N}_{h,R}(U).
\end{equation}
Since $\sum_jU_j=S_R$, $\sum_jH_j^{R,\varepsilon}=S_R$ and $\sum_j(\mathcal{N}_{h,R}U)_j=0$, we have
\begin{equation}
	\sum_{j\in\Lambda_R}(\Phi_s(U))_j=0.
\end{equation}
Thus, $\Phi_s$ takes values in the following set
\begin{equation}
	Y_R=\left\{V\in\mathbb{R}^{M_R}:\sum_{j\in\Lambda_R}V_j=0\right\}.
\end{equation}
We claim that $\Phi_s$ has no zero on $\partial K_R$. Indeed, if $U\in\partial K_R$, then $U_{j_0}=0$ for some $j_0\in\Lambda_R$. By Lemma~\ref{lem:zero_minimum_sign}, applied to the truncated operator,
\begin{equation}
	(\mathcal{N}_{h,R}U)_{j_0}\leq0.
\end{equation}
Therefore,
\begin{equation}
	(\Phi_s(U))_{j_0}=-H_{j_0}^{R,\varepsilon}+s\lambda_\alpha(\mathcal{N}_{h,R}U)_{j_0}<0.
\end{equation}
Hence $\Phi_s(U)\neq0$ for every $U\in\partial K_R$ and every $s\in[0,1]$. At $s=0$, equation $\Phi_0(U)=0$ has the unique solution $U=H^{R,\varepsilon}\in K_R^{\mathrm{int}}$. After identifying the affine space $X_R$ with the vector space $Y_R$ by translation, the corresponding Brouwer degree is
\begin{equation}
	\deg(\Phi_0,K_R^{\mathrm{int}},0)=1.
\end{equation}
By homotopy invariance of the degree,
\begin{equation}
	\deg(\Phi_1,K_R^{\mathrm{int}},0)=1.
\end{equation}
Therefore, there exists $U^{R,\varepsilon}\in K_R^{\mathrm{int}}$ satisfying
\begin{equation}
	U^{R,\varepsilon}+\lambda_\alpha\mathcal{N}_{h,R}(U^{R,\varepsilon})=H^{R,\varepsilon}.
\end{equation}
Since $K_R$ is compact, there exist a sequence $\varepsilon_\ell\rightarrow0$ and a vector $U^R\in K_R$ such that $U^{R,\varepsilon_\ell}\rightarrow U^R$. By the continuity of $\mathcal{N}_{h,R}$, passing to the limit gives
\begin{equation}
	U^R+\lambda_\alpha\mathcal{N}_{h,R}(U^R)=H^R,
\end{equation}
and, moreover,
\begin{equation}
	U_j^R\geq0, \quad \sum_{j\in\Lambda_R}U_j^R=S_R.
\end{equation}
We now pass to the limit $R\rightarrow\infty$. Set
\begin{equation}
	S=\sum_{j\in\mathbb{Z}^d}H_j.
\end{equation}
Since $H\geq0$, we have $S_R\leq S$ and extending $U^R$ by zero outside $\Lambda_R$, we obtain $0\leq U_j^R\leq S$ for every $j$ and $R$. Hence, by a diagonal argument, there exist a sequence $R_\ell\rightarrow\infty$ and a non-negative sequence $U=\{U_j\}_{j\in\mathbb{Z}^d}$ such that
\begin{equation}
	U_j^{R_\ell}\rightarrow U_j
\end{equation}
for every fixed $j\in\mathbb{Z}^d$. Furthermore, for every finite set $K\subset\mathbb{Z}^d$, we have
\begin{equation}
	\sum_{j\in K}U_j=\lim_{\ell\rightarrow\infty}\sum_{j\in K}U_j^{R_\ell}\leq S.
\end{equation}
Taking the supremum over all finite subsets $K$ gives
\begin{equation}
	\sum_{j\in\mathbb{Z}^d}U_j\leq S.
\end{equation}
Thus $U\in\ell^1(\mathbb{Z}^d)$. 

It remains to pass to the limit in the nonlocal operator. Fix $i\in\{1,\ldots,d\}$ and an interface $x_{j+\frac{1}{2}e_i}$. Then
\begin{equation}
	(G_{h,i}U^{R_\ell})_{j+\frac{1}{2}e_i}-(G_{h,i}U)_{j+\frac{1}{2}e_i}=\sum_{k\in\mathbb{Z}^d}g^i_{k-j-\frac{1}{2}e_i}(h)(U_k^{R_\ell}-U_k).
\end{equation}
Let $L>0$. In the finite set $|k|\leq L$, the pointwise convergence gives
\begin{equation}
	\sum_{|k|\leq L}g^i_{k-j-\frac{1}{2}e_i}(h)(U_k^{R_\ell}-U_k)\rightarrow0
\end{equation}
as $\ell\rightarrow\infty$. For the tails,
\begin{equation}
	\left|\sum_{|k|>L}g^i_{k-j-\frac{1}{2}e_i}(h)U_k^{R_\ell}\right|\leq S\sup_{|k|>L}\left|g^i_{k-j-\frac{1}{2}e_i}(h)\right|,
\end{equation}
and similarly
\begin{equation}
	\left|\sum_{|k|>L}g^i_{k-j-\frac{1}{2}e_i}(h)U_k\right|\leq S\sup_{|k|>L}\left|g^i_{k-j-\frac{1}{2}e_i}(h)\right|.
\end{equation}
Since $g_\gamma^i(h)\rightarrow0$ as $|\gamma|\rightarrow\infty$, first letting $\ell\rightarrow\infty$ and then $L\rightarrow\infty$ yields
\begin{equation}
	(G_{h,i}U^{R_\ell})_{j+\frac{1}{2}e_i}\rightarrow(G_{h,i}U)_{j+\frac{1}{2}e_i}.
\end{equation}
Fix now $j\in\mathbb{Z}^d$ and notice that for sufficiently large $\ell$, the nodes $j$ and $j\pm e_i$ lie in the interior of $\Lambda_{R_\ell}$. We have $m>1$ and hence, pointwise convergence implies $(U_j^{R_\ell})^{m-1}\rightarrow U_j^{m-1}$. Together with the convergence of the discrete pressure gradient and the continuity of the positive and negative parts, this yields
\begin{equation}
	F_{j\pm\frac{1}{2}e_i}^{(i)}(U^{R_\ell})\rightarrow F_{j\pm\frac{1}{2}e_i}^{(i)}(U).
\end{equation}
Note that it then follows that
\begin{equation}
	(\mathcal{N}_{h,R_\ell}U^{R_\ell})_j\rightarrow(\mathcal{N}_hU)_j.
\end{equation}
For every fixed $j$, we also have $H_j^{R_\ell}=H_j$ for all sufficiently large $\ell$. Passing to the limit in the truncated resolvent equation therefore yields
\begin{equation}
	U_j+\lambda_\alpha(\mathcal{N}_hU)_j=H_j.
\end{equation}
Thus, the resolvent problem has a non-negative solution $U\in\ell^1(\mathbb{Z}^d)$.

Applying this argument successively with $H=H^n$ proves the existence of a non-negative $\ell^1$ solution $U^n$ at every time level $n=1,\ldots,N$.
\end{proof}

We next derive an energy-dissipation estimate for the implicit scheme. It can be viewed as a stronger version of the $\ell^2$ stability estimate. To start, for the grid functions $W=\{W_j\}_{j\in\mathbb{Z}^d}$ and $Z=\{Z_j\}_{j\in\mathbb{Z}^d}$, we introduce the inner product
\begin{equation}
	\langle W,Z\rangle_h=h^d\sum_{j\in\mathbb{Z}^d}W_jZ_j
\end{equation}
and the corresponding norm
\begin{equation}
	\|W\|_{2,h}^2=\langle W,W\rangle_h=h^d\sum_{j\in\mathbb{Z}^d}|W_j|^2.
\end{equation}
Next, for a non-negative grid function $U$, define the discrete dissipation
\begin{equation}
	\mathcal{D}_h(U)=\frac{h^d}{2}\sum_{j\in\mathbb{Z}^d}\sum_{\beta\neq0}C_\beta(h)\left(U_j^{\frac{m+1}{2}}-U_{j+\beta}^{\frac{m+1}{2}}\right)^2.
\end{equation}
We are now ready to proceed with the dissipation lemma. 

\begin{lemma}[Discrete energy-dissipation estimate]
\label{lem:energy_dissipation}
Let $U^n$ be a bounded non-negative solution to \eqref{eq:discrete}. Assume that all sums appearing below are absolutely convergent. Then
\begin{equation}
	\frac{1}{2}D_\tau^\alpha\|U^n\|_{2,h}^2+\frac{4}{(m+1)^2}\mathcal{D}_h(U^n)\leq0.
\end{equation}
In particular,
\begin{equation}
	D_\tau^\alpha\|U^n\|_{2,h}^2\leq0,
\end{equation}
hence, the $\ell^2$ stability estimate of Lemma~\ref{lem:l2_stability} follows.
\end{lemma}
\begin{proof}
We start by multiplying \eqref{eq:discrete} by $h^dU_j^n$ and summing over $j\in\mathbb{Z}^d$. It follows that
\begin{equation}
	\left\langle U^n,D_\tau^\alpha U^n\right\rangle_h=-\left\langle U^n,\mathcal{N}_hU^n\right\rangle_h.
\end{equation}
Furthermore, by the discrete convexity inequality for the L1 operator, we have
\begin{equation}
	\frac{1}{2}D_\tau^\alpha\|U^n\|_{2,h}^2\leq\left\langle U^n,D_\tau^\alpha U^n\right\rangle_h,
\end{equation}
and, therefore,
\begin{equation}
	\frac{1}{2}D_\tau^\alpha\|U^n\|_{2,h}^2\leq-\left\langle U^n,\mathcal{N}_hU^n\right\rangle_h.
\end{equation}
Next, proceeding as in the proof of Lemma~\ref{lem:Nh_positivity}, with $\Phi(s)=s^m/m$, we have
\begin{equation}
	\left\langle U^n,\mathcal{N}_hU^n\right\rangle_h\geq\left\langle \Phi(U^n),L_hU^n\right\rangle_h,
\end{equation}
and by the symmetric representation of $L_h$, we obtain
\begin{equation}
	\left\langle \Phi(U^n),L_hU^n\right\rangle_h=\frac{h^d}{2}\sum_{j\in\mathbb{Z}^d}\sum_{\beta\neq0}C_\beta(h)\left(\Phi(U_j^n)-\Phi(U_{j+\beta}^n)\right)\left(U_j^n-U_{j+\beta}^n\right).
\end{equation}
Now, recall the two-point Stroock–Varopoulos inequality \cite{Stroock1984LargeDeviations, Varopoulos1985HardyLittlewood}, that is for any $a,b\geq0$
\begin{equation}
	(a^m-b^m)(a-b)\geq\frac{4m}{(m+1)^2}\left(a^{\frac{m+1}{2}}-b^{\frac{m+1}{2}}\right)^2,
\end{equation}
which can be rewritten in terms of $\Phi(s)=s^m/m$ as
\begin{equation}
	\left(\Phi(a)-\Phi(b)\right)(a-b)\geq\frac{4}{(m+1)^2}\left(a^{\frac{m+1}{2}}-b^{\frac{m+1}{2}}\right)^2.
\end{equation}
Consequently,
\begin{equation}
	\left\langle \Phi(U^n),L_hU^n\right\rangle_h\geq\frac{4}{(m+1)^2}\mathcal{D}_h(U^n),
\end{equation}
 combining it with the preceding estimates yields
\begin{equation}
	\frac{1}{2}D_\tau^\alpha\|U^n\|_{2,h}^2+\frac{4}{(m+1)^2}\mathcal{D}_h(U^n)\leq0,
\end{equation}
which completes the proof.
\end{proof}

\begin{remark}
Note that the quantity $\mathcal{D}_h(U)$ is the discrete analogue of the fractional energy associated with $U^{(m+1)/2}$. In particular, when the coefficients $C_\beta(h)$ are chosen to approximate the kernel of $(-\Delta)^{1-\sigma}$, the dissipation \textit{formally} corresponds to the $H^{1-\sigma}$ seminorm of $U^{(m+1)/2}$.
\end{remark}

\subsection{Consistency of the implicit scheme}

We now turn to the convergence analysis of the implicit scheme. In the sequel, we assume that the main problem \eqref{fpe} has a non-negative solution $u$ on $[0,T]$ that is sufficiently smooth in space and regular in time for $t>0$. The temporal assumptions stated below allow a weak singularity at $t=0$. Here, we are only concerned with the properties of the numerical scheme. Questions about existence and uniqueness of solutions to the continuous problem are not considered here and are left for future work. 

We start by analyzing the consistency of the scheme. For notation, define
\begin{equation}
	u_j^n=u(x_j,t_n), \quad e_j^n=U_j^n-u_j^n, \quad n=0,\ldots,N,
\end{equation}
and also let
\begin{equation}
	\mathcal{N}u=-\nabla\cdot\left(u^{m-1}\nabla(-\Delta)^{-\sigma}u\right)
\end{equation}
denote the continuous spatial operator. Then, equation \eqref{fpe} can be written as
\begin{equation}
	\partial_t^\alpha u+\mathcal{N}u=0.
\end{equation}
In a standard way, we decompose the truncation error into its temporal and spatial parts. Define
\begin{equation}
	R_{\tau,j}^n=D_\tau^\alpha u_j^n-\partial_t^\alpha u(x_j,t_n)
\end{equation}
and
\begin{equation}
	R_{h,j}^n=(\mathcal{N}_hu^n)_j-(\mathcal{N}u)(x_j,t_n).
\end{equation}
Notice that, since $u$ satisfies \eqref{fpe}, we obtain the exact identity
\begin{equation}
	D_\tau^\alpha u_j^n+(\mathcal{N}_hu^n)_j=R_j^n,
\end{equation}
where the total remainder is
\begin{equation}
	R_j^n=R_{\tau,j}^n+R_{h,j}^n.
\end{equation}
Proceeding further, we assume that the temporal and spatial discretizations are consistent. For the temporal part, we use a weighted estimate. The convergence argument below will use its counterpart for the cumulative solution. That is, there exist functions $\eta_\tau=\eta_\tau(\tau)$ and $\eta_h=\eta_h(h)$ satisfying $\eta_\tau\rightarrow0$ as $\tau\rightarrow0$ and $\eta_h\rightarrow0$ as $h\rightarrow0$, such that
\begin{equation}
	t_n^\alpha\|R_\tau^n\|_{2,h}\leq C_u\eta_\tau, \quad \|R_h^n\|_{2,h}\leq C_u\eta_h
\end{equation}
for $n=1,\ldots,N$. In addition, we take the convention that $C_u$ may depend on $T$, $\alpha$, $m$, $\sigma$ and sufficiently high norms of the exact solution, but is independent of $h$ and $\tau$. Therefore,
\begin{equation}\label{eq:consistency_error}
	\|R^n\|_{2,h}\leq C_u(\eta_\tau t_n^{-\alpha}+\eta_h).
\end{equation}
We now move to the derivation of a consistency estimate for the complete nonlinear spatial operator. This requires not only consistency of the discrete pressure gradient, but also consistency of one discrete derivative of its approximation. For a fixed time $t$, let
\begin{equation}
	A(x)=u(x,t)^{m-1}, \quad v_i(x)=\left(\nabla(-\Delta)^{-\sigma}u(x,t)\right)\cdot e_i.
\end{equation}
The continuous flux in the $i$-th direction is
\begin{equation}
	q_i(x)=-A(x)v_i(x),
\end{equation}
so that
\begin{equation}
	\mathcal{N}u=\sum_{i=1}^d\partial_{x_i}q_i.
\end{equation}
Also, for a grid function defined at the interfaces, we introduce a finite difference and a norm
\begin{equation}
	(\delta_i Z)_j=\frac{Z_{j+\frac{1}{2}e_i}-Z_{j-\frac{1}{2}e_i}}{h}, \quad \|Z\|_{2,h,i}^2=h^d\sum_{j\in\mathbb{Z}^d}\left|Z_{j+\frac{1}{2}e_i}\right|^2.
\end{equation}

\begin{proposition}[Consistency of the nonlinear spatial operator]
\label{prop:spatial_consistency}
We assume that the exact solution is sufficiently smooth and decays sufficiently fast at infinity. In particular, assume that $A=u^{m-1}$, $v_i$ and $q_i=-Av_i$ have all spatial derivatives required below uniformly bounded on $[0,T]$ in the corresponding $L^\infty$ and $L^2$ norms. Now, let
\begin{equation}
	V_{h,j+\frac{1}{2}e_i}^i=(G_{h,i}u)_{j+\frac{1}{2}e_i}
\end{equation}
and define the pressure-gradient error
\begin{equation}
	E_{j+\frac{1}{2}e_i}^i=V_{h,j+\frac{1}{2}e_i}^i-v_i(x_{j+\frac{1}{2}e_i}).
\end{equation}
Assume that there exists $\eta_G(h)\rightarrow0$ as $h\rightarrow0$ such that
\begin{equation}
	\|E^i\|_{2,h,i}+\|\delta_iE^i\|_{2,h}\leq C_u\eta_G(h),
	\label{eq:gradient_consistency}
\end{equation}
for $i=1,\ldots,d$, and that
\begin{equation}
	\|V_h^i\|_\infty\leq C_u.
\end{equation}
Then
\begin{equation}
	\|\mathcal{N}_hu-\mathcal{N}u\|_{2,h}\leq C_u\left(h+\eta_G(h)\right).
	\label{eq:full_spatial_consistency}
\end{equation}
\end{proposition}

\begin{proof}
Start by fixing $i\in\{1,\ldots,d\}$ and putting
\begin{equation}
	A_j=A(x_j), \quad \overline{A}_{j+\frac{1}{2}e_i}=\frac{A_{j+e_i}+A_j}{2}, \quad D_i^+A_j=\frac{A_{j+e_i}-A_j}{h}.
\end{equation}
By our notation, $r^+=(r+|r|)/2$ and $r^-=(r-|r|)/2$, the upwind flux can be written compactly as
\begin{equation}
	F_{j+\frac{1}{2}e_i}^{(i)}(u)=-V_{h,j+\frac{1}{2}e_i}^i\overline{A}_{j+\frac{1}{2}e_i}-\frac{h}{2}\left|V_{h,j+\frac{1}{2}e_i}^i\right|D_i^+A_j.
	\label{eq:upwind_flux_decomposition}
\end{equation}
Next, let $x_{j+\frac{1}{2}e_i}=x_j+\frac{h}{2}e_i$ and write
\begin{equation}
	A_{j+\frac{1}{2}e_i}=A(x_{j+\frac{1}{2}e_i}), \quad v_{j+\frac{1}{2}e_i}^i=v_i(x_{j+\frac{1}{2}e_i}).
\end{equation}
Notice that since
\begin{equation}
	q_i(x_{j+\frac{1}{2}e_i})=-v_{j+\frac{1}{2}e_i}^iA_{j+\frac{1}{2}e_i},
\end{equation}
we have from \eqref{eq:upwind_flux_decomposition} the following
\begin{equation}
	F_{j+\frac{1}{2}e_i}^{(i)}(u)-q_i(x_{j+\frac{1}{2}e_i})=-E_{j+\frac{1}{2}e_i}^i\overline{A}_{j+\frac{1}{2}e_i}-v_{j+\frac{1}{2}e_i}^i\left(\overline{A}_{j+\frac{1}{2}e_i}-A_{j+\frac{1}{2}e_i}\right)-\frac{h}{2}\left|V_{h,j+\frac{1}{2}e_i}^i\right|D_i^+A_j.
	\label{eq:flux_consistency_decomposition}
\end{equation}
Now, we estimate each of the three left-hand side terms separately. For the first term, the discrete product rule and the smoothness of $A$ give
\begin{equation}
	\left\|\delta_i\left(E^i\overline{A}\right)\right\|_{2,h}\leq C_u\left(\|\delta_iE^i\|_{2,h}+\|E^i\|_{2,h,i}\right)\leq C_u\eta_G(h).
\end{equation}
For the second term, the midpoint symmetry gives
\begin{equation}
	\overline{A}_{j+\frac{1}{2}e_i}-A_{j+\frac{1}{2}e_i}=O(h^2).
\end{equation}
Also, the same Taylor expansion applied after taking one discrete difference yields
\begin{equation}
	\left\|\delta_i\left[v_i\left(\overline{A}-A_{\mathrm{mid}}\right)\right]\right\|_{2,h}\leq C_uh^2.
\end{equation}
Next, for the upwind contribution, we have
\begin{equation}
	\frac{h}{2}\left\|\delta_i\left(|V_h^i|D_i^+A\right)\right\|_{2,h}\leq C_uh.
\end{equation}
To see this, note that the function $r\mapsto|r|$ is Lipschitz, and hence
\begin{equation}
	\left|\delta_i|V_h^i|\right|\leq|\delta_iV_h^i|.
\end{equation}
Moreover,
\begin{equation}
	\delta_iV_h^i=\delta_iv_i+\delta_iE^i,
\end{equation}
hence, the required bound follows from the regularity of $v_i$, \eqref{eq:gradient_consistency}, the uniform bound on $V_h^i$ and the smoothness of $A$. Combining these estimates with \eqref{eq:flux_consistency_decomposition}, we obtain
\begin{equation}
	\left\|\delta_iF^{(i)}(u)-\delta_iq_i\right\|_{2,h}\leq C_u\left(h+\eta_G(h)\right).
\end{equation}
It only remains to compare the centered discrete divergence of the exact flux with its continuous divergence. By Taylor expansion, we can write
\begin{equation}
	\left\|\frac{q_i(x_{j+\frac{1}{2}e_i})-q_i(x_{j-\frac{1}{2}e_i})}{h}-\partial_{x_i}q_i(x_j)\right\|_{2,h}\leq C_uh^2.
\end{equation}
Next, summing over $i=1,\ldots,d$ gives
\begin{equation}
	\|\mathcal{N}_hu-\mathcal{N}u\|_{2,h}\leq C_u\left(h+\eta_G(h)+h^2\right).
\end{equation}
Since $0<h\leq1$, this proves \eqref{eq:full_spatial_consistency}.
\end{proof}

\begin{remark}
If the pressure-gradient discretization satisfies
\begin{equation}
	\eta_G(h)=O(h^r)
\end{equation}
together with the corresponding estimate for one discrete derivative, then Proposition~\ref{prop:spatial_consistency} gives
\begin{equation}
	\|\mathcal{N}_hu-\mathcal{N}u\|_{2,h}=O\left(h^{\min\{1,r\}}\right).
\end{equation}
Thus, whenever the pressure-gradient approximation has order at least one in the sense of \eqref{eq:gradient_consistency}, the first-order upwind discretization determines the resulting consistency order of the complete nonlinear spatial operator.
\end{remark}

\subsection{Convergence of the one-dimensional scheme in integrated formulation}

Now, we move to the convergence analysis, but restrict it to the one-dimensional case, that is, we put $d=1$. Note that all the structural properties and existence results obtained above remain valid in arbitrary dimension. Next, we will work with integrated problem associated with \eqref{fpe} in $d=1$. This approach was studied analytically for $\alpha=1$ in \cite{chasseigne2017nonlocal, Stan2016FiniteInfinite} and was used for the derivation and analysis of the numerical scheme in \cite{delTesoJakobsen2026Convergent}. Let
\begin{equation}
	v(x,t)=\int_{-\infty}^{x}u(y,t)\,dy.
\end{equation}
Because $u$ is non-negative, $v$ is non-decreasing in $x$ and $\partial_xv=u$. Assuming sufficient regularity and decay at infinity, we may integrate the continuous equation over $(-\infty,x)$. Moreover,
\begin{equation}
	\partial_x(-\Delta)^{-\sigma}u=\partial_x(-\Delta)^{-\sigma}\partial_xv=-(-\Delta)^{1-\sigma}v.
\end{equation}
Hence,
\begin{equation}
	\partial_t^\alpha v=-|\partial_xv|^{m-1}(-\Delta)^{1-\sigma}v.
\end{equation}
Setting $s=1-\sigma$, the integrated problem therefore takes the form
\begin{equation}\label{eq:integrated_problem}
	\partial_t^\alpha v+|\partial_xv|^{m-1}(-\Delta)^sv=0.
\end{equation}
Next, we show that the conservative scheme \eqref{eq:discrete} admits a corresponding discrete integrated formulation. Define
\begin{equation}
	V_j^n=h\sum_{k=-\infty}^{j}U_k^n.
\end{equation}
We associate $V_j^n$ with the interface point $x_{j+\frac{1}{2}}$. Then
\begin{equation}
	U_j^n=D_h^-V_j^n, \quad U_{j+1}^n=D_h^+V_j^n,
\end{equation}
where
\begin{equation}
	D_h^-V_j=\frac{V_j-V_{j-1}}{h}, \quad D_h^+V_j=\frac{V_{j+1}-V_j}{h}.
\end{equation}
Also, recall that
\begin{equation}
	c_\beta(h)=g_{\beta-\frac{1}{2}}(h)-g_{\beta+\frac{1}{2}}(h).
\end{equation}
Using $U_k=(V_k-V_{k-1})/h$, we obtain
\begin{equation}
	(G_hU)_{j+\frac{1}{2}}=\frac{1}{h}\sum_{\beta\in\mathbb{Z}}c_\beta(h)V_{j+\beta}.
\end{equation}
Since $\sum_{\beta\in\mathbb{Z}}c_\beta(h)=0$, this can be rewritten as
\begin{equation}
	-(G_hU)_{j+\frac{1}{2}}=\sum_{\beta\neq0}C_\beta(h)(V_j-V_{j+\beta}),
\end{equation}
where
\begin{equation}
	C_\beta(h)=\frac{c_\beta(h)}{h}\geq0, \quad \beta\neq0.
\end{equation}
Hence, we define
\begin{equation}\label{eq:discrete_fractional_integrated}
	(\mathcal{A}_hV)_j=\sum_{\beta\neq0}C_\beta(h)(V_j-V_{j+\beta}).
\end{equation}
For well-posedness, we assume that $\sum_{\beta\neq0}C_\beta(h)<\infty$ for every fixed $h>0$, so that $\mathcal{A}_hV$ is well-defined for bounded grid functions. Also, symmetry $c_{-\beta}(h)=c_\beta(h)$ implies $C_{-\beta}(h)=C_\beta(h)$, and therefore $\mathcal{A}_h$ is a discrete non-negative symmetric operator.

For the one-dimensional convergence analysis we use the semigroup weights constructed in the appendix. The following result is an immediate consequence of Proposition~\ref{prop:semigroup_fractional_laplacian} and Corollary~\ref{cor:semigroup_consistency}.
\begin{lemma}[Identification and consistency of $\mathcal{A}_h$]
\label{lem:Ah_consistency}
Let $d=1$, $s=1-\sigma$, and assume that $G_h$ is defined using the semigroup weights of Appendix~A. Then
\begin{equation}
	\mathcal{A}_h=(-\Delta_h)^s.
\end{equation}
Moreover, for bounded sufficiently regular $v$, we have
\begin{equation}
	\|\mathcal{A}_hv-(-\Delta)^sv\|_\infty\leq C_v\eta_{\mathcal A}(h),
\end{equation}
where
\begin{equation}
	\eta_{\mathcal A}(h)=
	\begin{cases}
		h^{2\sigma}, & 0<\sigma<\frac12,\\
		h^\theta, \quad 0<\theta<1, & \sigma=\frac12,\\
		h^{2\sigma-1}, & \frac12<\sigma<1.
	\end{cases}
\end{equation}
The same estimate holds on the shifted grid $h(\mathbb{Z}+\frac12)$.
\end{lemma}
\begin{proof}
Note that for $\beta\geq1$, the semigroup weights are given by \eqref{eq:semigroup_weights_1d} and hence
\begin{equation}
	C_\beta(h)=\frac{g_{\beta-\frac12}(h)-g_{\beta+\frac12}(h)}{h}.
\end{equation}
Using the identities $\Gamma(z+1)=z\Gamma(z)$ and $s=1-\sigma$, we further obtain
\begin{equation}
	C_\beta(h)=\frac{4^s\Gamma(\frac12+s)}{\sqrt{\pi}|\Gamma(-s)|}\frac{1}{h^{2s}}\frac{\Gamma(\beta-s)}{\Gamma(\beta+1+s)}.
\end{equation}
By symmetry, the same formula holds with $\beta$ replaced by $|\beta|$. Therefore, we have
\begin{equation}
	(\mathcal{A}_hV)_j=\frac{1}{h^{2s}}\sum_{\beta\neq0}K_s(\beta)(V_j-V_{j+\beta}),
\end{equation}
where
\begin{equation}
	K_s(\beta)=\frac{4^s\Gamma(\frac12+s)\Gamma(|\beta|-s)}{\sqrt{\pi}|\Gamma(-s)|\Gamma(|\beta|+1+s)}.
\end{equation}
This is exactly the fractional discrete Laplacian $(-\Delta_h)^s$, and hence
\begin{equation}
	\mathcal{A}_hV=(-\Delta_h)^sV.
\end{equation}
The consistency estimate now follows from Theorem 1.1 of \cite{ciaurri2015fractional}. Indeed, since $s=1-\sigma$, the estimates in that theorem give
\begin{equation}
	\|(-\Delta_h)^sv-(-\Delta)^sv\|_\infty\leq C_v\eta_{\mathcal A}(h)
\end{equation}
with $\eta_{\mathcal A}(h)$ as stated above.

Finally, note that although the theorem is formulated on the lattice $h\mathbb{Z}$, the same estimate holds on $h(\mathbb{Z}+\frac12)$. Indeed, both $(-\Delta)^s$ and $(-\Delta_h)^s$ are translation invariant, and translation does not change the corresponding H\"older norms. This completes the proof.
\end{proof}
Now, the numerical flux can be written as
\begin{equation}
	F_{j+\frac{1}{2}}(U)=(D_h^-V_j)^{m-1}(\mathcal{A}_hV_j)^++(D_h^+V_j)^{m-1}(\mathcal{A}_hV_j)^-.
\end{equation}
By summing the conservative scheme over $k\leq j$ and using the vanishing of the flux at $-\infty$, we can obtain
\begin{equation}\label{eq:integrated_discrete_scheme}
	D_\tau^\alpha V_j^n=\mathcal{L}_h[V^n]_j,
\end{equation}
where
\begin{equation}\label{eq:integrated_spatial_operator}
	\mathcal{L}_h[V]_j=-(D_h^-V_j)^{m-1}(\mathcal{A}_hV_j)^+-(D_h^+V_j)^{m-1}(\mathcal{A}_hV_j)^-.
\end{equation}
Since $V$ is non-decreasing, both $D_h^-V_j$ and $D_h^+V_j$ are non-negative. Thus, \eqref{eq:integrated_spatial_operator} is an upwind approximation of the operator in \eqref{eq:integrated_problem}. The integrated equation leads to the pleasant property that the discrete spatial operator is monotone. 

\begin{lemma}[Monotonicity of the integrated spatial operator]
\label{lem:integrated_monotonicity}
Let $\phi$ and $\psi$ be bounded non-decreasing grid functions such that $\phi_k\leq\psi_k$ for all $k\in\mathbb{Z}$ and $\phi_j=\psi_j$ for some $j\in\mathbb{Z}$. Then
\begin{equation}
	\mathcal{L}_h[\phi]_j\leq\mathcal{L}_h[\psi]_j.
\end{equation}
\end{lemma}
\begin{proof}
Notice that since $C_\beta(h)\geq0$ and $\phi_j=\psi_j$, we have
\begin{equation}
	(\mathcal{A}_h\phi)_j\geq(\mathcal{A}_h\psi)_j,
\end{equation}
and, moreover,
\begin{equation}
	D_h^-\phi_j\geq D_h^-\psi_j, \quad D_h^+\phi_j\leq D_h^+\psi_j.
\end{equation}
Now, suppose first that $(\mathcal{A}_h\phi)_j\geq(\mathcal{A}_h\psi)_j\geq0$. Then
\begin{equation}
	-(D_h^-\phi_j)^{m-1}(\mathcal{A}_h\phi)_j\leq-(D_h^-\psi_j)^{m-1}(\mathcal{A}_h\psi)_j.
\end{equation}
But if instead $0\geq(\mathcal{A}_h\phi)_j\geq(\mathcal{A}_h\psi)_j$, then
\begin{equation}
	-(D_h^+\phi_j)^{m-1}(\mathcal{A}_h\phi)_j\leq-(D_h^+\psi_j)^{m-1}(\mathcal{A}_h\psi)_j.
\end{equation}
Finally, if $(\mathcal{A}_h\phi)_j\geq0\geq(\mathcal{A}_h\psi)_j$, then
\begin{equation}
	\mathcal{L}_h[\phi]_j\leq0\leq\mathcal{L}_h[\psi]_j.
\end{equation}
All three cases lead to our claim, and the proof is complete. 
\end{proof}
Next, notice that the operator $\mathcal{L}_h$ is invariant under addition of constants, that is,
\begin{equation}
	\mathcal{L}_h[V+c]=\mathcal{L}_h[V].
\end{equation}
Also, put $\lambda_\alpha=\Gamma(2-\alpha)\tau^\alpha$ and note that L1 discretization allows us to rewrite \eqref{eq:integrated_discrete_scheme} as
\begin{equation}\label{eq:integrated_resolvent}
	V_j^n-\lambda_\alpha\mathcal{L}_h[V^n]_j=H_j^n,
\end{equation}
where
\begin{equation}
H_j^n=\sum_{k=0}^{n-2}(a_k-a_{k+1})V_j^{n-1-k}+a_{n-1}V_j^0.
\end{equation}
The following result shows the relevant comparison principle.
\begin{lemma}[Comparison and stability of the discrete resolvent]
\label{lem:integrated_resolvent_stability}
Let $V$ and $W$ be bounded non-decreasing solutions of
\begin{equation}
	V-\lambda_\alpha\mathcal{L}_h[V]=H, \quad W-\lambda_\alpha\mathcal{L}_h[W]=K.
\end{equation}
If $H_j\leq K_j$ for all $j$, then $V_j\leq W_j$ for all $j$. Consequently,
\begin{equation}
	\|V-W\|_\infty\leq\|H-K\|_\infty.
	\label{eq:resolvent_contraction}
\end{equation}
\end{lemma}

\begin{proof}
Notice that, since $V$ and $W$ are bounded and non-decreasing, the limits
\begin{equation}
	V_\pm=\lim_{j\rightarrow\pm\infty}V_j, \quad W_\pm=\lim_{j\rightarrow\pm\infty}W_j
\end{equation}
exist. Moreover, bounded monotonicity implies
\begin{equation}
	D_h^-V_j\rightarrow0, \quad D_h^+V_j\rightarrow0
\end{equation}
as $j\rightarrow\pm\infty$, and similarly for $W$. Because $\sum_{\beta\neq0}C_\beta(h)<\infty$, the quantities $\mathcal{A}_hV$ and $\mathcal{A}_hW$ are uniformly bounded. Therefore,
\begin{equation}
	\mathcal{L}_h[V]_j\rightarrow0, \quad \mathcal{L}_h[W]_j\rightarrow0
\end{equation}
as $j\rightarrow\pm\infty$. It follows from the resolvent equations that
\begin{equation}
	H_j-V_j\rightarrow0, \quad K_j-W_j\rightarrow0
\end{equation}
as $j\rightarrow\pm\infty$. Since $H_j\leq K_j$, we obtain
\begin{equation}
	V_-\leq W_-, \quad V_+\leq W_+.
\end{equation}
Now, suppose that
\begin{equation}
	M=\sup_{j\in\mathbb{Z}}(V_j-W_j)>0.
\end{equation}
Since the limits of $V_j-W_j$ at both ends are non-positive, the supremum is attained at some finite index $j_0$. Put $\widetilde W=W+M$. Then $V\leq\widetilde W$ and $V_{j_0}=\widetilde W_{j_0}$. By Lemma~\ref{lem:integrated_monotonicity} and translation invariance, we have
\begin{equation}
	\mathcal{L}_h[V]_{j_0}\leq\mathcal{L}_h[\widetilde W]_{j_0}=\mathcal{L}_h[W]_{j_0}.
\end{equation}
Using the two resolvent equations, we obtain
\begin{equation}
	M=H_{j_0}-K_{j_0}+\lambda_\alpha\left(\mathcal{L}_h[V]_{j_0}-\mathcal{L}_h[W]_{j_0}\right)\leq0,
\end{equation}
which is a contradiction. Hence $V\leq W$.

Finally, to prove \eqref{eq:resolvent_contraction}, let $\delta=\|H-K\|_\infty$. Then $H\leq K+\delta$. Since $W+\delta$ solves the resolvent equation with the right-hand side $K+\delta$, the comparison gives $V\leq W+\delta$. The interchanging of $V$ and $W$ yields $W\leq V+\delta$, which is the proof of the result.
\end{proof}
Now, we can proceed to the main result of this subsection. Since $V_j^n$ is associated with the interface $x_{j+\frac{1}{2}}$, we define
\begin{equation}
	v_j^n=v(x_{j+\frac{1}{2}},t_n).
\end{equation}
We first show the consistency of the integrated spatial operator.
\begin{lemma}[Consistency of the integrated spatial operator]
\label{lem:integrated_spatial_consistency}
Assume that $m\geq2$ and that $v$ is non-decreasing and sufficiently smooth in space, with the required spatial bounds uniform on $[0,T]$. Assume also that
\begin{equation}
	\|\mathcal{A}_h v^n-(-\Delta)^s v^n\|_\infty\leq C_v\eta_{\mathcal A}(h),
\end{equation}
where $\eta_{\mathcal A}(h)\rightarrow0$ as $h\rightarrow0$. Then
\begin{equation}
	\left\|\mathcal{L}_h[v^n]+|\partial_xv^n|^{m-1}(-\Delta)^sv^n\right\|_\infty\leq C_v\left(h+\eta_{\mathcal A}(h)\right).
\end{equation}
\end{lemma}
\begin{proof}
To start, let
\begin{equation}
	a_j^n=(-\Delta)^sv(x_{j+\frac{1}{2}},t_n), \quad a_{h,j}^n=(\mathcal{A}_hv^n)_j, \quad p_j^n=\partial_xv(x_{j+\frac{1}{2}},t_n).
\end{equation}
Next, by Taylor expansion,
\begin{equation}
	|D_h^-v_j^n-p_j^n|+|D_h^+v_j^n-p_j^n|\leq C_vh.
\end{equation}
Because $v$ is non-decreasing, all these quantities are non-negative. Observe that for $m\geq2$, the map $r\mapsto r^{m-1}$ is Lipschitz on bounded intervals, so therefore
\begin{equation}
	\left|(D_h^-v_j^n)^{m-1}-(p_j^n)^{m-1}\right|+\left|(D_h^+v_j^n)^{m-1}-(p_j^n)^{m-1}\right|\leq C_vh.
\end{equation}
Using $(a_{h,j}^n)^++(a_{h,j}^n)^-=a_{h,j}^n$, we further obtain
\begin{equation}
	\mathcal{L}_h[v^n]_j+(p_j^n)^{m-1}a_j^n=-\left((D_h^-v_j^n)^{m-1}-(p_j^n)^{m-1}\right)(a_{h,j}^n)^+-\left((D_h^+v_j^n)^{m-1}-(p_j^n)^{m-1}\right)(a_{h,j}^n)^-+(p_j^n)^{m-1}(a_j^n-a_{h,j}^n).
\end{equation}
The assumed consistency of $\mathcal{A}_h$ implies that $a_{h,j}^n$ is uniformly bounded. Hence,
\begin{equation}
	\left|\mathcal{L}_h[v^n]_j+(p_j^n)^{m-1}a_j^n\right|\leq C_v\left(h+\eta_{\mathcal A}(h)\right),
\end{equation}
which proves the claim.
\end{proof}
Now, define the residual
\begin{equation}
	R_j^n=D_\tau^\alpha v_j^n-\mathcal{L}_h[v^n]_j,
\end{equation}

and assume that the L1 approximation satisfies the weighted estimate
\begin{equation}\label{eq:weighted_temporal_consistency}
	t_n^\alpha\|D_\tau^\alpha v^n-\partial_t^\alpha v^n\|_\infty\leq C_v\eta_\tau,
	\quad n=1,\ldots,N,
\end{equation}
where $\eta_\tau\rightarrow0$ as $\tau\rightarrow0$, uniformly with respect to the spatial mesh. Since $v$ satisfies \eqref{eq:integrated_problem}, Lemma~\ref{lem:integrated_spatial_consistency} gives
\begin{equation}\label{eq:integrated_consistency}
	\|R^n\|_\infty\leq C_v\left(\eta_\tau t_n^{-\alpha}+h+\eta_{\mathcal A}(h)\right).
\end{equation}
Thus, in the notation used below, we may take
\begin{equation}
	\eta_h=h+\eta_{\mathcal A}(h).
\end{equation}
We next show that the weighted temporal consistency assumption allows the usual weak singularity at the initial time.
\begin{lemma}[Weighted consistency of the L1 approximation]
\label{lem:L1_weighted_consistency}
Assume that $v\in C([0,T];L^\infty(\mathbb{R}))\cap C^2((0,T];L^\infty(\mathbb{R}))$ and that
\begin{equation}\label{eq:weak_time_regularity}
\|\partial_t v(\cdot,t)\|_\infty\leq C_vt^{\alpha-1},
\quad
\|\partial_{tt}v(\cdot,t)\|_\infty\leq C_vt^{\alpha-2},
\quad 0<t\leq T.
\end{equation}
Then, uniformly for $n=1,\ldots,N$,
\begin{equation}\label{eq:L1_weighted_truncation}
\|D_\tau^\alpha v^n-\partial_t^\alpha v^n\|_\infty
\leq C_{v,\alpha}\tau^\alpha t_n^{-\alpha}.
\end{equation}
In particular, we may take $\eta_\tau=\tau^\alpha$ in \eqref{eq:weighted_temporal_consistency}.
\end{lemma}
\begin{proof}
Let $\Pi_\tau v$ denote the piecewise linear interpolant of $v$ in time, and put
\begin{equation}
e_\tau(t)=\partial_t(\Pi_\tau v)(t)-\partial_t v(t).
\end{equation}
Since the L1 formula is the Caputo derivative of this interpolant, its error is bounded by
\begin{equation}
\|D_\tau^\alpha v^n-\partial_t^\alpha v^n\|_\infty
\leq\frac1{\Gamma(1-\alpha)}\int_0^{t_n}
(t_n-t)^{-\alpha}\|e_\tau(t)\|_\infty\,dt.
\end{equation}
For $n=1$, \eqref{eq:weak_time_regularity} gives
\begin{equation}
\|e_\tau(t)\|_\infty\leq C_v(\tau^{\alpha-1}+t^{\alpha-1}),
\quad 0<t<\tau.
\end{equation}
Substituting this bound into the preceding integral gives a constant independent of $\tau$. This proves \eqref{eq:L1_weighted_truncation} for $n=1$.

Now, let $n\geq2$ and put $q=\lfloor n/2\rfloor$. On the first interval, we have
\begin{equation}
\int_0^\tau\|e_\tau(t)\|_\infty\,dt\leq C_v\tau^\alpha.
\end{equation}
On each subsequent interval, the mean-value representation of the interpolant's slope gives
\begin{equation}
\int_{t_{k-1}}^{t_k}\|e_\tau(t)\|_\infty\,dt
\leq\tau\int_{t_{k-1}}^{t_k}\|\partial_{tt}v(\cdot,t)\|_\infty\,dt.
\end{equation}
Consequently,
\begin{equation}
\int_0^{t_q}\|e_\tau(t)\|_\infty\,dt
\leq C_v\left(\tau^\alpha+\tau\int_\tau^{t_q}t^{\alpha-2}\,dt\right)
\leq C_{v,\alpha}\tau^\alpha,
\end{equation}
where the integral is zero if $q=1$. Since $t_n-t\geq t_n/2$ for $0\leq t\leq t_q$, this part contributes at most $C_{v,\alpha}\tau^\alpha t_n^{-\alpha}$.

For the remaining intervals, notice that $t_{k-1}\geq t_q\geq t_n/3$. Therefore,
\begin{equation}
\|e_\tau(t)\|_\infty\leq C_{v,\alpha}\tau t_n^{\alpha-2},
\quad t_q<t<t_n.
\end{equation}
It follows that
\begin{equation}
\int_{t_q}^{t_n}(t_n-t)^{-\alpha}\|e_\tau(t)\|_\infty\,dt
\leq C_{v,\alpha}\frac{\tau}{t_n}
\leq C_{v,\alpha}\tau^\alpha t_n^{-\alpha},
\end{equation}
because $\tau\leq t_n$ and $0<\alpha<1$. Combining the two parts proves the claim.
\end{proof}
\begin{remark}
The estimate \eqref{eq:L1_weighted_truncation} does not require the first-step residual to tend to zero. Its contribution is controlled after multiplication by $t_1^\alpha=\tau^\alpha$. The bounds \eqref{eq:weak_time_regularity} are assumptions on the exact cumulative solution, and are not proved here for the continuous nonlinear problem.
\end{remark}

We also need consistency of the initial cumulative data. Recall that
\begin{equation}
V_j^0=h\sum_{k=-\infty}^{j}U_k^0
\end{equation}
is associated with the interface point $x_{j+\frac{1}{2}}$, while
\begin{equation}
v_j^0=v(x_{j+\frac{1}{2}},0)=\int_{-\infty}^{x_{j+\frac{1}{2}}}u_0(x)\,dx.
\end{equation}

\begin{lemma}[Consistency of the initial cumulative data]
\label{lem:initial_cumulative_consistency}
Assume that $U_j^0=u_0(x_j)$ and $u_0\in W^{1,1}(\mathbb{R})$. Then
\begin{equation}
	\|V^0-v^0\|_\infty\leq Ch\|u_0'\|_{L^1(\mathbb{R})}.
\end{equation}
In particular,
\begin{equation}
	\|V^0-v^0\|_\infty\rightarrow 0 \quad \text{as} \quad h\rightarrow 0.
\end{equation}
\end{lemma}
\begin{proof}
Notice that for every $j\in\mathbb{Z}$ we have
\begin{equation}
	V_j^0-v_j^0=\sum_{k=-\infty}^{j}\left(hu_0(x_k)-\int_{x_{k-\frac{1}{2}}}^{x_{k+\frac{1}{2}}}u_0(x)\,dx\right).
\end{equation}
Hence,
\begin{equation}
	|V_j^0-v_j^0|\leq\sum_{k=-\infty}^{j}\int_{x_{k-\frac{1}{2}}}^{x_{k+\frac{1}{2}}}|u_0(x_k)-u_0(x)|\,dx.
\end{equation}
Next, by the fundamental theorem of calculus,
\begin{equation}
	|u_0(x_k)-u_0(x)|\leq\int_{x_{k-\frac{1}{2}}}^{x_{k+\frac{1}{2}}}|u_0'(y)|\,dy
\end{equation}
for $x\in[x_{k-\frac{1}{2}},x_{k+\frac{1}{2}}]$. Therefore,
\begin{equation}
	|V_j^0-v_j^0|\leq h\sum_{k=-\infty}^{j}\int_{x_{k-\frac{1}{2}}}^{x_{k+\frac{1}{2}}}|u_0'(y)|\,dy\leq h\|u_0'\|_{L^1(\mathbb{R})}.
\end{equation}
Taking the supremum over $j$ completes the proof.
\end{proof}

\begin{remark}
If, instead, the initial values are defined by cell averages,
\begin{equation}
	U_j^0=\frac{1}{h}\int_{x_{j-\frac{1}{2}}}^{x_{j+\frac{1}{2}}}u_0(x)\,dx,
\end{equation}
then
\begin{equation}
	V_j^0=v(x_{j+\frac{1}{2}},0)
\end{equation}
holds exactly.
\end{remark}
Next, we recall a discrete comparison estimate for the L1 operator derived in \cite{delTeso2026Numerical}.
\begin{lemma}[Discrete comparison estimate for the L1 operator]
\label{lem:L1_scalar_comparison}
Let $\{y^n\}_{n=0}^N$ be a nonnegative sequence satisfying
\begin{equation}
	D_\tau^\alpha y^n\leq R, \quad n=1,\ldots,N,
\end{equation}
for some $R\geq0$. Then
\begin{equation}
	y^n\leq y^0+\frac{R}{\Gamma(1+\alpha)}t_n^\alpha, \quad n=0,\ldots,N.
\end{equation}
Consequently,
\begin{equation}
	y^n\leq y^0+C_{\alpha,T}R, \quad n=0,\ldots,N.
\end{equation}
\end{lemma}
For a time-dependent right-hand side, we shall use the following weighted version of the comparison estimate; see also Lemma~2.1 of \cite{Kopteva2019L1}. We include a proof for the uniform mesh used here.
\begin{lemma}[Weighted discrete comparison estimate]
\label{lem:L1_weighted_comparison}
Let $\{y^n\}_{n=0}^N$ be a real-valued sequence satisfying
\begin{equation}
D_\tau^\alpha y^n\leq r^n,\quad r^n\geq0,\quad n=1,\ldots,N.
\end{equation}
Then
\begin{equation}\label{eq:L1_weighted_comparison}
y^n\leq y^0+\Gamma(1-\alpha)
\max_{1\leq k\leq n}t_k^\alpha r^k,
\quad n=1,\ldots,N.
\end{equation}
\end{lemma}
\begin{proof}
Fix $n\geq1$. If $\max_{0\leq k\leq n}y^k\leq y^0$, the result is immediate. Otherwise, choose $\ell\in\{1,\ldots,n\}$ such that
\begin{equation}
y^\ell=\max_{0\leq k\leq n}y^k>y^0.
\end{equation}
Using the equivalent representation of the L1 operator, we obtain
\begin{equation}
\begin{aligned}
D_\tau^\alpha y^\ell
&=\frac{\tau^{-\alpha}}{\Gamma(2-\alpha)}
\left[y^\ell-\sum_{k=0}^{\ell-2}(a_k-a_{k+1})y^{\ell-1-k}-a_{\ell-1}y^0\right]\\
&\geq\frac{\tau^{-\alpha}a_{\ell-1}}{\Gamma(2-\alpha)}(y^\ell-y^0).
\end{aligned}
\end{equation}
Moreover,
\begin{equation}
a_{\ell-1}=(1-\alpha)\int_{\ell-1}^{\ell}r^{-\alpha}\,dr
\geq(1-\alpha)\ell^{-\alpha}.
\end{equation}
Hence,
\begin{equation}
r^\ell\geq D_\tau^\alpha y^\ell
\geq\frac{t_\ell^{-\alpha}}{\Gamma(1-\alpha)}(y^\ell-y^0).
\end{equation}
Since $y^n\leq y^\ell$, this proves \eqref{eq:L1_weighted_comparison}.
\end{proof}

\begin{theorem}[Convergence of the integrated implicit scheme]
\label{thm:integrated_convergence}
Assume that $m\geq2$ and that \eqref{eq:integrated_problem} has a bounded non-decreasing solution $v\in C([0,T];L^\infty(\mathbb{R}))$ satisfying the spatial assumptions of Lemma~\ref{lem:integrated_spatial_consistency} and the weighted temporal bound \eqref{eq:weighted_temporal_consistency}. Let $V^n$ be a bounded non-decreasing solution of \eqref{eq:integrated_discrete_scheme}, and assume that \eqref{eq:integrated_consistency} holds. Then
\begin{equation}
	\|V^n-v^n\|_\infty\leq\|V^0-v^0\|_\infty+C_{v,\alpha,T}(\eta_\tau+\eta_h), \quad n=1,\ldots,N,
\end{equation}
where $C_{v,\alpha,T}$ is independent of $h$ and $\tau$.

In particular, if $U_j^0=u_0(x_j)$ and $u_0\in W^{1,1}(\mathbb{R})$, then Lemma~\ref{lem:initial_cumulative_consistency} gives
\begin{equation}
	\|V^n-v^n\|_\infty\leq C\left(h+\eta_\tau+\eta_h\right), \quad n=1,\ldots,N,
\end{equation}
and therefore $V^n$ converges uniformly to the exact cumulative solution as $h,\tau\rightarrow 0$.

If, in addition, \eqref{eq:weak_time_regularity} holds, then Lemma~\ref{lem:L1_weighted_consistency} gives
\begin{equation}
\|V^n-v^n\|_\infty\leq C\left(h+\tau^\alpha+\eta_h\right),
\quad n=1,\ldots,N.
\end{equation}

\end{theorem}
\begin{proof}
Note that the exact solution satisfies
\begin{equation}
	v_j^n-\lambda_\alpha\mathcal{L}_h[v^n]_j=\sum_{k=0}^{n-2}(a_k-a_{k+1})v_j^{n-1-k}+a_{n-1}v_j^0+\lambda_\alpha R_j^n.
\end{equation}
Applying the resolvent contraction estimate \eqref{eq:resolvent_contraction} to $V^n$ and $v^n$ gives
\begin{equation}
	E^n\leq\sum_{k=0}^{n-2}(a_k-a_{k+1})E^{n-1-k}+a_{n-1}E^0+\lambda_\alpha\|R^n\|_\infty,
\end{equation}
where $E^n=\|V^n-v^n\|_\infty$. Equivalently, we can write
\begin{equation}
	D_\tau^\alpha E^n\leq\|R^n\|_\infty.
\end{equation}
Lemma~\ref{lem:L1_weighted_comparison} and \eqref{eq:integrated_consistency} yield
\begin{equation}
\begin{aligned}
E^n&\leq E^0+\Gamma(1-\alpha)\max_{1\leq k\leq n}t_k^\alpha\|R^k\|_\infty\\
&\leq E^0+C_{v,\alpha}\left(\eta_\tau+t_n^\alpha\eta_h\right)
\leq E^0+C_{v,\alpha,T}(\eta_\tau+\eta_h).
\end{aligned}
\end{equation}
The estimate for sampled initial data follows from Lemma~\ref{lem:initial_cumulative_consistency}. Finally, under \eqref{eq:weak_time_regularity}, Lemma~\ref{lem:L1_weighted_consistency} allows us to take $\eta_\tau=\tau^\alpha$. This proves the result.
\end{proof}
We can immediately obtain the weak-convergence.
\begin{corollary}[Weak convergence of the density]
\label{cor:density_weak_convergence}
Assume the hypotheses of Theorem~\ref{thm:integrated_convergence} and Lemma~\ref{lem:initial_cumulative_consistency}. Let $n=n(h,\tau)$ be a sequence of time indices such that
\begin{equation}
	t_{n(h,\tau)}\rightarrow t\in[0,T]
\end{equation}
as $h,\tau\rightarrow0$. Then, for every $\varphi\in C_c^1(\mathbb{R})$,
\begin{equation}
	h\sum_{j\in\mathbb{Z}}U_j^{n(h,\tau)}\varphi(x_j)\rightarrow\int_{\mathbb{R}}u(x,t)\varphi(x)\,dx.
\end{equation}
\end{corollary}
\begin{proof}
Since $U_j^n=D_h^-V_j^n$, discrete summation by parts gives
\begin{equation}
	h\sum_{j\in\mathbb{Z}}U_j^n\varphi(x_j)=-\sum_{j\in\mathbb{Z}}V_j^n\left(\varphi(x_{j+1})-\varphi(x_j)\right).
\end{equation}
Let $n=n(h,\tau)$ and write
\begin{equation}
	-\sum_jV_j^n\left(\varphi(x_{j+1})-\varphi(x_j)\right)=-\sum_jv_j^n\left(\varphi(x_{j+1})-\varphi(x_j)\right)-\sum_j(V_j^n-v_j^n)\left(\varphi(x_{j+1})-\varphi(x_j)\right).
\end{equation}
Notice that for the second term, we have
\begin{equation}
	\left|\sum_j(V_j^n-v_j^n)\left(\varphi(x_{j+1})-\varphi(x_j)\right)\right|\leq\|V^n-v^n\|_\infty\sum_j|\varphi(x_{j+1})-\varphi(x_j)|.
\end{equation}
But $\varphi\in C_c^1(\mathbb{R})$, so that
\begin{equation}
	\sum_j|\varphi(x_{j+1})-\varphi(x_j)|\leq\|\varphi'\|_{L^1(\mathbb{R})}+o(1),
\end{equation}
and Theorem~\ref{thm:integrated_convergence} together with Lemma~\ref{lem:initial_cumulative_consistency} implies that this term converges to zero.

For the first term, recall that $v_j^n=v(x_{j+\frac{1}{2}},t_n)$. Hence,
\begin{equation}
	-\sum_jv_j^n\left(\varphi(x_{j+1})-\varphi(x_j)\right)\rightarrow-\int_{\mathbb{R}}v(x,t)\varphi'(x)\,dx
\end{equation}
by the consistency of the midpoint Riemann sum and the continuity of $v$ in time. Finally, since $\partial_xv=u$,
\begin{equation}
	-\int_{\mathbb{R}}v(x,t)\varphi'(x)\,dx=\int_{\mathbb{R}}u(x,t)\varphi(x)\,dx.
\end{equation}
This concludes the proof. 
\end{proof}

\begin{remark}[Higher-dimensional convergence]
We would like to mention that the above convergence argument is specific to one space dimension, since it relies on the cumulative variable $v(x,t)=\int_{-\infty}^x u(y,t)\,dy$ and on the resulting monotone formulation of the integrated equation. In dimensions $d>1$, there is no direct scalar analog of this transformation that yields a closed equation with the same comparison structure.

Nevertheless, the multidimensional scheme developed in this paper has several estimates that may provide a basis for a full multidimensional convergence analysis. In particular, positivity, mass conservation, the uniform $\ell^\infty$ and $\ell^2$ bounds, and the discrete energy-dissipation estimate of Lemma~\ref{lem:energy_dissipation} are available in arbitrary dimension. This could lead to subsequential convergence of the scheme to a weak solution of the multidimensional problem.
\end{remark}

\subsection{Convergence of the explicit scheme}
We now consider the explicit version of the scheme given by
\begin{equation}\label{eq:explicit_scheme}
	D_\tau^\alpha U_j^n=-(\mathcal{N}_hU^{n-1})_j, \quad n\geq1.
\end{equation}
Equivalently, after rearranging the L1 approximation of the Caputo derivative, we obtain
\begin{equation}
	U_j^n=\sum_{k=0}^{n-2}(a_k-a_{k+1})U_j^{n-1-k}+a_{n-1}U_j^0-\Gamma(2-\alpha)\tau^\alpha(\mathcal{N}_hU^{n-1})_j.
\end{equation}
As before, let
\begin{equation}
	\mu_\alpha=\Gamma(2-\alpha)\tau^\alpha, \quad b_k=a_k-a_{k+1}, \quad k\geq0,
\end{equation}
and note that $b_k>0$ and
\begin{equation}
	\sum_{k=0}^{n-2}b_k+a_{n-1}=1.
\end{equation}

\begin{proposition}[Mass conservation of the explicit scheme]
\label{prop:explicit_mass}
Assume that $U^n\in\ell^1(\mathbb{Z}^d)$ is a non-negative solution of \eqref{eq:explicit_scheme}. Suppose additionally that, for all $n=0,\ldots,N-1$ and $i=1,\ldots,d$,
\begin{equation}
	\sum_{j\in\mathbb{Z}^d}\left|F_{j+\frac12e_i}^{(i)}(U^n)\right|<\infty.
\end{equation}
Then
\begin{equation}
	h^d\sum_{j\in\mathbb{Z}^d}U_j^n=h^d\sum_{j\in\mathbb{Z}^d}U_j^0, \quad n=0,\ldots,N.
\end{equation}
\end{proposition}

\begin{proof}
Let
\begin{equation}
	M^n=h^d\sum_{j\in\mathbb{Z}^d}U_j^n.
\end{equation}
Summing \eqref{eq:explicit_scheme} over $j\in\mathbb{Z}^d$ and using the absolute summability of the fluxes gives
\begin{equation}
	D_\tau^\alpha M^n=0.
\end{equation}
Equivalently,
\begin{equation}
	M^n=\sum_{k=0}^{n-2}(a_k-a_{k+1})M^{n-1-k}+a_{n-1}M^0.
\end{equation}
Since
\begin{equation}
	\sum_{k=0}^{n-2}(a_k-a_{k+1})+a_{n-1}=1,
\end{equation}
an induction in $n$ yields $M^n=M^0$.
\end{proof}

The explicit scheme is well-defined in an arbitrary dimension. However, for the rigorous convergence analysis below, we again restrict ourselves to $d=1$ and use the integrated formulation introduced in the preceding subsection. Thus,
\begin{equation}
	V_j^n=h\sum_{k=-\infty}^{j}U_k^n,
\end{equation}
where $V_j^n$ is associated with the interface point $x_{j+\frac12}$. Summing \eqref{eq:explicit_scheme} over the indices $k\leq j$ and using the vanishing of the flux at $-\infty$, we obtain
\begin{equation}\label{eq:explicit_integrated_scheme}
	D_\tau^\alpha V_j^n=\mathcal{L}_h[V^{n-1}]_j.
\end{equation}
Now, for $n\geq2$, the L1 formula allows us to rewrite \eqref{eq:explicit_integrated_scheme} as
\begin{equation}\label{eq:explicit_integrated_update}
	V_j^n=\mathcal{S}_h[V^{n-1}]_j+\sum_{k=1}^{n-2}b_kV_j^{n-1-k}+a_{n-1}V_j^0,
\end{equation}
where
\begin{equation}\label{eq:explicit_map}
	\mathcal{S}_h[V]_j=b_0V_j+\mu_\alpha\mathcal{L}_h[V]_j.
\end{equation}
In the first time step, we instead have
\begin{equation}
	V_j^1=\mathcal{S}_{h,0}[V^0]_j, \quad \mathcal{S}_{h,0}[V]_j=V_j+\mu_\alpha\mathcal{L}_h[V]_j.
\end{equation}

We now derive a sufficient CFL condition directly in terms of the coefficients of the fractional discrete Laplacian. Recall that
\begin{equation}
	(\mathcal{A}_hV)_j=\sum_{\beta\neq0}C_\beta(h)(V_j-V_{j+\beta}),
\end{equation}
where
\begin{equation}
	C_\beta(h)=\frac{c_\beta(h)}{h}\geq0.
\end{equation}
For the semigroup discretization, Proposition~\ref{prop:semigroup_fractional_laplacian} gives
\begin{equation}\label{eq:Cbeta_decay_explicit}
	C_\beta(h)\leq C_sh^{-2s}|\beta|^{-1-2s}, \quad \beta\neq0,
\end{equation}
where $s=1-\sigma$. In particular,
\begin{equation}\label{eq:Cbeta_sum_explicit}
	\Lambda_h:=\sum_{\beta\neq0}C_\beta(h)\leq C_sh^{-2s}.
\end{equation}
Next, for a bounded sequence $V$, let
\begin{equation}
	\operatorname{osc}(V)=\sup_{j\in\mathbb{Z}}V_j-\inf_{j\in\mathbb{Z}}V_j.
\end{equation}
and for $B,M>0$, define
\begin{equation}
	\mathcal{K}_{B,M}=\left\{V\in\ell^\infty(\mathbb{Z}):\operatorname{osc}(V)\leq B,\ 0\leq D_h^+V_j\leq M \text{ for all }j\in\mathbb{Z}\right\}.
\end{equation}

\begin{lemma}[Bound for the discrete fractional Laplacian]
\label{lem:Ah_bound_explicit}
Let $V\in\mathcal{K}_{B,M}$ and $0<h<1/2$. Then
\begin{equation}
	\|\mathcal{A}_hV\|_\infty\leq Q_s(h;B,M),
\end{equation}
where
\begin{equation}\label{eq:Qs_explicit}
	Q_s(h;B,M)=C_s
	\begin{cases}
		B+M, & 0<s<\frac12,\\
		B+M(1+|\log h|), & s=\frac12,\\
		B+Mh^{1-2s}, & \frac12<s<1.
	\end{cases}
\end{equation}
\end{lemma}

\begin{proof}
Since $0\leq D_h^+V_j\leq M$ and $\operatorname{osc}(V)\leq B$, we have
\begin{equation}
	|V_j-V_{j+\beta}|\leq\min\{Mh|\beta|,B\}.
\end{equation}
Therefore, by \eqref{eq:Cbeta_decay_explicit},
\begin{equation}
	|(\mathcal{A}_hV)_j|\leq C_sh^{-2s}\sum_{\beta\neq0}|\beta|^{-1-2s}\min\{Mh|\beta|,B\}.
\end{equation}
Now, we can split the sum into parts $1\leq|\beta|\leq h^{-1}$ and $|\beta|>h^{-1}$. For the first part, we have
\begin{equation}
	h^{-2s}\sum_{1\leq|\beta|\leq h^{-1}}|\beta|^{-1-2s}Mh|\beta|\leq C_sMh^{1-2s}\sum_{1\leq\beta\leq h^{-1}}\beta^{-2s}.
\end{equation}
Hence,
\begin{equation}
	h^{1-2s}\sum_{1\leq\beta\leq h^{-1}}\beta^{-2s}\leq C_s
	\begin{cases}
		1, & 0<s<\frac12,\\
		1+|\log h|, & s=\frac12,\\
		h^{1-2s}, & \frac12<s<1.
	\end{cases}
\end{equation}
For the remaining indices, it follows that
\begin{equation}
	h^{-2s}\sum_{|\beta|>h^{-1}}|\beta|^{-1-2s}B\leq C_sB.
\end{equation}
Combining these estimates proves the result.
\end{proof}

\begin{lemma}[Monotonicity of the explicit scheme]
\label{lem:explicit_monotonicity}
Assume that $m\geq2$ and let
\begin{equation}\label{eq:Kh_explicit}
	K_h=(m-1)M^{m-2}\frac{Q_s(h;B,M)}{h}+M^{m-1}\Lambda_h.
\end{equation}
If
\begin{equation}\label{eq:explicit_CFL}
	\mu_\alpha K_h\leq b_0,
\end{equation}
then $\mathcal{S}_h$ and $\mathcal{S}_{h,0}$ are order-preserving on $\mathcal{K}_{B,M}$.
\end{lemma}

\begin{proof}
Let
\begin{equation}
	A_j=(\mathcal{A}_hV)_j,
\end{equation}
and note that if $A_j>0$, then
\begin{equation}
	\mathcal{L}_h[V]_j=-(D_h^-V_j)^{m-1}A_j.
\end{equation}
Observe that all off-diagonal partial derivatives of $\mathcal{L}_h[V]_j$ are non-negative. Indeed, if $\beta\notin\{0,-1\}$, then
\begin{equation}
	\frac{\partial\mathcal{L}_h[V]_j}{\partial V_{j+\beta}}=(D_h^-V_j)^{m-1}C_\beta(h)\geq0,
\end{equation}
while for $\beta=-1$,
\begin{equation}
	\frac{\partial\mathcal{L}_h[V]_j}{\partial V_{j-1}}=\frac{m-1}{h}(D_h^-V_j)^{m-2}A_j+(D_h^-V_j)^{m-1}C_{-1}(h)\geq0.
\end{equation}
The diagonal derivative satisfies
\begin{equation}
	-\frac{\partial\mathcal{L}_h[V]_j}{\partial V_j}=\frac{m-1}{h}(D_h^-V_j)^{m-2}A_j+(D_h^-V_j)^{m-1}\Lambda_h\leq K_h.
\end{equation}
Now, if $A_j<0$, then
\begin{equation}
	\mathcal{L}_h[V]_j=-(D_h^+V_j)^{m-1}A_j.
\end{equation}
Again, we can see that all off-diagonal partial derivatives are non-negative. Moreover,
\begin{equation}
	-\frac{\partial\mathcal{L}_h[V]_j}{\partial V_j}=\frac{m-1}{h}(D_h^+V_j)^{m-2}|A_j|+(D_h^+V_j)^{m-1}\Lambda_h\leq K_h.
\end{equation}
By Lemma~\ref{lem:Ah_bound_explicit}, these estimates are uniform on $\mathcal{K}_{B,M}$. Consequently, \eqref{eq:explicit_CFL} implies
\begin{equation}
	\frac{\partial\mathcal{S}_h[V]_j}{\partial V_j}\geq b_0-\mu_\alpha K_h\geq0,
\end{equation}
and all off-diagonal derivatives are non-negative. Moreover,
\begin{equation}
	\frac{\partial\mathcal{S}_{h,0}[V]_j}{\partial V_j}\geq1-\mu_\alpha K_h\geq1-b_0>0.
\end{equation}
The operators are continuous across the set $A_j=0$, and $\mathcal{K}_{B,M}$ is convex. Hence, integration along the line segment that joins two ordered elements of $\mathcal{K}_{B,M}$ proves the order-preserving property. The proof is complete. 
\end{proof}

\begin{proposition}[Positivity and $\ell^\infty$ stability]
\label{prop:explicit_stability}
Assume that $m\geq2$ and
\begin{equation}
	0\leq U_j^0\leq M.
\end{equation}
Let
\begin{equation}
	V_j^0=h\sum_{k=-\infty}^{j}U_k^0
\end{equation}
and assume that
\begin{equation}
	\operatorname{osc}(V^0)\leq B.
\end{equation}
If the CFL condition \eqref{eq:explicit_CFL} holds, then
\begin{equation}
	\operatorname{osc}(V^n)\leq B, \quad 0\leq D_h^+V_j^n\leq M
\end{equation}
for all $j\in\mathbb{Z}$ and $n=0,\ldots,N$. Consequently,
\begin{equation}
	0\leq U_j^n\leq M.
\end{equation}
In particular,
\begin{equation}
	\|U^n\|_\infty\leq\|U^0\|_\infty.
\end{equation}
\end{proposition}

\begin{proof}
We start by letting
\begin{equation}
	a=\inf_{j\in\mathbb{Z}}V_j^0, \quad b=\sup_{j\in\mathbb{Z}}V_j^0,
\end{equation}
which shows that $b-a\leq B$. Since constant sequences belong to $\mathcal{K}_{B,M}$ and $\mathcal{L}_h$ annihilates constants,
\begin{equation}
	\mathcal{S}_{h,0}[a]=a, \quad \mathcal{S}_{h,0}[b]=b.
\end{equation}
Because $a\leq V^0\leq b$, Lemma~\ref{lem:explicit_monotonicity} gives
\begin{equation}
	a\leq V^1\leq b.
\end{equation}
For $n\geq2$, we similarly use
\begin{equation}
	\mathcal{S}_h[a]=b_0a, \quad \mathcal{S}_h[b]=b_0b.
\end{equation}
Because
\begin{equation}
	b_0+\sum_{k=1}^{n-2}b_k+a_{n-1}=1,
\end{equation}
an induction in \eqref{eq:explicit_integrated_update} yields
\begin{equation}
	a\leq V^n\leq b.
\end{equation}
Thus,
\begin{equation}
	\operatorname{osc}(V^n)\leq B.
\end{equation}
It remains to control the discrete slopes. Let $\mathcal{T}$ denote the spatial shift
\begin{equation}
	(\mathcal{T}V)_j=V_{j+1}.
\end{equation}
We note that the initial assumption implies
\begin{equation}
	V^0\leq\mathcal{T}V^0\leq V^0+Mh.
\end{equation}
Using monotonicity, translation invariance and invariance of $\mathcal{L}_h$ under the addition of constants, we obtain
\begin{equation}
	V^1=\mathcal{S}_{h,0}[V^0]\leq\mathcal{T}V^1\leq\mathcal{S}_{h,0}[V^0+Mh]=V^1+Mh.
\end{equation}
Hence, we arrive at
\begin{equation}
	0\leq D_h^+V_j^1\leq M.
\end{equation}
Now, assume that the same estimate holds up to the time level $n-1$. Then
\begin{equation}
	\mathcal{S}_h[V^{n-1}]\leq\mathcal{S}_h[\mathcal{T}V^{n-1}]\leq\mathcal{S}_h[V^{n-1}+Mh]=\mathcal{S}_h[V^{n-1}]+b_0Mh.
\end{equation}
Applying the same argument to the history terms in \eqref{eq:explicit_integrated_update}, we obtain
\begin{equation}
	V^n\leq\mathcal{T}V^n\leq V^n+Mh\left(b_0+\sum_{k=1}^{n-2}b_k+a_{n-1}\right)=V^n+Mh.
\end{equation}
Therefore,
\begin{equation}
	0\leq D_h^+V_j^n\leq M.
\end{equation}
Since
\begin{equation}
	U_j^n=D_h^-V_j^n=D_h^+V_{j-1}^n,
\end{equation}
the result follows and the proof is complete. 
\end{proof}

\begin{remark}[CFL scaling]
\label{rem:explicit_CFL_scaling}
For fixed $B$ and $M$, Lemma~\ref{lem:Ah_bound_explicit}, \eqref{eq:Cbeta_sum_explicit} and \eqref{eq:Kh_explicit} imply
\begin{equation}
	K_h\leq C_{s,m,B,M}
	\begin{cases}
		h^{-1}, & 0<s<\frac12,\\
		h^{-1}(1+|\log h|), & s=\frac12,\\
		h^{-2s}, & \frac12<s<1.
	\end{cases}
\end{equation}
Hence, a sufficient form of the CFL condition is
\begin{equation}\label{eq:explicit_CFL_scaling}
	\tau^\alpha\leq C
	\begin{cases}
		h, & 0<s<\frac12,\\
		\dfrac{h}{1+|\log h|}, & s=\frac12,\\
		h^{2s}, & \frac12<s<1.
	\end{cases}
\end{equation}
Since $s=1-\sigma$, this can equivalently be written as
\begin{equation}
	\tau^\alpha\leq C
	\begin{cases}
		h^{2-2\sigma}, & 0<\sigma<\frac12,\\
		\dfrac{h}{1+|\log h|}, & \sigma=\frac12,\\
		h, & \frac12<\sigma<1.
	\end{cases}
\end{equation}
Thus, the explicit stability restriction exhibits fractional-parabolic scaling for $\sigma<1/2$ and hyperbolic scaling for $\sigma>1/2$, with the usual logarithmic correction at the critical value $\sigma=1/2$.
\end{remark}

\begin{lemma}[Contractivity of the explicit scheme]
\label{lem:explicit_contraction}
Assume the hypotheses of Lemma~\ref{lem:explicit_monotonicity}. Then, for every $V,W\in\mathcal{K}_{B,M}$, we have
\begin{equation}
	\|\mathcal{S}_h[V]-\mathcal{S}_h[W]\|_\infty\leq b_0\|V-W\|_\infty.
\end{equation}
Moreover, the following holds
\begin{equation}
	\|\mathcal{S}_{h,0}[V]-\mathcal{S}_{h,0}[W]\|_\infty\leq\|V-W\|_\infty.
\end{equation}
\end{lemma}
\begin{proof}
Let
\begin{equation}
	\delta=\|V-W\|_\infty,
\end{equation}
and notice that
\begin{equation}
	V\leq W+\delta, \quad W\leq V+\delta.
\end{equation}
Since addition of a constant changes neither the oscillation nor the discrete derivatives, both $V+\delta$ and $W+\delta$ belong to $\mathcal{K}_{B,M}$. By Lemma~\ref{lem:explicit_monotonicity}, we have
\begin{equation}
	\mathcal{S}_h[V]\leq\mathcal{S}_h[W+\delta]=\mathcal{S}_h[W]+b_0\delta.
\end{equation}
Next, by interchanging $V$ and $W$ we have
\begin{equation}
	\mathcal{S}_h[W]\leq\mathcal{S}_h[V]+b_0\delta,
\end{equation}
and thus,
\begin{equation}
	\|\mathcal{S}_h[V]-\mathcal{S}_h[W]\|_\infty\leq b_0\delta.
\end{equation}
The proof for $\mathcal{S}_{h,0}$ is identical, using
\begin{equation}
	\mathcal{S}_{h,0}[V+c]=\mathcal{S}_{h,0}[V]+c.
\end{equation}
Therefore, we have proved our claim. 
\end{proof}
We now consider the consistency of the explicit integrated scheme. As in the implicit case, let
\begin{equation}
	v_j^n=v(x_{j+\frac12},t_n).
\end{equation}
Since the spatial operator in \eqref{eq:explicit_integrated_scheme} is evaluated at the previous time level, it is natural to define the residual by
\begin{equation}\label{eq:explicit_residual}
	R_j^n=D_\tau^\alpha v_j^n-\mathcal{L}_h[v^{n-1}]_j.
\end{equation}
For the continuous operator $\mathcal Lv=-|\partial_xv|^{m-1}(-\Delta)^sv$, define the temporal modulus
\begin{equation}\label{eq:explicit_time_modulus}
\omega_v(\tau)=\sup_{\substack{0\leq t,r\leq T\\|t-r|\leq\tau}}
\|\mathcal Lv(\cdot,t)-\mathcal Lv(\cdot,r)\|_\infty.
\end{equation}
Here, we assume that $t\mapsto\mathcal Lv(\cdot,t)$ belongs to $C([0,T];L^\infty(\mathbb{R}))$, so that $\omega_v(\tau)\rightarrow0$ as $\tau\rightarrow0$.
\begin{lemma}[Consistency of the explicit integrated scheme]
\label{lem:explicit_consistency}
Assume that $m\geq2$ and that $v$ is a bounded non-decreasing solution of \eqref{eq:integrated_problem}, sufficiently smooth in space with the required spatial bounds uniform on $[0,T]$. Also, suppose that \eqref{eq:explicit_time_modulus} satisfies $\omega_v(\tau)\rightarrow0$ and that
\begin{equation}
	t_n^\alpha\|D_\tau^\alpha v^n-\partial_t^\alpha v^n\|_\infty\leq C_v\eta_\tau
\end{equation}
and
\begin{equation}
	\|\mathcal{A}_hv^n-(-\Delta)^sv^n\|_\infty\leq C_v\eta_{\mathcal A}(h),
\end{equation}
with the temporal bound holding for $n=1,\ldots,N$ and the spatial bound for $n=0,\ldots,N$. Then
\begin{equation}\label{eq:explicit_consistency}
	\|R^n\|_\infty\leq C_v\left(\eta_\tau t_n^{-\alpha}+\omega_v(\tau)+h+\eta_{\mathcal A}(h)\right).
\end{equation}
\end{lemma}
\begin{proof}
For convenience, we define the following operator
\begin{equation}
	\mathcal{L}v=-|\partial_xv|^{m-1}(-\Delta)^sv.
\end{equation}
Since $v$ satisfies \eqref{eq:integrated_problem}, we have
\begin{equation}
	\partial_t^\alpha v=\mathcal{L}v,
\end{equation}
and therefore,
\begin{equation}
	R^n=\left(D_\tau^\alpha v^n-\partial_t^\alpha v^n\right)+\left(\mathcal{L}v^n-\mathcal{L}_h[v^{n-1}]\right).
\end{equation}
Now, we can decompose the second term as
\begin{equation}
	\mathcal{L}v^n-\mathcal{L}_h[v^{n-1}]=\left(\mathcal{L}v^n-\mathcal{L}_h[v^n]\right)+\left(\mathcal{L}_h[v^n]-\mathcal{L}_h[v^{n-1}]\right).
\end{equation}
To estimate the last difference, we write
\begin{equation}
	\mathcal{L}_h[v^n]-\mathcal{L}_h[v^{n-1}]=\left(\mathcal{L}_h[v^n]-\mathcal{L}v^n\right)+\left(\mathcal{L}v^n-\mathcal{L}v^{n-1}\right)+\left(\mathcal{L}v^{n-1}-\mathcal{L}_h[v^{n-1}]\right).
\end{equation}
By Lemma~\ref{lem:integrated_spatial_consistency}, we also have
\begin{equation}
	\|\mathcal{L}_h[v^k]-\mathcal{L}v^k\|_\infty\leq C_v\left(h+\eta_{\mathcal A}(h)\right), \quad k=n-1,n.
\end{equation}
Moreover, the definition of $\omega_v$ gives
\begin{equation}
	\|\mathcal{L}v^n-\mathcal{L}v^{n-1}\|_\infty\leq\omega_v(\tau),
\end{equation}
and consequently, we can obtain
\begin{equation}
	\|\mathcal{L}_h[v^n]-\mathcal{L}_h[v^{n-1}]\|_\infty\leq C_v\left(\omega_v(\tau)+h+\eta_{\mathcal A}(h)\right).
\end{equation}
Combining the preceding estimates yields
\begin{equation}
	\|R^n\|_\infty\leq C_v\left(\eta_\tau t_n^{-\alpha}+\omega_v(\tau)+h+\eta_{\mathcal A}(h)\right),
\end{equation}
which proves the claim and concludes the proof. 
\end{proof}
At this point, we have gathered all ingredients needed for the main result of this part - the proof of convergence. For convenience, let
\begin{equation}
	\eta_h=h+\eta_{\mathcal A}(h).
\end{equation}
\begin{theorem}[Convergence of the explicit integrated scheme]
\label{thm:explicit_convergence}
Assume that $m\geq2$ and that \eqref{eq:integrated_problem} admits a bounded non-decreasing solution $v\in C([0,T];L^\infty(\mathbb{R}))$ satisfying the hypotheses of Lemma~\ref{lem:explicit_consistency}, with $\eta_\tau\rightarrow0$, $\eta_{\mathcal A}(h)\rightarrow0$, and $\omega_v(\tau)\rightarrow0$. Assume that
\begin{equation}
	0\leq u(x,t)\leq M, \quad (x,t)\in\mathbb{R}\times[0,T],
\end{equation}
and that the numerical initial data satisfy
\begin{equation}
	0\leq U_j^0\leq M.
\end{equation}
Let $V^n$ be generated by \eqref{eq:explicit_integrated_scheme}, and suppose that there exists $B>0$ such that
\begin{equation}
	\operatorname{osc}(V^0)\leq B, \quad \operatorname{osc}(v^n)\leq B
\end{equation}
for $n=0,\ldots,N$, and that the CFL condition
\begin{equation}
	\mu_\alpha K_h\leq b_0,
\end{equation}
with $K_h$ defined by \eqref{eq:Kh_explicit}, holds. Then
\begin{equation}
	\|V^n-v^n\|_\infty\leq\|V^0-v^0\|_\infty+C_{v,\alpha,T}\left(\eta_\tau+\omega_v(\tau)+\eta_h\right), \quad n=1,\ldots,N.
\end{equation}
In particular, if $U_j^0=u_0(x_j)$ and $u_0\in W^{1,1}(\mathbb{R})$, then
\begin{equation}
	\|V^n-v^n\|_\infty\leq C\left(h+\eta_\tau+\omega_v(\tau)+\eta_h\right).
\end{equation}
If, in addition, \eqref{eq:weak_time_regularity} holds and $\omega_v(\tau)\leq C_v\tau^\alpha$, then
\begin{equation}
\|V^n-v^n\|_\infty\leq C\left(h+\tau^\alpha+\eta_h\right).
\end{equation}
The bound on $\omega_v$ is an additional time-regularity assumption on the continuous spatial operator.
\end{theorem}
\begin{proof}
Start by letting
\begin{equation}
	E^n=\|V^n-v^n\|_\infty.
\end{equation}
By Proposition~\ref{prop:explicit_stability}, we have
\begin{equation}
	V^n\in\mathcal{K}_{B,M}
\end{equation}
for every $n$. Moreover, since $0\leq\partial_xv=u\leq M$ and $\operatorname{osc}(v^n)\leq B$, the sampled exact solution also satisfies
\begin{equation}
	v^n\in\mathcal{K}_{B,M}.
\end{equation}
In the first time step, the numerical solution satisfies
\begin{equation}
	V^1=\mathcal{S}_{h,0}[V^0].
\end{equation}
On the other hand, by \eqref{eq:explicit_residual}, we have
\begin{equation}
	v^1=\mathcal{S}_{h,0}[v^0]+\mu_\alpha R^1.
\end{equation}
Therefore, Lemma~\ref{lem:explicit_contraction} gives
\begin{equation}
	E^1\leq E^0+\mu_\alpha\|R^1\|_\infty.
\end{equation}
Also, observe that for $n\geq2$, the exact solution satisfies
\begin{equation}
	v_j^n=\mathcal{S}_h[v^{n-1}]_j+\sum_{k=1}^{n-2}b_kv_j^{n-1-k}+a_{n-1}v_j^0+\mu_\alpha R_j^n.
\end{equation}
Next, subtracting this identity from \eqref{eq:explicit_integrated_update}, taking the $\ell^\infty$ norm, and applying Lemma~\ref{lem:explicit_contraction}, we obtain
\begin{equation}
	E^n\leq b_0E^{n-1}+\sum_{k=1}^{n-2}b_kE^{n-1-k}+a_{n-1}E^0+\mu_\alpha\|R^n\|_\infty.
\end{equation}
But together with the estimate for $n=1$, this is equivalent to
\begin{equation}
	D_\tau^\alpha E^n\leq\|R^n\|_\infty.
\end{equation}
Lemma~\ref{lem:L1_weighted_comparison} and Lemma~\ref{lem:explicit_consistency} yield
\begin{equation}
\begin{aligned}
E^n&\leq E^0+\Gamma(1-\alpha)\max_{1\leq k\leq n}t_k^\alpha\|R^k\|_\infty\\
&\leq E^0+C_{v,\alpha,T}\left(\eta_\tau+\omega_v(\tau)+\eta_h\right).
\end{aligned}
\end{equation}
Finally, Lemma~\ref{lem:initial_cumulative_consistency} gives
\begin{equation}
	E^0\leq Ch
\end{equation}
when $U_j^0=u_0(x_j)$ and $u_0\in W^{1,1}(\mathbb{R})$. Under the additional assumptions, Lemma~\ref{lem:L1_weighted_consistency} gives $\eta_\tau=\tau^\alpha$, and the bound on $\omega_v$ yields the last estimate. This completes the proof.
\end{proof}
\begin{corollary}[Weak convergence of the explicit density scheme]
\label{cor:explicit_density_convergence}
Assume the hypotheses of Theorem~\ref{thm:explicit_convergence} and Lemma~\ref{lem:initial_cumulative_consistency}, and let $n=n(h,\tau)$ satisfy
\begin{equation}
	t_{n(h,\tau)}\rightarrow t\in[0,T]
\end{equation}
as $h,\tau\rightarrow0$. Then, for every $\varphi\in C_c^1(\mathbb{R})$,
\begin{equation}
	h\sum_{j\in\mathbb{Z}}U_j^{n(h,\tau)}\varphi(x_j)\rightarrow\int_{\mathbb{R}}u(x,t)\varphi(x)\,dx.
\end{equation}
\end{corollary}

\begin{proof}
The proof is identical to that of Corollary~\ref{cor:density_weak_convergence}, using Theorem~\ref{thm:explicit_convergence} in place of Theorem~\ref{thm:integrated_convergence}.
\end{proof}

\begin{remark}[Higher-dimensional explicit scheme]
The explicit scheme \eqref{eq:explicit_scheme} is defined in an arbitrary dimension, and its conservative structure, in particular Proposition~\ref{prop:explicit_mass}, is not restricted to $d=1$. However, the convergence analysis above relies on the one-dimensional cumulative formulation and on the monotonicity of the corresponding integrated operator.

For $m>1$, a direct multidimensional argument based on order preservation of the density update is not available in general. Indeed, nonlocal velocity depends on $U$ through antisymmetric weights, while numerical flux contains spatially varying mobility $U^{m-1}$. Consequently, ordering two density grid functions does not in general induce the same ordering of their explicit spatial updates, and this obstruction is not removed by decreasing the time step.

A possible approach in dimensions $d>1$ would instead combine suitable a priori estimates with compactness. In particular, positivity and mass bounds, together with discrete energy estimates and control of the discrete Caputo derivative in an appropriate negative norm, could potentially yield subsequential convergence toward a weak solution. We leave this question for future work.
\end{remark}

\section{Numerical experiments}
In all numerical examples we consider only one spatial dimension $d=1$. The weighted temporal estimates above allow a weak singularity at the initial time. The resulting $O(\tau^\alpha)$ bound is uniform over all time levels under the stated assumptions; errors measured at a fixed positive time or after the first time step are different quantities.
\begin{example}[Linear convergence test]
    Let us first consider the simplest linear example, i.e., $m=1$. Then, we end up with linear space-time fractional diffusion equation in the form
    \begin{equation}
    \begin{split}
        \partial_t^\alpha u +(-\Delta)^{1-\sigma}u=0 \quad \text{ in } \mathbb{R}\times (0,T),\\
        u(\cdot, 0) = 1 + \cos(2x)/2\quad \text{ in } \mathbb{R}. 
    \end{split}
    \end{equation}
In such a case the solution is given by
\begin{equation}\label{exactsol}
u(x, t) = 1 + \frac{1}{2} E_\alpha(-4^{1-\sigma}t^{\alpha})\cos(2x).
\end{equation}
\begin{figure}[ht!]
\centering
\includegraphics[width=0.5\textwidth, height=0.25\textheight]{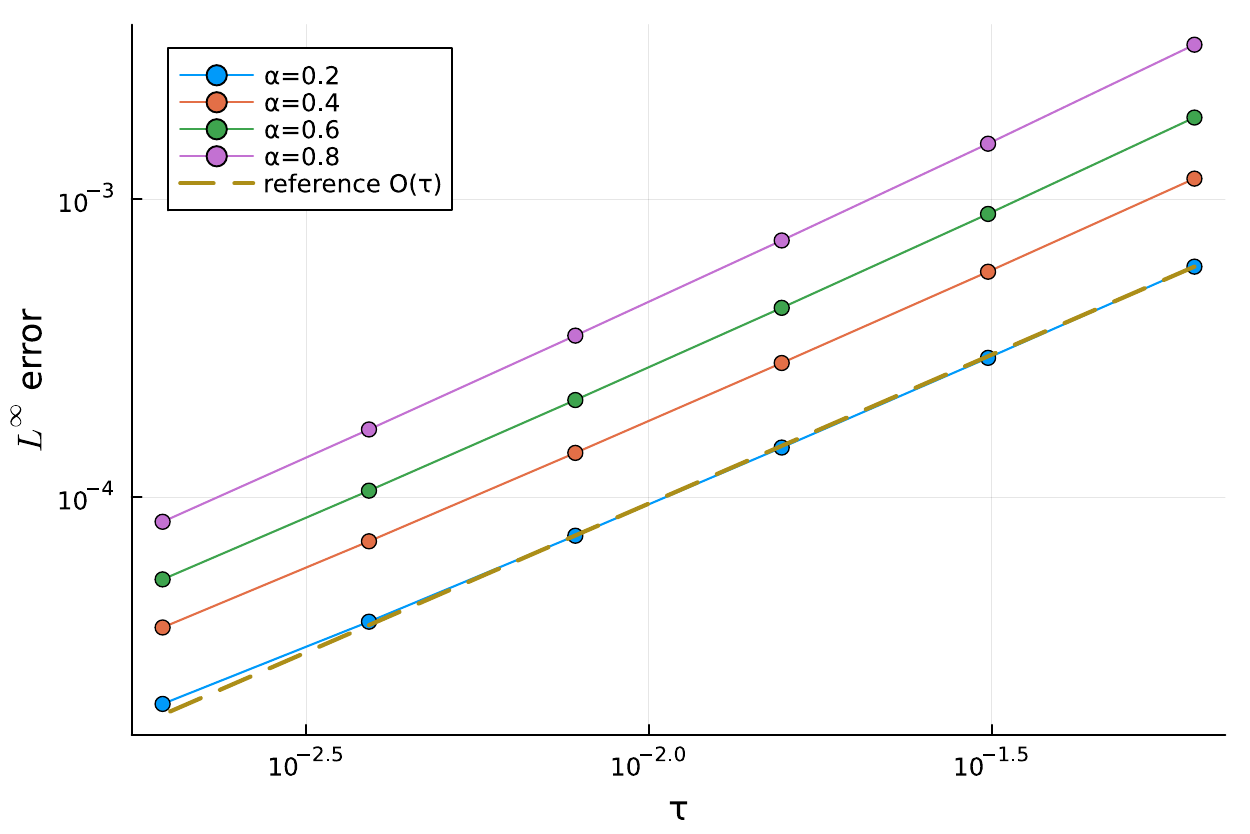}
\caption{Maximum in time error between numerical and analytical \eqref{exactsol} solution to \eqref{fpe} for $m=1$, $\sigma=0.25$, $h=10^{-2}$ and domain $[-\pi, \pi]$.}\label{fig::linearconvergenceendpoint}
\end{figure}
\begin{figure}[ht!]
\centering
\includegraphics[width=0.5\textwidth, height=0.25\textheight]{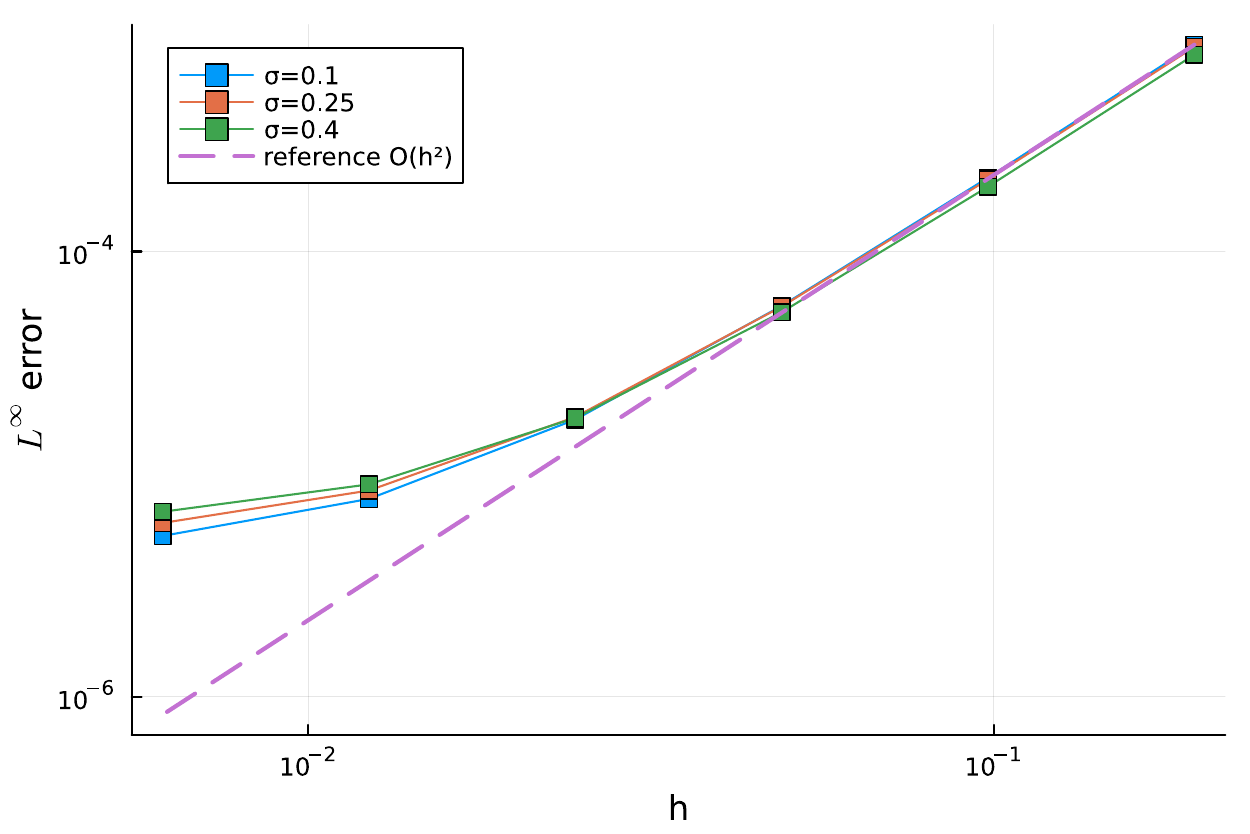}
\caption{Maximum in space error between numerical and analytical \eqref{exactsol} solution to \eqref{fpe} for $m=1$, $\sigma=0.25$, $\tau=2^{-12}$ and domain $[-\pi, \pi]$ calculated after the first time step.}\label{fig::linearconvergencefirstpoint}
\end{figure}
\end{example}
\begin{example}[Nonlinear convergence test]
   In this example we use Gaussian initial condition
   \begin{equation}
        u_0(x) = \frac{1}{\sqrt{4\pi\epsilon}}\exp(-\frac{x^2}{4\epsilon}), \quad \epsilon>0.
   \end{equation}
   We present profiles of the solution with such initial condition in the Fig. \ref{fig::profilesm2}, \ref{fig::profilesm4}. Estimated order of convergence using Aitken method for fully nonlinear equation with Gaussian initial condition solved by nonintegrated scheme are presented in Tab. \ref{tab:aitkengaussspat} and \ref{tab:aitkengausstemp}. For integrated scheme they are presented in Tab. \ref{tab:aitkengaussspatint} and \ref{tab:aitkengausstempint}.
\begin{table}[ht]
\centering
\begin{tabular}{cccc}
\hline
$\alpha\backslash\sigma$ & $0.1$ & $0.25$ & $0.4$ \\
\hline
$0.1$ & $0.95191$ & $0.95817$ & $0.96481$ \\
$0.25$ & $0.95162$ & $0.95869$ & $0.96581$ \\
$0.5$ & $0.93326$ & $0.93596$ & $0.94732$ \\
$0.6$ & $0.89075$ & $0.89368$ & $0.89327$ \\
$0.7$ & $0.82842$ & $0.82323$ & $0.82607$ \\
$0.8$ & $0.75022$ & $0.71790$ & $0.75196$ \\
$0.9$ & $0.65671$ & $0.63090$ & $0.74570$ \\
\hline
\end{tabular}
\caption{Aitken spatial convergence orders in the $L^\infty$ norm for the nonintegrated discrete scheme with $m=2$, $\varepsilon=0.01$, and fixed $\tau=2^{-12}$}
\label{tab:aitkengaussspat}
\end{table}

\begin{table}[ht]
\centering
\begin{tabular}{cccc}
\hline
$\alpha\backslash\sigma$ & $0.1$ & $0.25$ & $0.4$ \\
\hline
$0.1$ & $1.0551$ & $1.0550$ & $1.0550$ \\
$0.25$ & $1.0694$ & $1.0697$ & $1.0701$ \\
$0.5$ & $1.0983$ & $1.1022$ & $1.1047$ \\
$0.6$ & $1.1045$ & $1.1103$ & $1.1124$ \\
$0.7$ & $1.1015$ & $1.1077$ & $1.1065$ \\
$0.8$ & $1.0859$ & $1.0252$ & $0.98299$ \\
$0.9$ & $0.9942$ & $0.9365$ & $0.9127$ \\
\hline
\end{tabular}
\caption{Aitken temporal convergence orders in the $L^\infty$ norm for the nonintegrated discrete scheme with $m=2$, $\varepsilon=0.01$, and fixed $h=2^{-8}$.}
\label{tab:aitkengausstemp}
\end{table}

\begin{table}[ht]
\centering
\begin{tabular}{cccc}
\hline
$\alpha\backslash\sigma$ & $0.1$ & $0.25$ & $0.4$ \\
\hline
$0.1$ & $0.98246$ & $0.98276$ & $0.98212$ \\
$0.25$ & $0.98019$ & $0.98046$ & $0.97975$ \\
$0.5$ & $0.97151$ & $0.97227$ & $0.97222$ \\
$0.6$ & $0.96403$ & $0.96475$ & $0.96528$ \\
$0.7$ & $0.95349$ & $0.95425$ & $0.95503$ \\
$0.8$ & $0.94006$ & $0.94201$ & $0.94369$ \\
$0.9$ & $0.92601$ & $0.92894$ & $0.93301$ \\
\hline
\end{tabular}
\caption{Aitken spatial convergence orders in the $L^\infty$ norm for the integrated discrete scheme with $m=2$, $\varepsilon=0.01$, and fixed $\tau=2^{-12}$.}
\label{tab:aitkengaussspatint}
\end{table}

\begin{table}[ht]
\centering
\begin{tabular}{cccc}
\hline
$\alpha\backslash\sigma$ & $0.1$ & $0.25$ & $0.4$ \\
\hline
$0.1$ & $1.0573$ & $1.0574$ & $1.0575$ \\
$0.25$ & $1.0741$ & $1.0754$ & $1.0763$ \\
$0.5$ & $1.0868$ & $1.0931$ & $1.0958$ \\
$0.6$ & $1.0758$ & $1.0829$ & $1.0831$ \\
$0.7$ & $1.0539$ & $1.0601$ & $1.0550$ \\
$0.8$ & $1.0240$ & $1.0277$ & $1.0160$ \\
$0.9$ & $0.9919$ & $0.9918$ & $0.9750$ \\
\hline
\end{tabular}
\caption{Aitken temporal convergence orders in the $L^\infty$ norm for the integrated discrete scheme with $m=2$, $\varepsilon=0.01$, and fixed $h=2^{-8}$.}
\label{tab:aitkengausstempint}
\end{table}
\begin{figure}[ht!]
\centering
\includegraphics[width=0.5\textwidth, height=0.25\textheight]{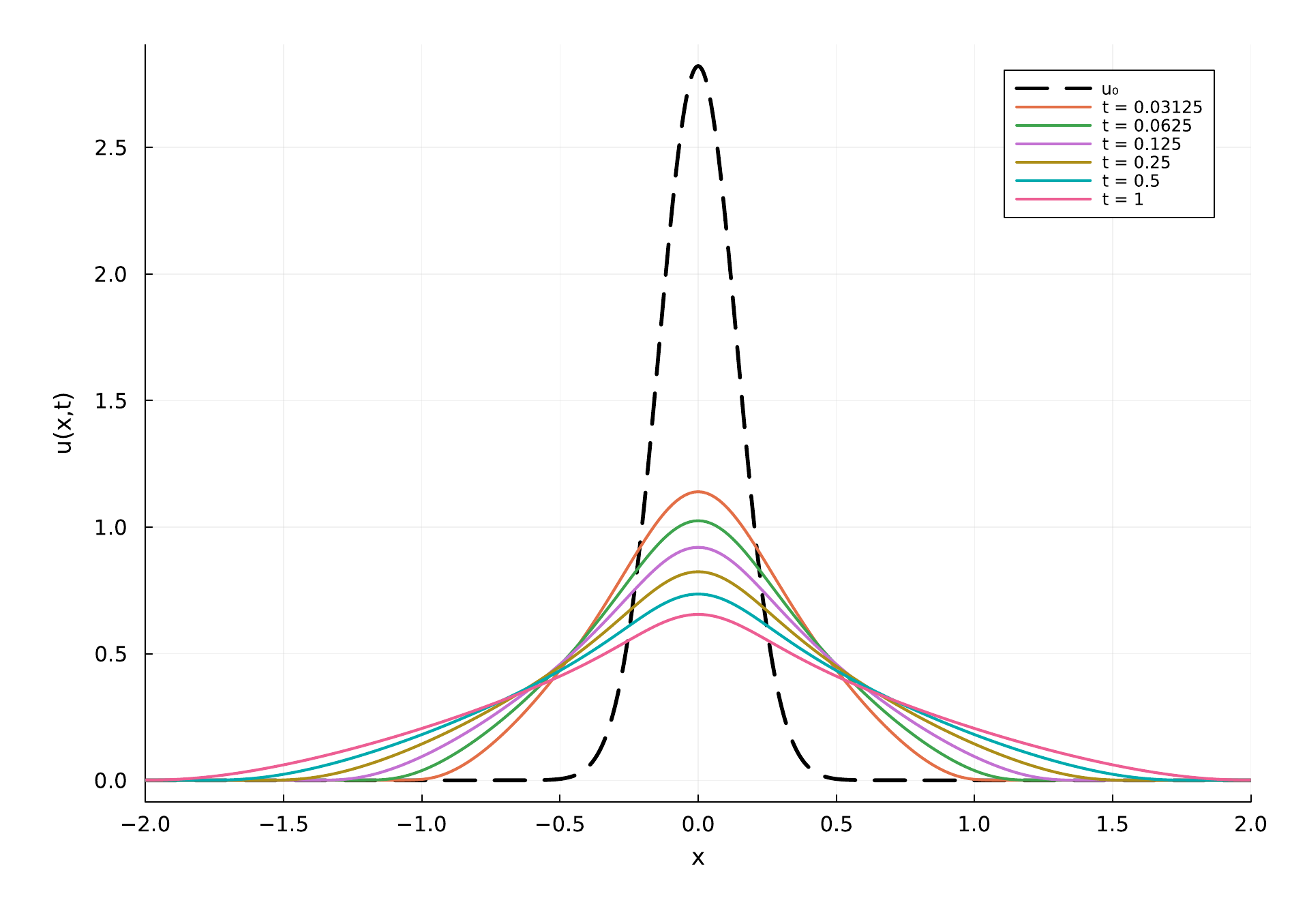}
\caption{Profiles of solutions to \eqref{eq:discrete} for $d=1$, $\alpha=0.5$, $\sigma=0.25$, $m=2$, $U_j^0=\frac{1}{\sqrt{4\pi\epsilon}}\exp(-\frac{x^2}{4\epsilon})$.}\label{fig::profilesm2}
\end{figure}
\begin{figure}[ht!]
\centering
\includegraphics[width=0.5\textwidth, height=0.25\textheight]{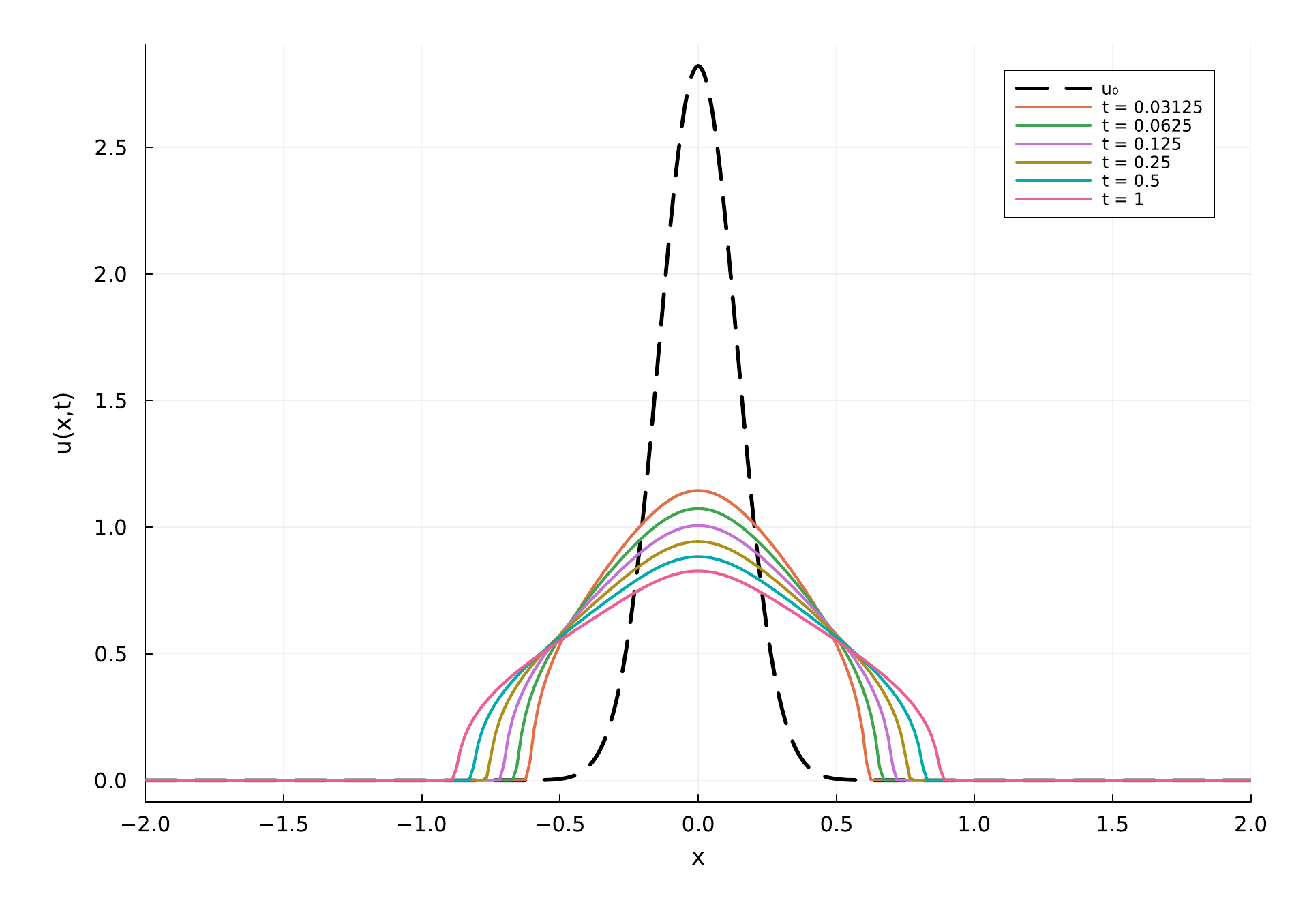}
\caption{Profiles of solutions to \eqref{eq:discrete} for $d=1$, $\alpha=0.5$, $\sigma=0.25$, $m=4$, $U_j^0=\frac{1}{\sqrt{4\pi\epsilon}}\exp(-\frac{x^2}{4\epsilon})$.}\label{fig::profilesm4}
\end{figure}
\end{example}
\begin{example}[Structural properties of the scheme]\label{ex3}
    In this example we consider compactly supported initial datum
    \begin{equation}\label{compinitdata}
        u_{0,1}(x) = \frac{15}{16}(1-x^2)_+^2.
    \end{equation}
    The constant $15/16$ is chosen, because it makes mass of the initial datum to be unit. In this experiment, we set $\alpha=0.5$, $\sigma=0.25$, $m=2$, $\tau=2^{-5}$, $h=2^{-6}$, the computational domain $[-6,6]$. In the Fig. \ref{fig::massex3} we present the evolution of mass in time and observe, that it is indeed unit and conserved. Next, we compare the density profiles and test order preservation for the cumulative scheme. To do this we introduce the initial datum
    \begin{equation}
        u_{0,2}(x) = \frac{1}{2}u_{0,1}(x-0.5) + \frac{1}{2}u_{0,1}(x-2.5).
    \end{equation}
    It is also of unit mass. Let us denote the cumulative initial data
    \begin{equation}
        v_{0,i}(x) = \int_{-\infty}^{x} u_{0,i}(y)\,dy, \quad i=1,2.
    \end{equation}
    It is easy to observe, that $u_{0,1}$ and $u_{0,2}$ cross, but $v_{0,2}\le v_{0,1}$ for all $x$. Let us denote corresponding solutions, by $u_1$, $u_2$, $v_1$, $v_2$. Figure~\ref{fig::nocomparison} shows crossing density profiles for initially unordered data; this is not a test of order preservation for the density scheme. In the Fig. \ref{fig::comparison} we observe that for the integrated scheme comparison holds. In the Fig. \ref{fig::dissipation}, we present all ingredients of the discrete dissipation inequality and then inequality itself. We observe that it holds numerically too.
    \begin{figure}[ht!]
\centering
\includegraphics[width=0.5\textwidth, height=0.25\textheight]{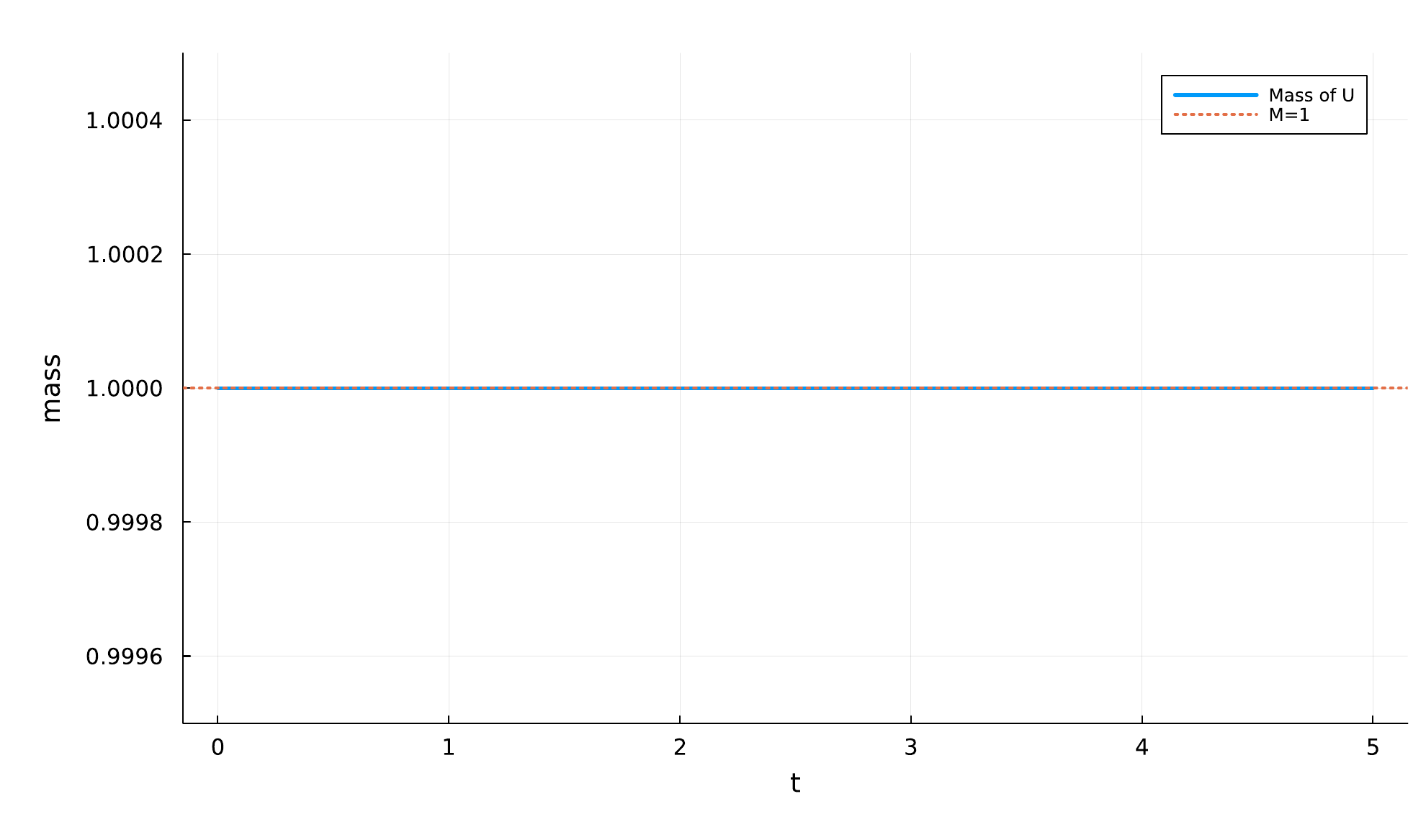}
\caption{Mass in time in Experiment 3.}\label{fig::massex3}
\end{figure}
\begin{figure}[ht!]
\centering
\includegraphics[width=0.7\textwidth, height=0.35\textheight]{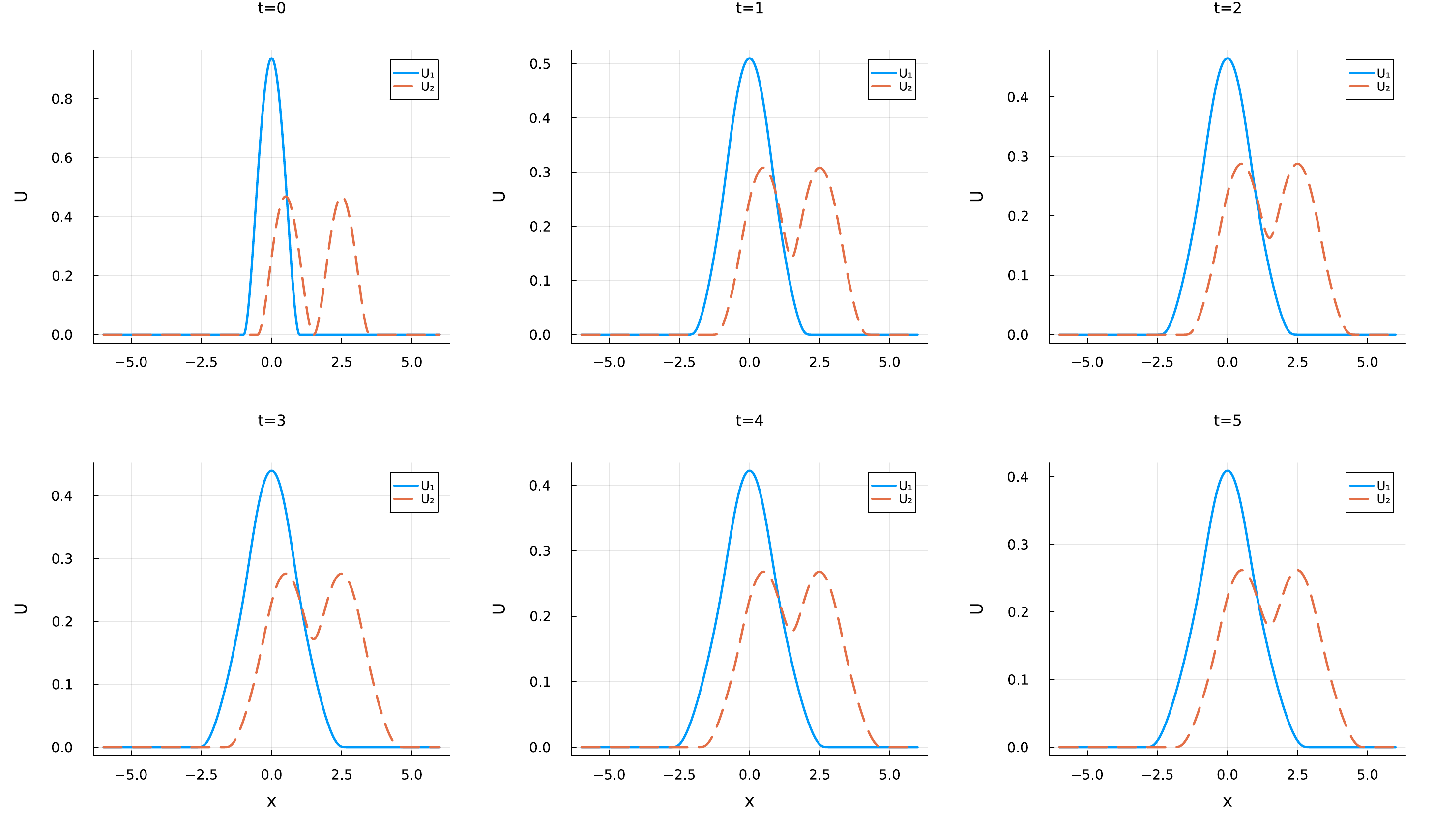}
\caption{Crossing density profiles $u_1$ and $u_2$ for initially unordered data.}\label{fig::nocomparison}
\end{figure}
\begin{figure}[ht!]
\centering
\includegraphics[width=0.7\textwidth, height=0.35\textheight]{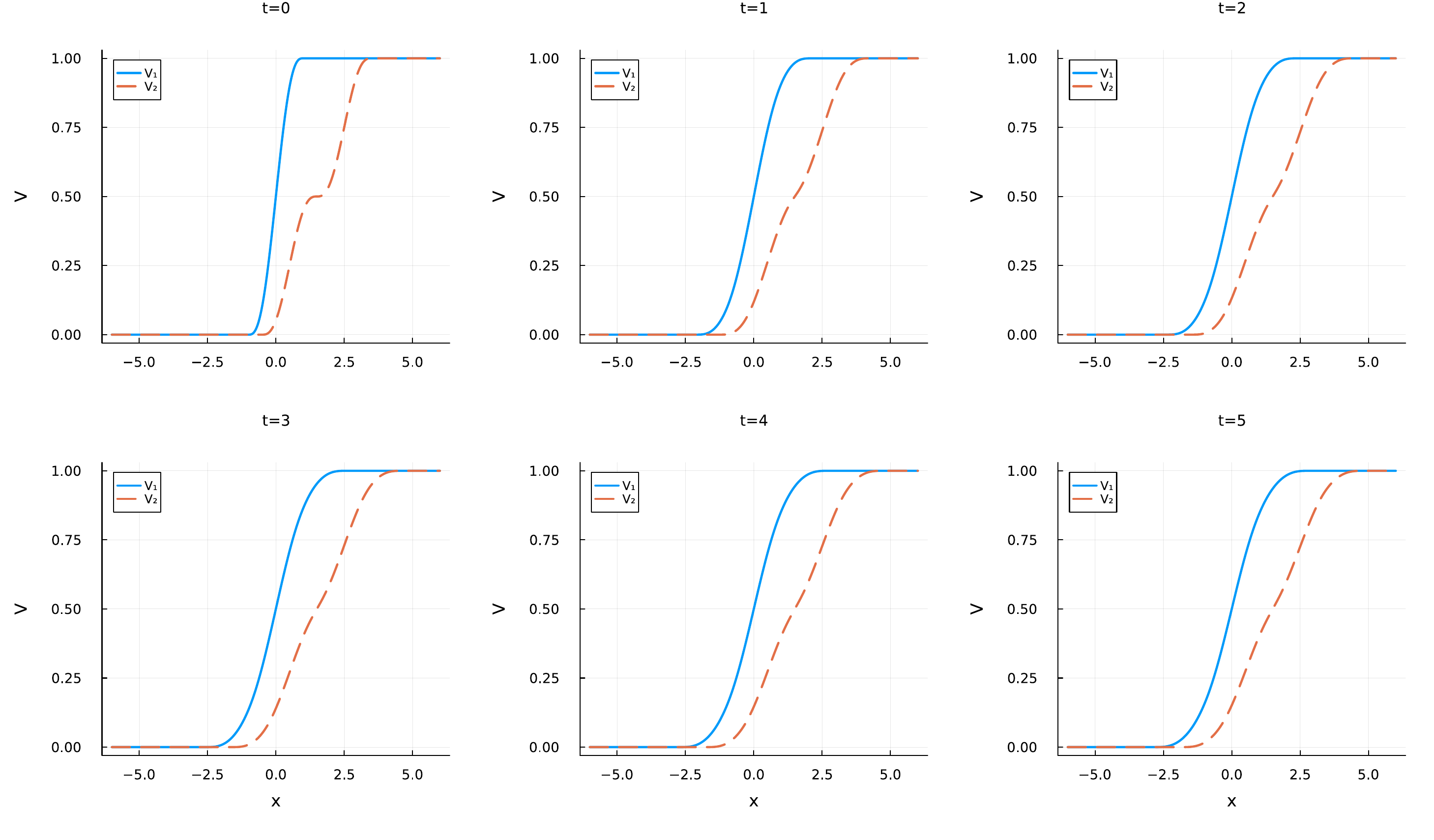}
\caption{Comparison between $v_1$ and $v_2$}\label{fig::comparison}
\end{figure}
\begin{figure}[ht!]
\centering
\includegraphics[width=0.7\textwidth, height=0.35\textheight]{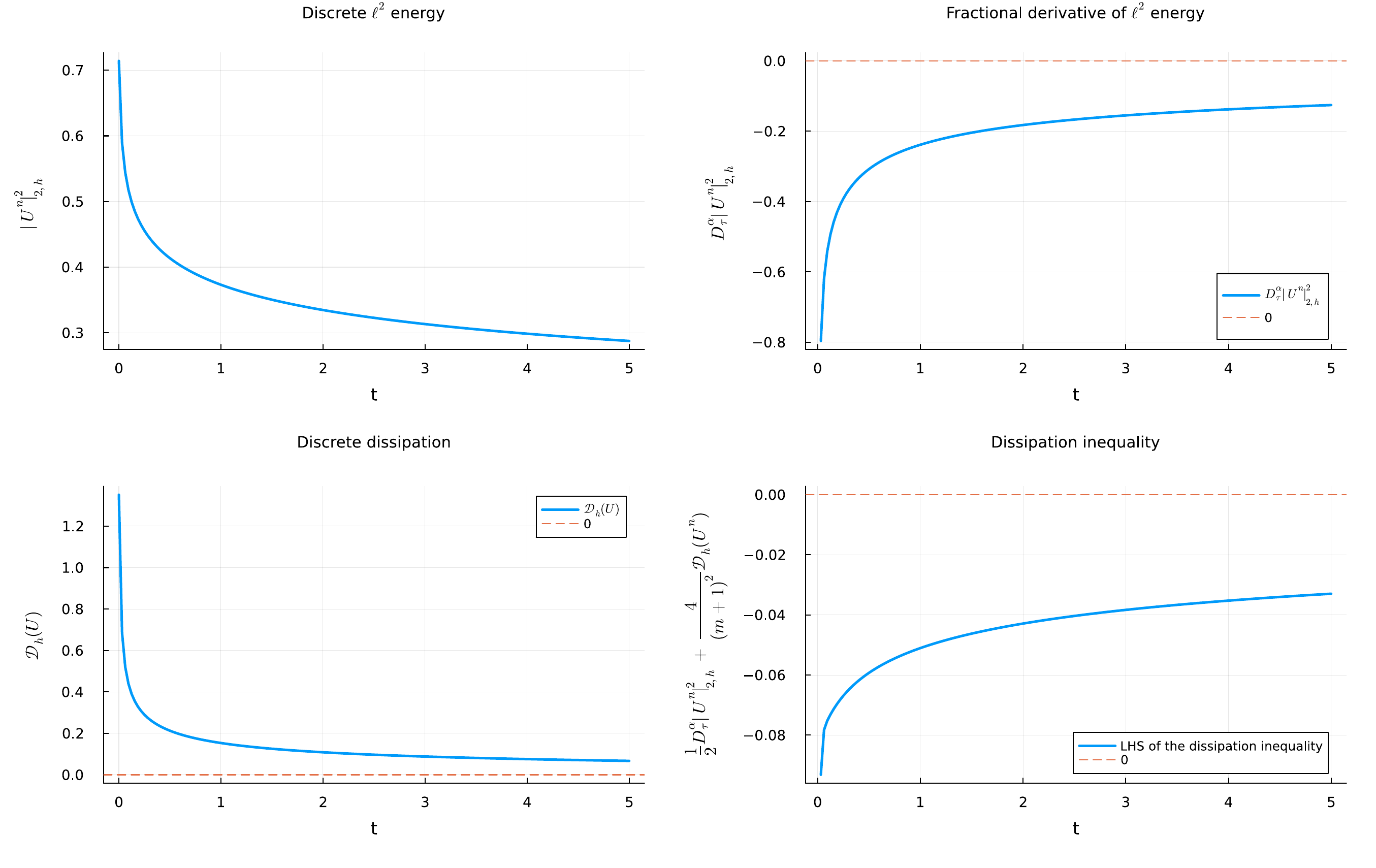}
\caption{Discrete energy, dissipation and dissipation inequality.}\label{fig::dissipation}
\end{figure}
\end{example}
\begin{example}[Explicit scheme]
    In the last example, we test the order of explicit scheme. We use the same initial condition as in Example \ref{ex3}. We again use Aitken method, since explicit solution is not known. We test the scheme for $\sigma\in\{0.25, 0.5, 0.75\}$, because we want to examine all three cases of CFL condition.
    \begin{table}[ht]
\centering
\begin{tabular}{c|ccc}
\hline
$\alpha\backslash\sigma$ & $0.25$ & $0.50$ & $0.75$ \\
\hline
$0.50$ & 1.0003 & 1.0001 & 1.0000 \\
$0.75$ & 0.9423 & 0.9423 & 0.9423 \\
\hline
\end{tabular}
\caption{Temporal Aitken convergence rates for the explicit scheme for fixed $h=2^{-6}$.}
\label{tab:explicit_temporal_rates}
\end{table}

\begin{table}[ht]
\centering
\begin{tabular}{c|ccc}
\hline
$\alpha\backslash\sigma$ & $0.25$ & $0.50$ & $0.75$ \\
\hline
$0.50$ & 0.9894 & 0.9944 & 1.0015 \\
$0.75$ & 0.9900 & 0.9947 & 1.0017 \\
\hline
\end{tabular}
\caption{Spatial Aitken convergence rates for the explicit scheme fixed $\tau=3.06\cdot 10^{-8}$.}
\label{tab:explicit_spatial_rates}
\end{table}
\end{example}
\section*{Acknowledgment}
L.~P. would like to thank F\'elix del Teso and David G\'omez Castro
(Universidad Aut\'onoma de Madrid) for many insightful discussions on the nonlocal pressure formulation and for introducing him to the method of integrated solutions.

This research was supported by the National Science Center, Poland (NCN) under the grant Sonata Bis with a number NCN 2020/38/E/ST1/00153.

The authors used ChatGPT (OpenAI) for language editing, drafting the introduction and literature overview. The authors are responsible for the final content of the manuscript, including the accuracy of the mathematical results and proofs, the precision of the references, and the numerical results.
\appendix
\section{Discretization of the fractional pressure gradient}

In this appendix, we describe the discretization of the pressure-gradient operator
\begin{equation}
Gv=\nabla(-\Delta)^{-\sigma}v.
\end{equation}
For smooth functions, this operator may be written, up to a dimensional constant, as the principal value singular integral
\begin{equation}
Gv(x)
=
c_{d,\sigma}\operatorname{P.V.}
\int_{\mathbb{R}^d}
v(z)\frac{z-x}{|z-x|^{d+2-2\sigma}}\,dz.
\end{equation}
Equivalently, since the kernel is odd, one may subtract the value $v(x)$ inside the principal value integral and write
\begin{equation}
Gv(x)
=
c_{d,\sigma}\operatorname{P.V.}
\int_{\mathbb{R}^d}
\left(v(z)-v(x)\right)
\frac{z-x}{|z-x|^{d+2-2\sigma}}\,dz.
\end{equation}
The latter form is useful for understanding the singularity of the kernel. Equivalently, after the change of variables $z=x+y$ and using the oddness of the kernel, we may write
\begin{equation}
	Gv(x)=c_{d,\sigma}\operatorname{P.V.}\int_{\mathbb{R}^d}
	\frac{v(x+y)-v(x)}{|y|^{d+1-2\sigma}}\frac{y}{|y|}\,dy.
\end{equation}
The subtraction of $v(x)$ makes the integrand locally integrable at the origin for every $\sigma\in(0,1)$. If $0<\sigma<1/2$ and $v$ is bounded, the integral is also absolutely convergent at infinity. For $\sigma\geq1/2$, a symmetric interpretation at infinity or an additional decay assumption is required.
\subsection{Piecewise constant approximation}
We want to approximate this operator at the point
\begin{equation}
x_m=mh, \quad m=j+\frac{1}{2}e_i,
\end{equation}
where the pressure gradient is evaluated at the cell interface in the $i$-th direction. Assume that each point $x_k=kh$, $k\in\mathbb{Z}^d$, lies in the middle of the cell
\begin{equation}
C_k = \prod_{\ell=1}^d \left[\left(k_\ell-\frac{1}{2}\right)h, \left(k_\ell+\frac{1}{2}\right)h\right].
\end{equation}
We approximate $v(z)$ by piecewise constant functions
\begin{equation}
v(z) \approx\sum_{k\in\mathbb{Z}^d}U_k \chi_{C_k}(z).
\end{equation}
Formally, this gives
\begin{equation}
Gv(x_m) \approx
c_{d,\sigma}\sum_{k\in\mathbb{Z}^d}U_k
\operatorname{P.V.}\int_{C_k}
\frac{z-x_m}{|z-x_m|^{d+2-2\sigma}}\,dz.
\end{equation}
Let
\begin{equation}
\beta=k-m, \quad \eta=\frac{z-x_m}{h}.
\end{equation}
Then the cell $C_k$ is transformed into the unit cell
\begin{equation}
C_\beta =
\prod_{\ell=1}^d
\left[\beta_\ell-\frac{1}{2},\beta_\ell+\frac{1}{2}\right],
\end{equation}
and
\begin{equation}
\begin{aligned}
	\int_{C_k}
	\frac{z-x_m}{|z-x_m|^{d+2-2\sigma}}\,dz
	&=
	\int_{C_\beta}
	\frac{h\eta}{|h\eta|^{d+2-2\sigma}}h^d\,d\eta \\
	&=
	h^{2\sigma-1}
	\int_{C_\beta}
	\frac{\eta}{|\eta|^{d+2-2\sigma}}\,d\eta.
\end{aligned}
\end{equation}
Thus, for the regular part of the quadrature, we obtain
\begin{equation}
Gv(x_m)
\approx
h^{2\sigma-1}
\sum_{k\in\mathbb{Z}^d}
w_{k-m}U_k,
\end{equation}
where
\begin{equation}
w_\beta
=
c_{d,\sigma}
\operatorname{P.V.}
\int_{C_\beta}
\frac{\eta}{|\eta|^{d+2-2\sigma}}\,d\eta.
\end{equation}
Equivalently, if the main text uses the notation,
\begin{equation}
(G_hU)_m=\sum_{k\in\mathbb{Z}^d}g_{k-m}(h)U_k,
\end{equation}
then the two sets of weights are related by
\begin{equation}
g_\beta(h)=h^{2\sigma-1}w_\beta.
\end{equation}
The weights inherit the oddness of the continuous kernel. Namely, whenever the principal value is well-defined,
\begin{equation}
w_{-\beta}=-w_\beta.
\end{equation}
Consequently, constants are annihilated by the corresponding discrete operator, which is the discrete analogue of the principal value cancellation. In the notation of the main text, this is expressed as
\begin{equation}
\sum_{\beta}g_\beta(h)=0,
\end{equation}
where the sum is understood in the principal value sense.

From the Taylor expansion around the center of the cell $C_\beta$, for large $|\beta|$, we have
\begin{equation}
\int_{C_\beta}
\frac{\eta}{|\eta|^{d+2-2\sigma}}\,d\eta
=
\frac{\beta}{|\beta|^{d+2-2\sigma}}
+
\mathcal{O}\left(|\beta|^{-(d+3-2\sigma)}\right).
\end{equation}
This yields the decay estimate
\begin{equation}	
|w_\beta|
\leq
C|\beta|^{-(d+1-2\sigma)},
\end{equation}
and therefore
\begin{equation}
|g_\beta(h)|
\leq
Ch^{2\sigma-1}|\beta|^{-(d+1-2\sigma)}.
\end{equation}
We emphasize that, in several dimensions, this asymptotic decay and the oddness of the kernel do not by themselves imply the sign condition used in the maximum principle. Therefore, in the analysis we assume that the discrete weights are chosen so that the induced coefficients
\begin{equation}
c^i_\beta(h)
=
g^i_{\beta-\frac{1}{2}e_i}(h)
-
g^i_{\beta+\frac{1}{2}e_i}(h)
\end{equation}
satisfy
\begin{equation}
\frac1h\sum_{i=1}^d c^i_\beta(h)\geq0
\quad \text{for }\beta\neq0,
\quad
\sum_{\beta}c^i_\beta(h)=0.
\end{equation}
This is the structural condition required for the discrete maximum principle.
In one space dimension, the regular cell weights can be computed explicitly away from the singular cell. For $\beta>0$ and $\sigma\neq 1/2$,
\begin{equation}
w_\beta
=
c_{1,\sigma}
\frac{
	\left(\beta+\frac{1}{2}\right)^{2\sigma-1}
	-
	\left(\beta-\frac{1}{2}\right)^{2\sigma-1}
}{2\sigma-1}.
\end{equation}
For $\beta<0$, the weights are defined by antisymmetry,
\begin{equation}
w_\beta=-w_{-\beta}.
\end{equation}
In the borderline case $\sigma=1/2$, the corresponding formula becomes logarithmic: for $\beta>0$,
\begin{equation}
w_\beta
=
c_{1,1/2}
\log\left(
\frac{\beta+\frac{1}{2}}{\beta-\frac{1}{2}}
\right).
\end{equation}

For $\frac12<\sigma<1$, the kernel
\begin{equation}
	K(y)=\frac{y}{|y|^{d+2-2\sigma}}
\end{equation}
is locally integrable. Hence, the cell integrals defining the piecewise constant weights are well-defined, including for the cells adjacent to the evaluation point. On the other hand, for $0<\sigma\leq\frac12$, these cell integrals are singular at the interface, and the piecewise constant construction requires an additional principal-value, finite-part, or other cancellation-preserving regularization. We do not consider such a regularization here.

\begin{theorem}[Piecewise truncation error]
\label{thm:piecewise_truncation}
Let $\frac12<\sigma<1$, and let the discrete approximation be evaluated at the cell-interface point
\begin{equation}
	x_m=mh, \quad m=j+\frac12e_i,
\end{equation}
using the values of the cell-midpoint $V_k=v(x_k)$. Assume that $v\in C_c^2(\mathbb{R}^d)$. Then, for
\begin{equation}
	E_h(x_m)=Gv(x_m)-(G_hV)_m,
\end{equation}
we have
\begin{equation}
	|E_h(x_m)|\leq C_vh^{2\sigma}.
\end{equation}
\end{theorem}

\begin{proof}
Let
\begin{equation}
	K(y)=\frac{y}{|y|^{d+2-2\sigma}}.
\end{equation}
Writing $\beta=k-m\in\Gamma_i=\mathbb{Z}^d-\frac12e_i$, the quadrature error can be expressed as
\begin{equation}
	E_h(x_m)=c_{d,\sigma}\sum_{\beta\in\Gamma_i}\int_{C_\beta^h}K(y)\left(v(x_m+y)-v(x_m+\beta h)\right)\,dy,
\end{equation}
where
\begin{equation}
	C_\beta^h=\beta h+\left[-\frac{h}{2},\frac{h}{2}\right]^d.
\end{equation}

Since $\Gamma_i$ is symmetric with respect to the origin and $K$ is odd, only the odd part of $v$ around $x_m$ contributes to the error. Define
\begin{equation}
	v_{\mathrm{odd}}(y)=\frac12\left(v(x_m+y)-v(x_m-y)\right).
\end{equation}
Then
\begin{equation}
	E_h(x_m)=c_{d,\sigma}\sum_{\beta\in\Gamma_i}\int_{C_\beta^h}K(y)\left(v_{\mathrm{odd}}(y)-v_{\mathrm{odd}}(\beta h)\right)\,dy.
\end{equation}

Fix a sufficiently large integer $M$, independent of $h$, and split the sum into
\begin{equation}
	A=\{\beta\in\Gamma_i:\|\beta\|_\infty\leq M\}, \quad B=\Gamma_i\setminus A.
\end{equation}
We write
\begin{equation}
	E_h=E_A+E_B.
\end{equation}

For $\beta\in A$, Taylor expansion around the origin gives
\begin{equation}
	v_{\mathrm{odd}}(y)=\nabla v(x_m)\cdot y+R(y), \quad |R(y)|\leq C_v|y|^2.
\end{equation}
Hence,
\begin{equation}
	v_{\mathrm{odd}}(y)-v_{\mathrm{odd}}(\beta h)
	=\nabla v(x_m)\cdot(y-\beta h)+R(y)-R(\beta h).
\end{equation}
Since $A$ is finite and $|y|\leq C_Mh$ on the corresponding cells,
\begin{equation}
	|R(y)-R(\beta h)|\leq C_vh^2.
\end{equation}
Therefore,
\begin{equation}
	|E_A|\leq C_v\sum_{\beta\in A}\int_{C_\beta^h}|K(y)||y-\beta h|\,dy
	+C_vh^2\int_{|y|\leq C_Mh}|K(y)|\,dy.
\end{equation}
Because $\sigma>\frac12$, both terms are finite. Rescaling $y=h\eta$ gives
\begin{equation}
	\sum_{\beta\in A}\int_{C_\beta^h}|K(y)||y-\beta h|\,dy\leq C_Mh^{2\sigma}
\end{equation}
and
\begin{equation}
	h^2\int_{|y|\leq C_Mh}|K(y)|\,dy\leq C_Mh^{2\sigma+1}.
\end{equation}
Consequently,
\begin{equation}
	|E_A|\leq C_vh^{2\sigma}.
\end{equation}

We next consider the far-field contribution. For $\beta\in B$, write
\begin{equation}
	y=\beta h+\xi, \quad \xi\in\left[-\frac{h}{2},\frac{h}{2}\right]^d.
\end{equation}
Taylor expansion yields
\begin{equation}
	v_{\mathrm{odd}}(\beta h+\xi)-v_{\mathrm{odd}}(\beta h)=\nabla v_{\mathrm{odd}}(\beta h)\cdot\xi+R_{\beta}(\xi),
\end{equation}
where
\begin{equation}
	|R_{\beta}(\xi)|\leq C_v|\xi|^2.
\end{equation}
Since the cell is symmetric,
\begin{equation}
	\int_{\left[-\frac{h}{2},\frac{h}{2}\right]^d}K(\beta h)\nabla v_{\mathrm{odd}}(\beta h)\cdot\xi\,d\xi=0.
\end{equation}
Moreover,
\begin{equation}
	|\nabla K(y)|\leq C|y|^{-d-2+2\sigma}.
\end{equation}
Choosing $M$ sufficiently large, $|y|$ is comparable to $|\beta|h$ throughout $C_\beta^h$, and therefore
\begin{equation}
	\left|\int_{C_\beta^h}\left(K(y)-K(\beta h)\right)\nabla v_{\mathrm{odd}}(\beta h)\cdot(y-\beta h)\,dy\right|
	\leq C_vh^{2\sigma}|\beta|^{-d-2+2\sigma}.
\end{equation}
Since
\begin{equation}
	\sum_{\beta\in B}|\beta|^{-d-2+2\sigma}<\infty
\end{equation}
for $\sigma<1$, these terms contribute at most $C_vh^{2\sigma}$.

The remainder terms contain an additional factor $|y-\beta h|^2$. Since $v\in C_c^2(\mathbb{R}^d)$ and $K$ is locally integrable for $\sigma>\frac12$, their total contribution is bounded by
\begin{equation}
	Ch^2.
\end{equation}
Consequently,
\begin{equation}
	|E_B|\leq C_v\left(h^{2\sigma}+h^2\right)\leq C_vh^{2\sigma},
\end{equation}
because $\frac12<\sigma<1$. Combining the estimates for $E_A$ and $E_B$ completes the proof.
\end{proof}

\begin{remark}
For $0<\sigma\leq\frac12$, the cells adjacent to the interface contain a non-integrable singularity of the kernel. A piecewise constant approximation can still be considered after introducing a suitable cancellation-preserving regularization, but its construction and consistency analysis are different from the one above and are not pursued here.
\end{remark}

\subsection{Semigroup approach}

We can also discretize the fractional pressure gradient using the semigroup approach. Let
\begin{equation}
	(\Delta_h\psi)_j=\frac{1}{h^2}\sum_{i=1}^d\left(\psi_{j+e_i}+\psi_{j-e_i}-2\psi_j\right)
\end{equation}
be the standard discrete Laplacian. The associated discrete heat semigroup satisfies
\begin{equation}
	(e^{t\Delta_h}\psi)_j=\sum_{k\in\mathbb{Z}^d}G\left(j-k,\frac{t}{h^2}\right)\psi_k,
\end{equation}
where
\begin{equation}\label{eq:prod_form_g}
	G(\beta,r)=e^{-2dr}\prod_{i=1}^dI_{|\beta_i|}(2r).
\end{equation}
Here, $I_\nu$ denotes the modified Bessel function of the first kind. The discrete heat kernel satisfies
\begin{equation}
	G(\beta,r)\geq0, \quad G(-\beta,r)=G(\beta,r), \quad \sum_{\beta\in\mathbb{Z}^d}G(\beta,r)=1.
\end{equation}

For the $i$-th component of the pressure gradient, evaluated at the interface $x_{j+\frac12e_i}$, we use the midpoint difference
\begin{equation}
	D_{h,i}^+\psi_j=\frac{\psi_{j+e_i}-\psi_j}{h}.
\end{equation}
We define
\begin{equation}\label{eq:semigroup_gradient}
	(G_{h,i}U)_{j+\frac12e_i}=\frac{1}{\Gamma(\sigma)}\int_0^\infty D_{h,i}^+\left(e^{t\Delta_h}U\right)_j\frac{dt}{t^{1-\sigma}}.
\end{equation}
This is the discrete analogue of the semigroup representation
\begin{equation}
	\nabla(-\Delta)^{-\sigma}v=\frac{1}{\Gamma(\sigma)}\int_0^\infty\nabla e^{t\Delta}v\frac{dt}{t^{1-\sigma}}.
\end{equation}

Let
\begin{equation}
	\Gamma_i=\mathbb{Z}^d-\frac12e_i.
\end{equation}
Using the heat-kernel representation in \eqref{eq:semigroup_gradient} and the change of variables $t=h^2r$, we obtain
\begin{equation}
	(G_{h,i}U)_{j+\frac12e_i}=\sum_{k\in\mathbb{Z}^d}g_{k-j-\frac12e_i}^i(h)U_k,
\end{equation}
where, for $\gamma\in\Gamma_i$,
\begin{equation}\label{eq:semigroup_weights}
	g_\gamma^i(h)=\frac{h^{2\sigma-1}}{\Gamma(\sigma)}\int_0^\infty\left[G\left(\gamma-\frac12e_i,r\right)-G\left(\gamma+\frac12e_i,r\right)\right]r^{\sigma-1}\,dr.
\end{equation}
The symmetry of the discrete heat kernel immediately gives
\begin{equation}
	g_{-\gamma}^i(h)=-g_\gamma^i(h).
\end{equation}
Consequently, constants are annihilated by $G_{h,i}$, with the corresponding sum of the weights understood in the symmetric sense whenever it is not absolutely convergent.

Under the corresponding discrete heat-kernel difference estimates, the weights satisfy
\begin{equation}
	|g_\gamma^i(h)|\leq Ch^{2\sigma-1}|\gamma|^{-(d+1-2\sigma)}.
\end{equation}
This is the decay assumption used in the multidimensional structural analysis. In the one-dimensional case needed for the convergence theory, the same estimate follows directly from the explicit formula below.

Next, we verify that the semigroup weights satisfy the sign and cancellation properties required in the main text. Notice that the sign condition is imposed on the sum over the spatial directions.

\begin{proposition}[Structural properties of the semigroup weights]
\label{prop:semigroup_multidimensional}
Let $d\geq1$, $0<\sigma<1$, and let $g_\gamma^i(h)$ be given by \eqref{eq:semigroup_weights}. Then, for every fixed $h>0$ and $i=1,\ldots,d$,
\begin{equation}
\sum_{\beta\in\mathbb{Z}^d}|c_\beta^i(h)|<\infty,
\quad
\sum_{\beta\in\mathbb{Z}^d}c_\beta^i(h)=0.
\end{equation}
Moreover, putting $s=1-\sigma$, for $\beta\neq0$ we have
\begin{equation}\label{eq:semigroup_aggregate_kernel}
C_\beta(h)=\frac1h\sum_{i=1}^d c_\beta^i(h)
=\frac{h^{-2s}}{|\Gamma(-s)|}
\int_0^\infty G(\beta,r)\frac{dr}{r^{1+s}}\geq0.
\end{equation}
Consequently, $C_{-\beta}(h)=C_\beta(h)$, $\sum_{\beta\neq0}C_\beta(h)<\infty$, and
\begin{equation}
L_h=-\sum_{i=1}^dD_{h,i}^-G_{h,i}=(-\Delta_h)^s.
\end{equation}
The last identity holds, in particular, for bounded grid functions for which the pressure-gradient operators are well-defined.
\end{proposition}
\begin{proof}
To start, define the second difference on the unit lattice by
\begin{equation}
(\delta_iG)(\beta,r)=G(\beta-e_i,r)-2G(\beta,r)+G(\beta+e_i,r).
\end{equation}
It follows from \eqref{eq:semigroup_weights} that
\begin{equation}
c_\beta^i(h)=\frac{h^{2\sigma-1}}{\Gamma(\sigma)}
\int_0^\infty(\delta_iG)(\beta,r)r^{\sigma-1}\,dr.
\end{equation}
We first check absolute summability. Let $q_k(r)=e^{-2r}I_{|k|}(2r)$ be the one-dimensional heat kernel. Since $q_k(r)$ is symmetric and non-increasing in $|k|$, its first difference has $\ell^1$ norm $2q_0(r)$. Furthermore,
\begin{equation}
q_0(r)=\frac1{2\pi}\int_{-\pi}^{\pi}e^{-2r(1-\cos\theta)}\,d\theta
\leq C(1+r)^{-1/2}.
\end{equation}
Using the product formula \eqref{eq:prod_form_g} for $G$, its semigroup property, and the convolution inequality in $\ell^1$, we therefore obtain
\begin{equation}
\|\delta_iG(\cdot,r)\|_{\ell^1}
\leq\|D_i^-G(\cdot,r/2)\|_{\ell^1}
\|D_i^+G(\cdot,r/2)\|_{\ell^1}
\leq C\min\{1,r^{-1}\},
\end{equation}
where $D_i^\pm$ denote first differences on the unit lattice. Consequently,
\begin{equation}
\sum_\beta|c_\beta^i(h)|
\leq Ch^{2\sigma-1}\int_0^\infty
\min\{1,r^{-1}\}r^{\sigma-1}\,dr<\infty.
\end{equation}
This also justifies interchanging the sum and the integral. Since $\sum_\beta(\delta_iG)(\beta,r)=0$, the cancellation property follows.

Next, the discrete heat equation gives
\begin{equation}
\sum_{i=1}^d(\delta_iG)(\beta,r)=\partial_rG(\beta,r).
\end{equation}
Hence, for $\beta\neq0$,
\begin{equation}
C_\beta(h)=\frac{h^{2\sigma-2}}{\Gamma(\sigma)}
\int_0^\infty\partial_rG(\beta,r)r^{\sigma-1}\,dr.
\end{equation}
Notice that $G(\beta,r)=O(r^{|\beta|_1})$ as $r\rightarrow0$, while $0\leq G(\beta,r)\leq1$. Therefore, the boundary terms vanish when integrating by parts, and we obtain
\begin{equation}
C_\beta(h)=\frac{(1-\sigma)h^{2\sigma-2}}{\Gamma(\sigma)}
\int_0^\infty G(\beta,r)r^{\sigma-2}\,dr.
\end{equation}
Using $s=1-\sigma$ and $|\Gamma(-s)|=\Gamma(1-s)/s$ gives \eqref{eq:semigroup_aggregate_kernel}. Positivity and symmetry now follow directly from the corresponding properties of $G$. Summability follows from the directional estimate proved above.

Finally, the semigroup formula for the fractional discrete Laplacian is
\begin{equation}
(-\Delta_h)^sU_j=\frac1{|\Gamma(-s)|}
\int_0^\infty\left(U_j-(e^{t\Delta_h}U)_j\right)\frac{dt}{t^{1+s}}.
\end{equation}
Substituting the heat kernel and changing variables $t=h^2r$ yields
\begin{equation}
(-\Delta_h)^sU_j=\sum_{\beta\neq0}C_\beta(h)(U_j-U_{j+\beta})=(L_hU)_j.
\end{equation}
The interchange is justified by $\sum_{\beta\neq0}C_\beta(h)<\infty$ for bounded $U$. This completes the proof.
\end{proof}

We now specialize to one space dimension, where a closed formula is available.

\begin{lemma}[One-dimensional semigroup weights]
\label{lem:semigroup_weights_1d}
Let $d=1$ and $\gamma\in\mathbb{Z}+\frac12$. For $\gamma>0$,
\begin{equation}\label{eq:semigroup_weights_1d}
	g_\gamma(h)=\frac{2^{1-2\sigma}\Gamma(\frac32-\sigma)}{\sqrt{\pi}\Gamma(\sigma)}h^{2\sigma-1}\frac{\Gamma(\gamma-\frac12+\sigma)}{\Gamma(\gamma+\frac32-\sigma)},
\end{equation}
and
\begin{equation}
	g_{-\gamma}(h)=-g_\gamma(h).
\end{equation}
Moreover, $g_\gamma(h)>0$ for $\gamma>0$, and the sequence $g_\gamma(h)$ is strictly decreasing as $\gamma$ increases over the positive half-integers.
\end{lemma}

\begin{proof}
Formula \eqref{eq:semigroup_weights_1d} follows from the explicit kernel representation of the fractional powers of the one-dimensional discrete Laplacian. Positivity is immediate from the Gamma-function formula. Furthermore,
\begin{equation}
	\frac{g_{\gamma+1}(h)}{g_\gamma(h)}=\frac{\gamma-\frac12+\sigma}{\gamma+\frac32-\sigma}<1
\end{equation}
for $0<\sigma<1$, which proves the monotonicity.
\end{proof}

We next identify the operator induced by the differences of the pressure-gradient weights. Define
\begin{equation}
c_\beta(h)=g_{\beta-\frac12}(h)-g_{\beta+\frac12}(h), \quad \beta\in\mathbb{Z}.
\end{equation}
Then $c_{-\beta}(h)=c_\beta(h)$ and
\begin{equation}
c_\beta(h)>0, \quad \beta\neq0.
\end{equation}

\begin{proposition}[Identification with the fractional discrete Laplacian]
\label{prop:semigroup_fractional_laplacian}
Let
\begin{equation}
	s=1-\sigma.
\end{equation}
Then, for $\beta\neq0$,
\begin{equation}
	\frac{c_\beta(h)}{h}=\frac{4^s\Gamma(\frac12+s)}{\sqrt{\pi}|\Gamma(-s)|}\frac{1}{h^{2s}}\frac{\Gamma(|\beta|-s)}{\Gamma(|\beta|+1+s)}.
\end{equation}
Consequently,
\begin{equation}
	(\mathcal{A}_hV)_j=\sum_{\beta\neq0}\frac{c_\beta(h)}{h}(V_j-V_{j+\beta})=(-\Delta_h)^sV_j.
\end{equation}
\end{proposition}

\begin{proof}
It is sufficient to consider $\beta\geq1$. Using \eqref{eq:semigroup_weights_1d} and $s=1-\sigma$, we have
\begin{equation}
	g_{\beta-\frac12}(h)=C_{\sigma}h^{2\sigma-1}\frac{\Gamma(\beta-s)}{\Gamma(\beta+s)}
\end{equation}
and
\begin{equation}
	g_{\beta+\frac12}(h)=C_{\sigma}h^{2\sigma-1}\frac{\Gamma(\beta+1-s)}{\Gamma(\beta+1+s)},
\end{equation}
where
\begin{equation}
	C_{\sigma}=\frac{2^{1-2\sigma}\Gamma(\frac32-\sigma)}{\sqrt{\pi}\Gamma(\sigma)}.
\end{equation}
Using $\Gamma(z+1)=z\Gamma(z)$, we obtain
\begin{equation}
	c_\beta(h)=2sC_{\sigma}h^{2\sigma-1}\frac{\Gamma(\beta-s)}{\Gamma(\beta+1+s)}.
\end{equation}
Since
\begin{equation}
	|\Gamma(-s)|=\frac{\Gamma(1-s)}{s},
\end{equation}
this gives
\begin{equation}
	\frac{c_\beta(h)}{h}=\frac{4^s\Gamma(\frac12+s)}{\sqrt{\pi}|\Gamma(-s)|}\frac{1}{h^{2s}}\frac{\Gamma(\beta-s)}{\Gamma(\beta+1+s)}.
\end{equation}
By symmetry, the same formula holds with $\beta$ replaced by $|\beta|$. The resulting kernel is precisely the kernel of the fractional discrete Laplacian $(-\Delta_h)^s$, which proves the claim.
\end{proof}

\begin{corollary}[Consistency in one dimension]
\label{cor:semigroup_consistency}
Let $s=1-\sigma$ and let $r_h^{1/2}v$ denote the restriction of $v$ to the shifted grid,
\begin{equation}
	(r_h^{1/2}v)_j=v(x_{j+\frac12}).
\end{equation}
For sufficiently smooth bounded functions $v$,
\begin{equation}
	\|\mathcal{A}_hr_h^{1/2}v-r_h^{1/2}(-\Delta)^sv\|_\infty\leq C_v\eta_{\mathcal A}(h),
\end{equation}
where
\begin{equation}
	\eta_{\mathcal A}(h)=
	\begin{cases}
		h^{2\sigma}, & 0<\sigma<\frac12,\\
		h^\theta, \quad 0<\theta<1, & \sigma=\frac12,\\
		h^{2\sigma-1}, & \frac12<\sigma<1.
	\end{cases}
\end{equation}
For example, it is sufficient to assume $v\in C^{1,1}(\mathbb{R})$ when $0<\sigma<1/2$, $v\in C^{1,\theta}(\mathbb{R})$ when $\sigma=1/2$, and $v\in C^{0,1}(\mathbb{R})$ when $1/2<\sigma<1$.
\end{corollary}

\begin{proof}
By Proposition~\ref{prop:semigroup_fractional_laplacian},
\begin{equation}
	\mathcal{A}_h=(-\Delta_h)^s.
\end{equation}
The result therefore follows from the consistency estimates for the fractional discrete Laplacian proved in Theorem~1.1 of \cite{ciaurri2015fractional}. Since both continuous and discrete fractional Laplacians are translation invariant, the same estimates hold on the shifted grid $h(\mathbb{Z}+\frac12)$.
\end{proof}

\begin{remark}
The above consistency result is the one used in the one-dimensional convergence analysis. Therefore, a separate second-order consistency estimate for the pressure-gradient operator $G_h$ is not required.
\end{remark}

\bibliographystyle{plain}
\bibliography{bibliography}

@article{ToppYangari2017,
  author  = {Topp, Erwin and Yangari, Miguel},
  title   = {Existence and uniqueness for parabolic problems with Caputo time derivative},
  journal = {Journal of Differential Equations},
  year    = {2017},
  volume  = {262},
  number  = {12},
  pages   = {6018--6046},
  doi     = {10.1016/j.jde.2017.02.024}
}

@article{Allenetal2017,
  author  = {Allen, Mark and Caffarelli, Luis and Vasseur, Alexis},
  title   = {Porous medium flow with both a fractional potential pressure and fractional time derivative},
  journal = {Chinese Annals of Mathematics, Series B},
  year    = {2017},
  volume  = {38},
  number  = {1},
  pages   = {45--82},
  doi     = {10.1007/s11401-016-1063-4}
}

@article{chasseigne2017nonlocal,
  title={On nonlocal quasilinear equations and their local limits},
  author={Chasseigne, Emmanuel and Jakobsen, Espen R.},
  journal={Journal of Differential Equations},
  volume={262},
  number={6},
  pages={3759--3804},
  year={2017},
  publisher={Elsevier},
  doi={10.1016/j.jde.2016.11.033}
}

@article{ciaurri2018nonlocal,
  title={Nonlocal discrete diffusion equations and the fractional discrete Laplacian, regularity and applications},
  author={Ciaurri, {\'O}scar and Roncal, Luz and Stinga, Pablo Ra{\'u}l and Torrea, Jos{\'e} L. and Varona, Juan Luis},
  journal={Advances in Mathematics},
  volume={330},
  pages={688--738},
  year={2018},
  publisher={Elsevier}
}

@article{delTeso2018,
  author  = {del Teso, F{\'e}lix and Endal, J{\o}rgen and Jakobsen, Espen R.},
  title   = {Robust Numerical Methods for Nonlocal (and Local) Equations of Porous Medium Type. Part II: Schemes and Experiments},
  journal = {SIAM Journal on Numerical Analysis},
  year    = {2018},
  volume  = {56},
  pages   = {3611--3647},
  doi     = {10.1137/18m1180748}
}

@book{Stroock1984LargeDeviations,
  author    = {Stroock, Daniel W.},
  title     = {An Introduction to the Theory of Large Deviations},
  series    = {Universitext},
  publisher = {Springer-Verlag},
  address   = {New York},
  year      = {1984},
  doi       = {10.1007/978-1-4613-8514-1},
  mrnumber  = {0755154}
}

@article{Varopoulos1985HardyLittlewood,
  author  = {Varopoulos, Nicholas Th.},
  title   = {Hardy--Littlewood Theory for Semigroups},
  journal = {Journal of Functional Analysis},
  volume  = {63},
  number  = {2},
  pages   = {240--260},
  year    = {1985},
  doi     = {10.1016/0022-1236(85)90087-4},
  mrnumber = {0803094}
}

@article{delTeso2026Numerical,
  author  = {del Teso, F{\'e}lix and P{\l}ociniczak, {\L}ukasz},
  title   = {Numerical methods and regularity properties for viscosity solutions of nonlocal in space and time diffusion equations},
  journal = {IMA Journal of Numerical Analysis},
  volume  = {46},
  number  = {2},
  pages   = {795--837},
  year    = {2026},
  doi     = {10.1093/imanum/draf011}
}

@article{Stan2016FiniteInfinite,
  author  = {Stan, Diana and del Teso, F{\'e}lix and V{\'a}zquez, Juan Luis},
  title   = {Finite and infinite speed of propagation for porous medium equations with nonlocal pressure},
  journal = {Journal of Differential Equations},
  volume  = {260},
  number  = {2},
  pages   = {1154--1199},
  year    = {2016},
  doi     = {10.1016/j.jde.2015.09.023}
}

@article{delTesoJakobsen2026Convergent,
  author  = {del Teso, F{\'e}lix and Jakobsen, Espen R.},
  title   = {A Convergent Finite Difference-Quadrature Scheme for the Porous Medium Equation with Nonlocal Pressure},
  journal = {Foundations of Computational Mathematics},
  year    = {2026},
  doi     = {10.1007/s10208-026-09752-y}
}

@article{CaffarelliVazquez2011,
  author  = {Caffarelli, Luis and V{\'a}zquez, Juan Luis},
  title   = {Nonlinear Porous Medium Flow with Fractional Potential Pressure},
  journal = {Archive for Rational Mechanics and Analysis},
  volume  = {202},
  number  = {2},
  pages   = {537--565},
  year    = {2011},
  doi     = {10.1007/s00205-011-0420-4}
}

@misc{ciaurri2015fractional,
  author        = {Ciaurri, {\'O}scar and Roncal, Luz and Stinga, Pablo Ra{\'u}l and Torrea, Jos{\'e} L. and Varona, Juan Luis},
  title         = {Fractional discrete {Laplacian} versus discretized fractional {Laplacian}},
  year          = {2015},
  eprint        = {1507.04986},
  archivePrefix = {arXiv},
  primaryClass  = {math.AP},
  doi           = {10.48550/arXiv.1507.04986},
  note          = {arXiv:1507.04986; Theorem 1.1}
}

@article{CarrilloFronzoniSuli2025,
  author  = {Carrillo, Jos{\'e} A. and Fronzoni, Stefano and S{\"u}li, Endre},
  title   = {Finite element scheme for the fractional porous medium equation with fractional pressure},
  journal = {Numerische Mathematik},
  volume  = {157},
  number  = {5},
  pages   = {1537--1614},
  year    = {2025},
  doi     = {10.1007/s00211-025-01486-3}
}

@article{LinXu2007,
  author  = {Lin, Yumin and Xu, Chuanju},
  title   = {Finite difference/spectral approximations for the time-fractional diffusion equation},
  journal = {Journal of Computational Physics},
  volume  = {225},
  number  = {2},
  pages   = {1533--1552},
  year    = {2007},
  doi     = {10.1016/j.jcp.2007.02.001}
}

@article{Kopteva2019L1,
  author  = {Kopteva, Natalia},
  title   = {Error analysis of the {L1} method on graded and uniform meshes for a fractional-derivative problem in two and three dimensions},
  journal = {Mathematics of Computation},
  volume  = {88},
  number  = {319},
  pages   = {2135--2155},
  year    = {2019},
  doi     = {10.1090/mcom/3410}
}

@article{choi2021classical,
  title   = {Classical solutions for fractional porous medium flow},
  author  = {Choi, Young-Pil and Jeong, In-Jee},
  journal = {Nonlinear Analysis: Theory, Methods \& Applications},
  volume  = {210},
  pages   = {112393},
  year    = {2021},
  doi     = {10.1016/j.na.2021.112393}
}

@article{Biler2010,
  author  = {Biler, Piotr and Karch, Grzegorz and Monneau, R{\'e}gis},
  title   = {Nonlinear diffusion of dislocation density and self-similar solutions},
  journal = {Communications in Mathematical Physics},
  year    = {2010},
  volume  = {294},
  number  = {1},
  pages   = {145--168},
  doi     = {10.1007/s00220-009-0855-8}
}

@article{Horvath1959,
  author  = {Horv{\'a}th, J.},
  title   = {On some composition formulas},
  journal = {Proceedings of the American Mathematical Society},
  year    = {1959},
  volume  = {10},
  number  = {3},
  pages   = {433--437},
  doi     = {10.1090/s0002-9939-1959-0107788-4}
}

@article{Plociniczak2019,
  author  = {P{\l}ociniczak, {\L}ukasz},
  title   = {Derivation of the nonlocal pressure form of the fractional porous medium equation in the hydrological setting},
  journal = {Communications in Nonlinear Science and Numerical Simulation},
  year    = {2019},
  volume  = {76},
  pages   = {66--70},
  doi     = {10.1016/j.cnsns.2019.04.014}
}

@article{JinLazarovZhou2016L1,
  author  = {Jin, Bangti and Lazarov, Raytcho and Zhou, Zhi},
  title   = {An analysis of the {L1} scheme for the subdiffusion equation with nonsmooth data},
  journal = {IMA Journal of Numerical Analysis},
  year    = {2016},
  volume  = {36},
  number  = {1},
  pages   = {197--221},
  doi     = {10.1093/imanum/dru063}
}

@article{StynesORiordanGracia2017,
  author  = {Stynes, Martin and O'Riordan, Eugene and Gracia, Jos{\'e} Luis},
  title   = {Error Analysis of a Finite Difference Method on Graded Meshes for a Time-Fractional Diffusion Equation},
  journal = {SIAM Journal on Numerical Analysis},
  year    = {2017},
  volume  = {55},
  number  = {2},
  pages   = {1057--1079},
  doi     = {10.1137/16M1082329}
}

\end{document}